\documentclass[12pt, oneside,reqno]{amsart}
\usepackage[dvipsnames]{xcolor}
\usepackage{amsfonts,amsmath,mathabx,fullpage,amssymb,amsthm,tikz,mathtools,lmodern,paralist,mathrsfs,enumitem,tikz-cd,ytableau,tcolorbox,cases, soul, accents,thmtools}
\usepackage[normalem]{ulem}
\usetikzlibrary{calc,3d,patterns}

\usepackage{color} 
\usepackage[color]{xy}
\colorlet{cyan}[rgb]{cyan} 

\usepackage[backref=page]{hyperref}
\hypersetup{colorlinks, citecolor=red, filecolor=black, linkcolor=ForestGreen, urlcolor=blue}

\ytableausetup{smalltableaux, aligntableaux=center}

\newcommand{\pad}{\rule[-3mm]{0mm}{8mm}}

\newcommand{\excise}[1]{}

\DeclareMathOperator{\Gale}{\mathsf{gale}}
\DeclareMathOperator{\Comp}{\mathsf{Comp}}
\DeclareMathOperator{\comp}{\mathsf{comp}}
\DeclareMathOperator{\codim}{codim}
\DeclareMathOperator{\conv}{conv}
\DeclareMathOperator{\Id}{Id}

\DeclareMathOperator{\link}{link}
\DeclareMathOperator{\supp}{supp}
\DeclareMathOperator{\des}{des}
\DeclareMathOperator{\Des}{Des}

\DeclareMathOperator{\Symm}{Sym}

\DeclareMathOperator{\St}{St}
\let\Re\relax \DeclareMathOperator{\Re}{Re}
\DeclareMathOperator{\shuffle}{\mathsf{Shuffle}}
\DeclareMathOperator{\Sh}{\mathsf{Sh}}
\DeclareMathOperator{\Initial}{Init} 
\DeclareMathOperator{\initial}{ini}
\newcommand{\ini}[2]{\initial_{#1}(#2)} 
\DeclareMathOperator{\final}{fin}
\newcommand{\fin}[2]{\final_{#1}(#2)} 

\DeclareMathOperator{\anti}{\mathsf{s}}
\DeclareMathOperator{\Scr}{\mathsf{Scr}}
\DeclareMathOperator{\BC}{BC}
\DeclareMathOperator{\rBC}{\overline{BC}}

\newcommand{\sx}[1]{\langle{#1}\rangle}

\newcommand{\rev}{{\mathrm{rev}}}
\newcommand{\defterm}[1]{\textbf{\boldmath{#1}\unboldmath}}

\newcommand{\Bb}{\mathbb{B}}
\newcommand{\Cc}{\mathbb{C}}

\newcommand{\kk}{\Bbbk}
\newcommand{\Nn}{\mathbb{N}}

\newcommand{\Rr}{\mathbb{R}}
\newcommand{\Ss}{\mathbb{S}}

\newcommand{\BB}{\mathcal{B}}
\newcommand{\CC}{\mathcal{C}}
\newcommand{\EE}{\mathcal{E}}
\newcommand{\FF}{\mathcal{F}}
\newcommand{\GG}{\mathcal{G}}
\newcommand{\II}{\mathcal{I}}
\newcommand{\JJ}{\mathcal{J}}

\newcommand{\NN}{\mathcal{N}}

\newcommand{\CCC}{\mathsf{C}}

\newcommand{\EEE}{\mathsf{E}}
\newcommand{\FFF}{\mathsf{F}}
\newcommand{\GGG}{\mathsf{G}}

\newcommand{\cA}{\mathrm{A}}
\newcommand{\cB}{\mathrm{B}}
\newcommand{\cC}{\mathrm{C}}
\newcommand{\cD}{\mathrm{D}}
\newcommand{\cE}{\mathrm{E}}
\newcommand{\cN}{\mathrm{N}}
\newcommand{\cW}{\mathrm{W}}
\newcommand{\oneblock}{\cA_0}  

\newcommand{\join}{\vee}
\newcommand{\meet}{\wedge}
\newcommand{\htop}{\simeq}

\newcommand{\inj}{\hookrightarrow}

\newcommand{\bd}{\partial}
\newcommand{\Simplex}[1]{\check{#1}}
\newcommand{\ee}{\mathbf{e}}

\newcommand{\xx}{\mathbf{x}}
\newcommand{\yy}{\mathbf{y}}
\newcommand{\zz}{\mathbf{z}}

\newcommand{\Faces}{\mathfrak{F}}
\renewcommand{\aa}{\mathfrak{a}}
\newcommand{\bb}{\mathfrak{b}}

\newcommand{\pp}{\mathfrak{p}}
\newcommand{\qq}{\mathfrak{q}}

\renewcommand{\ss}{\mathfrak{s}}
\newcommand{\zzz}{\mathfrak{z}}

\newcommand{\oj}[1]{\underset{#1}{*}} 
\newcommand{\dju}{\sqcup} 
\newcommand{\0}{\emptyset}
\newcommand{\isom}{\cong}
\newcommand{\sd}{\,\triangle\,} 
\newcommand{\sm}{\setminus}
\newcommand{\ov}{\overline}
\newcommand{\Sym}{\mathfrak{S}}
\newcommand{\x}{\times}
\newcommand{\compn}{\vDash}
\newcommand{\leqgale}{\leq_g}
\newcommand{\Langle}{\langle\!\langle}
\newcommand{\Rangle}{\rangle\!\rangle}

\newcommand{\hclass}{\textbf{\textsc{h}}}  
\newcommand{\triv}{\textbf{\textsc{triv}}}
\newcommand{\osim}{\textbf{\textsc{osim}}}
\newcommand{\slsh}{\textbf{\textsc{sls}}}
\newcommand{\omat}{\textbf{\textsc{omat}}}
\newcommand{\oum}{\textbf{\textsc{oum}}}
\newcommand{\qmin}{\textbf{\textsc{qmin}}} 
\newcommand{\qc}{\textbf{\textsc{qc}}}
\newcommand{\qe}{\textbf{\textsc{qe}}}
\newcommand{\qi}{\textbf{\textsc{qi}}}
\newcommand{\shift}{\textbf{\textsc{shift}}}
\newcommand{\Color}{\textbf{\textsc{color}}}
\newcommand{\bc}{\textbf{\textsc{bc}}}
\newcommand{\pure}{\textbf{\textsc{pure}}} 
\newcommand{\uhopf}{\textbf{\textsc{pre}}} 

\newcommand{\refines}{\vartriangleright}
\newcommand{\refinedby}{\vartriangleleft}
\newcommand{\refineseq}{\trianglerighteq}
\newcommand{\refinedbyeq}{\trianglelefteq}

\newcommand{\notrefinedbyeq}{\ntrianglelefteq}

\newcommand{\bh}{\mathbf{h}}	\newcommand{\bH}{\mathbf{H}}
	
\newcommand{\bl}{\pmb{\ell}}		\newcommand{\bL}{\mathbf{L}}
\newcommand{\GP}{\mathbf{GP}}
\newcommand{\OGP}{\mathbf{OGP}}
\newcommand{\OIGP}{\mathbf{OIGP}}
\newcommand{\Mat}{\mathbf{Mat}}
\newcommand{\OMat}{\mathbf{OMat}}

\newcommand{\UHopf}{\mathbf{Pre}}
\newcommand{\SC}{\mathbf{SC}}

\numberwithin{equation}{section}

\declaretheorem[style=plain,numberwithin=section]{theorem}
\declaretheorem[style=plain,sibling=theorem]{lemma}
\declaretheorem[style=plain,sibling=theorem]{corollary}
\declaretheorem[style=plain,sibling=theorem]{conjecture}
\declaretheorem[style=plain,sibling=theorem]{proposition}

\declaretheorem[style=definition,sibling=theorem]{definition}
\declaretheorem[style=definition,qed=$\blacktriangleleft$,sibling=theorem]{example}
\declaretheorem[style=definition,sibling=theorem]{remark}

\begin{document}
\title{Hopf monoids of ordered simplicial complexes}

\author{Federico Castillo}
\address{Departamento de Matem\'aticas, Pontificia Universidad Cat\'olica de Chile, Santiago, Chile}
\email{federico.castillo@mat.uc.cl}
\thanks{FCC was supported in part by FONDECYT Grant \#1221133.}
\author{Jeremy L.\ Martin}
\thanks{JLM was supported in part by Simons Foundation Collaboration Grant \#315347.}
\address{Department of Mathematics, University of Kansas, Lawrence, USA}
\email{jlmartin@ku.edu}
\author{Jos\'e A.\ Samper}
\address{Departamento de Matem\'aticas, Pontificia Universidad Cat\'olica de Chile, Santiago, Chile}
\email{jsamper@mat.uc.cl}
\begin{abstract}
We study Hopf classes: families of pure ordered simplicial complexes that give rise to Hopf monoids under join and deletion/contraction. The prototypical Hopf class is the family of ordered matroids.  The idea of a Hopf class leads to a systematic study of simplicial complexes related to matroids, including shifted complexes and broken-circuit complexes.  We compute the Hopf antipodes in two cases: facet-initial complexes (which generalize shifted complexes) and unbounded ordered matroids.  The latter calculation uses the topological method of Aguiar and Ardila, complicated by certain auxiliary simplicial complexes that we call Scrope complexes, whose Euler characteristics control the coefficients of the antipode.  The resulting antipode formula is multiplicity-free and cancellation-free.  
\end{abstract}
\maketitle

\section{Introduction}

\subsection{Background: Matroids and combinatorial Hopf theory}
The study of Hopf algebras and related structures in combinatorics dates back to the seminal work of Rota in the 1970s, notably his work with Joni \cite{JoniRota} on coalgebras and bialgebras.  The underlying idea is simple: joining and breaking of combinatorial objects --- graphs, trees, posets, matroids, symmetric and quasisymmetric functions --- are modeled algebraically by multiplication and comultiplication.  Major works include Schmitt's study of incidence Hopf algebras \cite{Schmitt} and Aguiar, Bergeron and Sottile's theory of combinatorial Hopf algebras \cite{ABS}, as well as the monograph by Grinberg and Reiner \cite{GrinRei}.  More recently, the subject has turned in a more category-theoretic direction with the introduction of \emph{Hopf monoids} by Aguiar and Mahajan \cite{AgMa}; a survey of the subject appears in~\cite{AA}.  Broadly speaking, a Hopf algebra is generated by unlabeled combinatorial objects while a Hopf monoid is generated by labeled objects, so the latter keeps track of more information.

Here is a brief working definition of a Hopf monoid.  A \textit{vector species} is a functor $\bH$ from finite sets $I$ with bijections to vector spaces $\bH[I]$ with linear maps; one should think of $\bH[I]$ as the space spanned by the set $\bh[I]$ of combinatorial objects of a given type labeled by~$I$ (for example, graphs with vertices $I$ or matroids on ground set $I$).  A \textit{Hopf monoid in vector species} is a vector species $\bH$ equipped with linear maps $\mu_{S,T}:\bH[S]\otimes\bH[T]\to\bH[I]$ (\textit{product}) and $\Delta_{S,T}:\bH[I]\to\bH[S]\otimes\bH[T]$ (\textit{coproduct}) for all decompositions $I=S\dju T$; the maps must satisfy certain compatibility conditions.  A \textit{connected} vector species ($\dim\bH[\0]=1$) has \textit{antipode} maps $\anti_I:\bH[I]\to\bH[I]$ given by the \textit{Takeuchi formula} \cite[Defn.~1.1.11]{AA}, \cite[Prop.~8.13]{AgMa}:
\begin{equation} \label{Takeuchi}
\anti^{\bH}[I] = \sum_{\cA\in\Comp(I)} (-1)^{|\cA|}\mu_{\cA}\circ \Delta_{\cA}.
\end{equation}
The antipode can also be defined via commutative diagrams~\cite[Defn.~1.15]{AgMa}.
Every group algebra is a Hopf algebra in which the antipode is (the linearization of) inversion in the group \cite[Ex.~1.31]{GrinRei}, \cite[p.88]{AgMa}, so the antipode can be regarded as a generalization of group inversion.  The Takeuchi formula, typically exhibits massive cancellation, so one of the fundamental problems in studying a given Hopf algebra or monoid is to give a cancellation-free formula for its antipode.

The \textit{Hopf algebra of matroids} was introduced by Crapo and Schmitt \cite{crapo-schmittI,crapo-schmittII,crapo-schmittIII}.  The corresponding Hopf monoid was described by Aguiar and Mahajan \cite[\S13.8.2]{AgMa} and has attracted recent interest; see, e.g.,  \cite{ardila_sanchez,bastidas,Supina}.  There are many definitions of a matroid (see, e.g.,~\cite[Section 1]{Oxley}), but for our purposes the most convenient definition is that a matroid is a simplicial complex $\Gamma$ on vertex set $I$ such that the induced subcomplex $\Gamma|S=\{\sigma\in\Gamma:\ \sigma\subseteq S\}$ is pure for every $S\subseteq I$ (i.e., every facet of~$\Gamma|S$ has the same size).  In fact the following condition, which we call \textit{link-invariance}, is equivalent: for every $S\subseteq I$, all facets of $\Gamma|S$ have identical links.  (Recall that the \emph{link} of a face in a simplicial complex is defined as $\link_\Gamma(\varphi)=\{\sigma\in\Gamma:\ \sigma\cap\varphi=\0,\ \sigma\cup\varphi\in\Gamma\}$.)  This characterization of matroid complexes appears to be new (Theorem~\ref{thm:new-matroid}).

Link-invariance is crucial to the definition of the Hopf monoid $\Mat$ of matroids: product is given by join (which corresponds to direct sum of matroids) and coproduct by restriction/contraction.  Specifically, for matroids $\Gamma_1,\Gamma_2,\Gamma$ on ground sets $S,T,I$ one defines
\begin{equation} \label{hopfops}
\begin{aligned}
\mu_{S,T}(\Gamma_1,\Gamma_2) &= \Gamma_1*\Gamma_2 = \{\sigma_1\cup\sigma_2:\ \sigma_1\in \Gamma_1,\ \sigma_2\in \Gamma_2\},\\
\Delta_{S,T}(\Gamma) &= \Gamma|S \otimes \Gamma/S
\end{aligned}
\end{equation}
where $\Gamma/S=\link_\Gamma(\varphi)$ for any facet $\varphi\in\Gamma|S$.  Since link-invariance characterizes matroid complexes, no larger subspecies of the species $\SC$ of simplicial complexes can be made into a Hopf monoid in this way.\footnote{Benedetti, Hallam and Machacek~\cite{BHM} considered a different Hopf structure on simplicial complexes, motivated by an analogous Hopf structure for graphs, in which product is given by disjoint union rather than join.  Note that the disjoint union of matroids is in general not a matroid.}

\subsection{Ordered simplicial complexes}
One of the goals of our work is to use combinatorial Hopf theory to study pure simplicial complexes that generalize or behave similarly to matroids, or that exhibit similar behavior.  Two well-known examples of such classes include pure \textit{shifted complexes}  \cite{Kalai-shifting} and \textit{broken-circuit complexes} of matroids \cite{bry}.  However, as we have just seen, there are two major difficulties to defining a Hopf structure on any species containing $\Mat$.  Difficulty~1, as mentioned above, is that link-invariance characterizes matroids, meaning that $\Mat$ cannot be extended within $\SC$.  Difficulty~2 is that even the class of pure complexes is not closed under restriction; indeed, the largest class of pure complexes closed under restriction is precisely the set of matroid complexes.

To resolve these difficulties, we introduce a linear order~$w$ on the vertex set of~$\Gamma$.  On a technical level, we can now overcome Difficulty~1 by defining $\Gamma/S$ unambiguously as the link of the \textit{lexicographically minimal} facet of~$\Gamma|S$ with respect to $w$ (following the ideas of~\cite{Samper}).  Introducing an order on the ground set is actually quite natural in the context of matroids and related structures.  Matroids are characterized by uniform behavior in many ways with respect to all possible ground set orderings, while both shifted complexes and broken-circuit complexes (as well as the special class of matroids known as \textit{positroids}; see \cite{ARW}) are controlled strongly by linear orderings of their vertex sets.  As discussed in, e.g., \cite{HS, Samper}, it would be convenient to extend matroids to a larger class of complexes to facilitate various inductive arguments.  As one example, one would like to attack Stanley's  notorious conjecture on pure O-sequences \cite[Conjecture~III.3.6]{GreenBook} by induction on the number of bases.  Removing a basis does not in general preserve the property of being a matroid complex, but it may preserve membership in a larger class.  Another motivation for extending matroids arises in the theory of combinatorial Laplacians of simplicial complexes.  The Laplacians of matroid complexes \cite{KRS} and shifted complexes \cite{DR} are known to have integer eigenvalues and satisfy a common recurrence relation \cite{Duval}, raising the question of determining the largest class of simplicial complexes satisfying the recurrence.

Thus we are now looking for Hopf monoids in species whose standard bases are tensors $w\otimes\Gamma$, where $w$ is a linear order on the vertex set of the simplicial complex $\Gamma$.  That is, these species are subspecies of a Hadamard product of linear orders with $\SC$, with product and coproduct of simplicial complexes defined as just explained.  (For more information on Hadamard products, see \cite[\S8.13.1]{AgMa} and \cite{AgMa14}.) There are two basic Hopf monoid structures on the species of linear orders, called~$\bL$ and~$\bL^*$ in \cite[Examples~8.16 and~8.24]{AgMa}.  The product in~$\bL$ is given by concatenation, while in~$\bL^*$ it is shuffle product.

The shuffle product is much better for our purposes, for several reasons.  First and most importantly, the coproduct $\Delta_{S,T}$ in~$\bL^*$ is nonzero only on linear orders $w$ in which every element of~$S$ precedes every element of~$T$.  This resolves Difficulty~2: rather than requiring that $\Gamma$ be a matroid, we need only require that it be \defterm{prefix-pure}, that is, that every subcomplex \emph{induced by an initial segment of~$w$} is pure.  (This is in practice a mild restriction; many classes of complexes of interest, including broken-circuit complexes and pure shifted complexes, are prefix-pure.) Second, shuffle product (but not concatenation) is commutative, meaning that the inherent commutativity of join is preserved in the Hadamard product.  Third, using $\bL^*$ rather than $\bL$ turns out to be a more natural choice from the geometric viewpoint that we will describe soon.\footnote{On the other hand, positroids are closed under contractions and restrictions~\cite[Prop.~3.5]{ARW}, but only under join if the orders are concatenated rather than shuffled.}

Accordingly, we define a \defterm{Hopf class} as a class of ordered complexes that is closed under initial restriction, initial contraction, and ordered join.  All the Hopf classes we will consider in this paper consist of prefix-pure complexes.  Our main results about Hopf classes are as follows.

\begin{theorem}(Proposition~\ref{universal-Hopf} + Theorem~\ref{thm:quasiHopf})
Every Hopf class $\hclass$ gives rise to a commutative Hopf monoid $\bH$ that is a vector subspecies (though not a Hopf submonoid) of $\bL^*\x\SC$.
Moreover, the Hopf class $\uhopf$ of prefix-pure ordered complexes is the largest Hopf class all of whose members are pure complexes.
\end{theorem}

The Hopf product and coproduct on $\bH$ are given in \eqref{Hopf-class-product} and \eqref{Hopf-class-coproduct}.

\begin{theorem}(Section \ref{sec:Hopf-class-zoo})
The following collections (among others) are Hopf classes of prefix-pure ordered complexes.
\begin{enumerate}
    \item Ordered matroids;
    \item Strongly lexicographically shellable complexes, i.e., those for which every ordering on the ground set induces a shelling order on every restriction to an initial segment;\footnote{This definition of lex-shellability is not to be confused with CL- or EL-shellability of the order complex of a poset as in, e.g., \cite{Bjorner-SCM,BW}.}
    \item Broken-circuit complexes and their contractions;
    \item Pure shifted simplicial complexes and their joins;
    \item Color-shifted complexes in the sense of Babson and Novik~\cite{BN};
    \item Any complex in a quasi-matroidal class in the sense of \cite{Samper};
    \item Gale truncations of ordered matroids. 
\end{enumerate}
The known inclusions between these Hopf classes are shown in Figure~\ref{fig:Hopf-class-hierarchy}.
\end{theorem}

The condition of prefix-purity actually appears (without a name) in the work of Brylawski~\cite[p.~430]{bry} on broken-circuit complexes.  There it is observed that matroids are strictly contained in the class of (reduced) broken-circuit complexes, which in turn are strictly contained in prefix-pure complexes.

In the course of this work, we have checked computationally that there exists  simplicial complex (Lockeberg's simplicial 3-sphere) that is shellable but not lexicographically shellable under any order (Remark \ref{lex-shellable-shellable}). We believe this observation to be new.

\subsection{Polyhedra}
Having developed a purely algebraic and combinatorial theory of Hopf monoids of ordered simplicial complexes extending matroids, we now move to geometry.  Every matroid $M$ has an associated \textit{base polytope} $\pp_M$ defined as the convex hull of the characteristic vectors of its bases (and thus contains all the data of~$M$).  A far-reaching result of Gelfand, Goresky, MacPherson and Serganova~\cite[Thm.~4.1]{GGMS} states that a polyhedron $\pp\subset\Rr^E$ is a matroid base polytope for some matroid on $E$ if and only if $\pp$ satisfies the following three conditions:
\begin{enumerate}
    \item[(M1)] $\pp$ is bounded;
    \item[(M2)] $\pp$ is an 0/1-polyhedron, that is, the coordinates of all vertices consist entirely of 0/1 vectors;
    \item[(M3)] Every edge of $\pp$ is parallel to some $\ee_i-\ee_j$, where $\{\ee_i\}_{i\in E}$ is the standard basis of $\Rr^E$.  Equivalently, the normal fan of $\pp$ is a coarsening of the fan defined by the braid arrangement.
\end{enumerate}
Conditions~(M1) and~(M3) define the class of \textit{generalized permutohedra}, an important family of polytopes introduced under that name by Postnikov \cite{PRW} and equivalent to the \textit{polymatroids} studied by Edmonds~\cite{edmonds}.  (See \cite[Remark~1.3.10]{AA} for further discussion of related notions.)  Generalized permutohedra form a Hopf monoid $\GP$ that was studied intensively by Aguiar and Ardila \cite{AA}; in particular, the antipode of a generalized permutohedron $\pp$ has an elegant formula ~\cite[Thm.~1.6.1]{AA} in terms of the faces of $\pp$.  The Aguiar-Ardila formula is both \emph{cancellation-free} (all summands are distinct) and \emph{multiplicity-free} (all coefficients are $\pm1$).  

This theory carries over with little change to the family of \textit{extended generalized permutohedra}, or possibly-unbounded polyhedra satisfying condition~(M3).  Moreover, the map $M\mapsto\pp_M$ identifies $\Mat$ with a Hopf submonoid of~$\GP$.  Meanwhile, every polyhedron $\pp$ that satisfies (M2) and whose vertices have fixed coordinate sum gives rise to a pure simplicial complex $\Upsilon(\pp)$, its \textit{indicator complex}, whose facets are the supports of vertices of~$\pp$.

Accordingly, we next study Hopf monoid structures on species of more general ordered polyhedra than matroid base polytopes.  (Here ``ordered polyhedron'' means ``polyhedron whose ambient space is equipped with a linear order on its coordinates.'')  We first replace $\GP$ with $\GP^+$, the Hopf monoid of extended (i.e., possibly unbounded) generalized permutohedra; see \cite[\S1.4.5]{AA}.  We do not wish to drop~(M3) which is fundamental to the structure of matroid polytopes.  In fact, $\GP^+$ cannot be extended to any larger Hopf monoid of polyhedra with the same product and coproduct~\cite[Thm. 1.5.1]{AA},
so we really are restricted to subspecies of $\bL^*\x\GP^+$; that is, whose basis elements are tensors $w\otimes\pp$, where $\pp$ is an (extended) generalized permutohedron in $\Rr^I$ and $w$ is a linear order on~$I$.
On the other hand, it is possible to drop one or both of conditions~(M1) and~(M2) to get interesting Hopf monoids of ordered polyhedra that generalize matroid polytopes.

\begin{theorem} (Theorem~\ref{omatplus-hclass} + Theorem~\ref{thm:ogp-plus})
The following relaxations of the conditions (M1) and (M2), while retaining (M3), produce Hopf monoids:
\begin{enumerate}
\item Retaining (M2) but dropping (M1) yields the species $\OIGP^+\subseteq\bL^*\x\GP^+$, of \emph{ordered 0/1 extended generalized permutohedra}. The map sending a polyhedron to its indicator complex gives rise to a Hopf morphism $\tilde\Upsilon:\OIGP^+\to\UHopf$.  The image is a Hopf monoid $\OMat^+$ that arises from a Hopf class $\omat^+$ of complexes that we call \emph{unbounded matroids}.
\item Retaining (M1) but dropping (M2) yields the Hopf monoid $\OGP$ of \emph{ordered generalized permutohedra}, which is just the Hadamard product $\bL^*\x\GP$.
\item Dropping both (M1) and (M2) yields the Hopf monoid $\bL^*\x\GP^+$, which contains the submonoid $\OGP^+\subset\bL^*\x\GP^+$ spanned by pairs $w\otimes\pp$ such that $\pp$ is bounded in the direction defined by~$w$.  
\end{enumerate}
\end{theorem}

Two different 0/1 polyhedra can have the same indicator complex. This implies that the map $\tilde\Upsilon:\OIGP^+\to\UHopf$ mentioned above has a nontrivial kernel. 

\subsection{Antipodes}
Having constructed a large family of Hopf monoids, we now wish to compute their antipodes by simplifying the general Takeuchi formula~\eqref{Takeuchi} to a cancellation- and multiplicity-free expression.  We focus our attention on the special cases of shifted complexes and ordered generalized permutohedra.

In Section \ref{sec:qantipode}, we first focus on the class of ordered simplicial complexes that are \textit{facet-initial}, i.e., whose lex-minimal facet is an initial segment of the order.  This condition is much milder than being shifted, and in fact is enough to expand the Takeuchi formula and track most or all of the cancellation.  The results may be summarized as follows:

\begin{theorem}(Theorem~\ref{thm:antipode-FI} + Section~\ref{sec:shifted-matroids})\label{lthm:shifted}
\begin{enumerate}
\item Equation~\eqref{antipode-facet-initial} gives a simple (but not entirely cancellation-free) formula for the antipode of a facet-initial ordered complex.
\item For a shifted complex without loops or coloops, the formula~\eqref{antipode-facet-initial} is cancellation-free.
\item Equation~\eqref{antipode:ver4} gives a cancellation-free (but more complicated) antipode formula for facet-initial complexes.  In particular, every coefficient in the antipode of a facet-initial complex is~$\pm1$.
\item For a shifted complex $\Gamma$, equation~\eqref{antipode:shift:ver4} gives a slightly less complicated cancellation-free formula for the antipode.  Moreover, each term in this formula can be interpreted geometrically as a face of a matroid polytope, for any matroid containing $\Gamma$ as a subcomplex.
\end{enumerate}
\end{theorem}

In Section~\ref{sec:antipode}, we compute the antipode in the Hopf monoid $\OGP^+$ of ordered extended generalized permutohedra, and thus for its Hopf submonoid $\OMat$ of ordered matroids.  In general it is difficult to compute the antipode of a Hadamard product even if the antipodes of the factors are understood; see~\cite[\S8.13]{AgMa} for further discussion of the problem.
Our argument is inspired by the topological approach of Aguiar and Ardila: expand the Takeuchi formula in the standard basis for $\OGP^+$, interpret each coefficient as the Euler characteristic of a simplicial complex obtained by intersecting some faces of the braid arrangement with the unit sphere, then use geometric arguments (e.g., convexity) to observe that these complexes are topological balls or spheres.  (This technique applies more generally in the context of Hopf monoids relative to a hyperplane arrangement; see \cite[\S7.7]{AM17} and \cite[\S12.9]{AM20}.)  However, the interaction between $\bL^*$ and $\GP^+$ produces considerable unforeseen complications.  In particular, some of the terms in the antipode are governed by the Euler characteristics of certain auxiliary simplicial complexes that we call \textit{Scrope complexes}, generated by complements of intervals (see Section~\ref{sec:scrope}).  A simple inductive argument (Proposition~\ref{scrope-homotopy}) shows that every Scrope complex is a homotopy ball or sphere, hence has Euler characteristic 1, 0, or $-1$; however, it is not clear how to ``see'' the topological type of a Scrope complex from its list of generators.  Nevertheless, these complexes enable us to track all the cancellation in the Takeuchi formula.

\begin{theorem}(Theorem~\ref{OGP-antipode})\label{lthm:OGP}
The antipode of the Hopf monoid $\OGP^+$ is given by the multiplicity-free and cancellation-free formula~\eqref{giant-antipode-formula}.
\end{theorem}

In light of Theorems~\ref{lthm:shifted} and~\ref{lthm:OGP}, we make the following conjecture.

\begin{conjecture}(Conjecture~\ref{conj:mult-free})
The antipode for the Hopf monoid of the universal Hopf class $\uhopf$ is multiplicity-free when expressed in the standard basis (i.e., all its coefficients are $\pm1$).
\end{conjecture}

The organization of the paper is as follows.

Section~\ref{sec:background} reviews background material on simplicial complexes, matroids, generalized permutohedra and Hopf monoids.  A reader familiar with the literature (particularly~\cite{AgMa} and~\cite{AA}) may wish to skip this section and refer to it as needed. Section~\ref{sec:Hopf-class} defines the main objects of study: Hopf classes of ordered complexes.  We give numerous examples and show how Hopf classes give rise to Hopf monoids. Section~\ref{sec:ogp} defines the Hopf monoids $\OGP$ and $\OGP^+$ of ordered (extended) generalized permutohedra, which extend $\OMat$ and enable us to study it geometrically on the level of ordered matroid polytopes. 
Section \ref{sec:OIGP} describes the new class of unbounded matroids, defined as ordered simplicial complexes obtained from 0/1-generalized permutohedra that are not necessarily bounded. 
Sections~\ref{sec:qantipode} and~\ref{sec:antipode} contain the antipode formulas for facet-initial complexes and ordered generalized permutohedra, respectively.
Section~\ref{sec:special} studies antipodes of special classes of generalized permutohedra for which the Scrope complexes can be understood explicitly.  These polyhedra include hypersimplices (Proposition~\ref{hypersimplex-antipode}) and certain graphical zonotopes (Proposition~\ref{star-antipode}); notably, in both of these cases, the Scrope complexes arise from certain \emph{spider preposets}.
Section~\ref{sec:open} concludes with several open questions.

\section{Background and notation} \label{sec:background}

For the sake of brevity, we will assume the reader is familiar with the combinatorics of simplicial complexes \cite{GreenBook}, preposets \cite{PRW}, polytopes \cite{Ziegler}, hyperplane arrangements~\cite{Stanley-HA}, generalized permutohedra \cite{PRW}, and Hopf monoids \cite{AA,AgMa}.  Throughout, we will adhere to the notational conventions in the following table.
\medskip

\begin{center}
\begin{tabular}{lll}
\textbf{Object} & \textbf{Font} & \textbf{Examples}\\ \hline
Set compositions & Roman capital letters & $\cA$, $\cD$, \dots\\
Albums (families of set compositions) & Sans-serif capital letters & $\EEE$, $\FFF$, \dots\\
Fans & Calligraphic letters & $\EE$, $\FF$, \dots\\
Simplicial complexes & Capital Greek letters & $\Gamma$, $\Sigma$, \dots\\
Polytopes & Fraktur & $\pp$, $\qq$, \dots\\
Hopf monoids in set species & Boldface lower-case & $\bl$, \dots\\
Hopf monoids in vector species & Boldface upper-case & $\bL$, $\Mat$, $\GP$, \dots
\end{tabular}
\end{center}
\medskip

The symbol $<$ always denotes the natural order on $\Rr$; other partial orderings are denoted by symbols such as $\prec$ and $\refinedby$.  The symbol $\blacktriangleleft$ denotes the end of an example.

\subsection{Simplicial complexes} \label{sec:scs}

We assume familiarity with the basic combinatorics of simplicial complexes; see, e.g., \cite[\S0.3]{GreenBook}. 
For any non-void pure simplicial complex $\Gamma$ on $I$ of dimension $d$, we define the \defterm{indicator polytope} $\pp_\Gamma\subset\Rr^I$ to be the convex hull of the indicator vectors of the facets of~$\Gamma$.  
Conversely, for any $0/1$-polyhedron $\pp\subset\{x\in\Rr^I: \sum_i x_i=r\}\subset\Rr^I$, its \defterm{indicator complex} $\Upsilon(\pp)$ is defined as the pure simplicial complex generated by the supports of the vertices of $\pp$.
In general $\Upsilon(\pp_\Gamma)=\Gamma$ for all simplicial complexes~$\Gamma$, but $\pp_{\Upsilon(\pp)}=\pp$ if and only if $\pp$ is a polytope.  For later use, we remark translated poset cones can have the same apex and be different 0/1-polyhedra with identical indicator complexes.

If $\Gamma_1,\Gamma_2$ are simplicial complexes on vertex sets $I_1,I_2$ then the \defterm{join} $\Gamma_1*\Gamma_2$ is the complex on $I_1\sqcup I_2$ defined by
\begin{equation} \label{define-join}
\Gamma_1*\Gamma_2=\{\gamma_1\cup\gamma_2:\ \gamma_1\in \Gamma_1,\ \gamma_2\in \Gamma_2\}.
\end{equation}
(Here and subsequently $I_1\sqcup I_2$ means the set-theoretic disjoint union, or equivalently the coproduct in the category of sets.) At the level of polytopes we have
\begin{equation}
\pp_{\Gamma_1}\times \pp_{\Gamma_2}=\pp_{\Gamma_1*\Gamma_2}.
\end{equation}

\subsection{Set compositions, preposets, and the braid fan} \label{sec:set-comps-and-gps}

Let $S$ be a finite set.  A \defterm{set composition} $\cA=A_1|\cdots|A_k$ of $S$ is an ordered list of nonempty, pairwise-disjoint subsets $A_i$ (\defterm{blocks}) whose union is $S$.   The symbol $\Comp(S)$ denotes the family of all set compositions of $S$; we abbreviate $\Comp(n)=\Comp([n])$ and write $\cA\compn S$ to indicate that $\cA\in\Comp(S)$.  In this notation, the vertical bars are called \defterm{separators}.  We typically drop the commas and braces, e.g., writing $14|256|3$ rather than $\{1,4\}|\{2,5,6\}|\{3\}$.  The order of elements within each block is not significant.  A family of set compositions of~$S$ (i.e., a subset of $\Comp(S)$) is called an \defterm{album}.

The set $\Comp(n)$ is partially ordered by refinement: $\cA\refineseq\cB$ means that every block of $\cB$ is of the form $A_i\cup A_{i+1}\cup\cdots\cup A_{j-1}\cup A_j$.  Equivalently, $\cB$ can be written by removing zero or more separators from~$\cA$: e.g., $14|2|5|67|3\refines 14|25|367$.
The refinement ordering is ranked, with rank function $r(\cA)=|\cA|-1$, and has a unique minimal element, namely the set composition $\oneblock$ with one block.  In fact, $\Comp(n)$ is a meet-semilattice, with meet $x\preceq_{\cA\meet\cB}y$ if $x\preceq_\cA y$ or $x\preceq_\cB y$. 
Every permutation $w\in\Sym_n$ gives rise to a set composition with $n$ singleton blocks, namely $\comp(w)=w(1)\,|\,w(2)\,|\,\cdots\,|\,w(n)$.

A \defterm{preposet} $Q$ on $S$ is given by  a relation $\preceq_Q$ on~$S$ that is reflexive ($x\preceq_Qx$ for all $x\in S$) and transitive (if $x\preceq_Qy$ and $y\preceq_Qz$, then $x\preceq_Qz$).   We write $x\prec y$ if $x\preceq y$ and $y\not\preceq x$.  The notation $x\equiv_Qy$ means that both $x\preceq_Qy$ and $x\succeq_Qy$; this is evidently an equivalence relation, whose equivalence classes are called the \defterm{blocks} of $Q$.  An \defterm{antichain} in $Q$ is a subset $T\subseteq S$ such that $x\not\prec y$ for all $x,y\in T$.  (Thus an antichain may contain more than one element of a a block.)

The preposet $Q$ gives rise to a poset $Q/\!\!\equiv_Q$ on its blocks.  If this poset is a chain, i.e., if either $x\preceq_Qy$ or $x\succeq_Qy$ for every $x,y\in S$, then $Q$ is a \defterm{preorder}.  A \defterm{linear extension} of a preposet $Q$ is a preorder $R$ with the same blocks as $Q$ and such that $x\preceq_Qy$ implies $x\preceq_R y$ for all $x,y$.  A preorder $R$ contains the same information as the set composition $\cA=A_1|\cdots|A_k$, where the $A_i$ are the blocks of~$R$, and $x\preceq_Ry$ whenever $x\in A_i$, $y\in A_j$, and $i\leq j$. 

The \defterm{closure} of a preposet $Q$ is the album
\begin{equation}\label{eq:album-closure}
\CCC_Q= \{\cA\in\Comp(n):\ i\preceq_Qj ~\implies~ i\preceq_{\cA}j\}.
\end{equation}
The closure is an order ideal of $\Comp(n)$ under refinement, hence a sub-meet-semilattice.  In the case that $Q$ is a set composition, the closure is a Boolean poset.

Every set composition $\cA\compn[n]$ is in bijection\footnote{This is the reverse of the convention from that used
in~\cite[\S1.3.5]{AA}, where earlier parts of the set composition correspond to larger coefficients.  Cf.~Remark~\ref{rem:reverse}.}
 with a relatively open face $\sigma_{\cA}$ in the braid arrangement $\BB_n$ with $\dim\sigma_{\cA}=|\cA|$, namely
\[\sigma_{\cA} = \{(x_1,\dots,x_n)\in\Rr^n:\ x_i \mathrel{\substack{<\\=\\>}} x_j\text{ according as } i \mathrel{\substack{\prec_{\cA}\\=_{\cA}\\\succ_\cA}} j\},\]
In fact $\cA\refineseq \cB$ if and only if $\overline{\sigma_{\cA}}\supseteq\sigma_{\cB}$ (where the bar denotes topological closure), so the correspondence may be viewed as an isomorphism of posets.  In particular, the maximal faces $\sigma_w\in\BB_n$ correspond to permutations $w\in\Sym_n$.  For each preposet $Q$, the album $\CCC_Q$ corresponds to the closed subfan
\[\CC_Q = \{\sigma_{\cA}\in\BB_n:\ \cA\in\CCC_Q\}\]
whose maximal faces correspond to the linear extensions of $Q$.  The closed subfans of $\BB_n$ that arise in this way are precisely those whose union is convex.  In addition,
\[\CC_Q = \overline{\bigcup_\cA \sigma_{\cA}}\]
where $\cA$ ranges over all linear extensions of $Q$.

Let $\Sigma^{n-2}$ be the intersection of the unit sphere in $\Rr^n$ with the hyperplane $x_1+\cdots+x_n=0$.  Thus $\Sigma^{n-2}$ is an $(n-2)$-sphere, with a polytopal (in fact, simplicial) cell structure whose (open) faces are the intersections
\[\Simplex{\sigma}_\cA=\Sigma^{n-2}\cap \sigma_{\cA}.\]
The facets correspond to permutations in $\Sym_n$, the vertices correspond to the separators of a set composition, and the empty face corresponds to the set composition $\oneblock$.  Thus, for any closed subfan $\FF\subseteq\BB_n$, we can interpret the quantity $\sum_{\sigma\in\FF} (-1)^{|\sigma|}$ as the reduced Euler characteristic of the simplicial complex $\Simplex{\FF}=\{\Simplex{\sigma}:\ \sigma\in\FF\}$.  

For example, let $Q$ be the preposet on $\{1,2,\dots,8\}$ whose quotient poset $Q/\!\!\equiv$ is shown at left below.  The simplicial complex on the right is $\Simplex{\CC}_Q=\Sigma^{n-2}\cap\CC_Q$; the labels of its faces comprise $\CCC_Q$.  We have abbreviated the set compositions by, e.g., $\mathsf{ac|bd=167|23458}$.
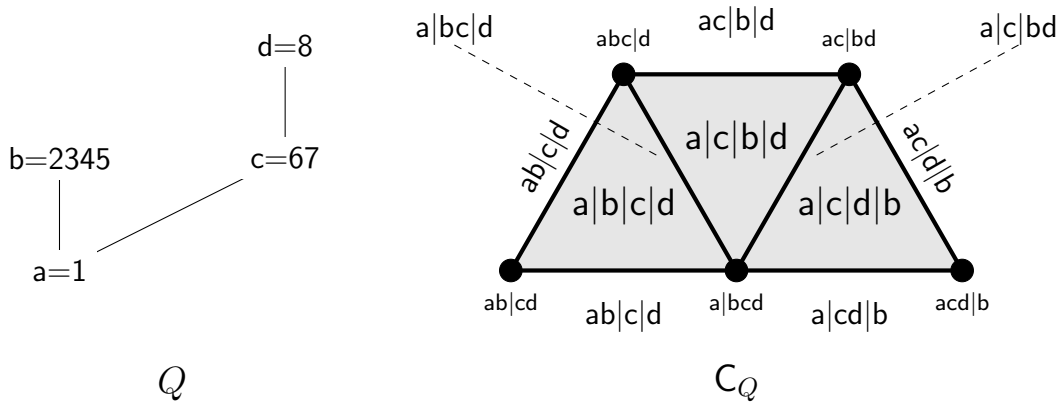
\begin{center}
\begin{tikzpicture}[scale=0.8]
\newcommand{\len}{3}
\draw (0,\len/2)--(0,0)--(\len,\len/2)--(\len,\len);
\node[fill=white] at (0,\len/2) {\sf b=2345};
\node[fill=white] at (\len,\len/2) {\sf c=67};
\node[fill=white] at (0,0) {\sf a=1};
\node[fill=white] at (\len,\len) {\sf d=8};
\begin{scope}[shift={(8,0)}]
\coordinate (P1) at (0,0);
\coordinate (P2) at (\len,0);
\coordinate (P3) at (2*\len,0);
\coordinate (P4) at (\len/2,\len*.87);
\coordinate (P5) at (3*\len/2,\len*.87);
\draw[ultra thick, fill=black!10!white] (P1)--(P3)--(P5)--(P4)--cycle;
\draw[ultra thick] (P4)--(P2)--(P5);
\foreach \co in {P1,P2,P3,P4,P5} \draw[black,fill=black] (\co) circle (.15);
\node at (\len/2,\len*.29) {\large$\mathsf{a|b|c|d}$};
\node at (\len,\len*.58) {\large$\mathsf{a|c|b|d}$};
\node at (3*\len/2,\len*.29) {\large$\mathsf{a|c|d|b}$};
\node at (\len,1.1*\len) {$\mathsf{ac|b|d}$};
\node at (0.5*\len,-.2*\len) {$\mathsf{ab|c|d}$};
\node at (1.5*\len,-.2*\len) {$\mathsf{a|cd|b}$};
\node[rotate=60] at (.15*\len,.5*\len) {$\mathsf{ab|c|d}$};
\node[rotate=-60] at (1.85*\len,.5*\len) {$\mathsf{ac|d|b}$};
\node at (-.25*\len,1.07*\len) {$\mathsf{a|bc|d}$}; \draw[dashed](-.25*\len,\len)--(2/3*\len,.5*\len);
\node at (2.25*\len,1.07*\len) {$\mathsf{a|c|bd}$}; \draw[dashed](2.25*\len,\len)--(4/3*\len,.5*\len);
\node at (0,-.15*\len) {\scriptsize$\mathsf{ab|cd}$};
\node at (\len,-.15*\len) {\scriptsize$\mathsf{a|bcd}$};
\node at (2*\len,-.15*\len) {\scriptsize$\mathsf{acd|b}$};
\node at (.5*\len,1.02*\len) {\scriptsize$\mathsf{abc|d}$};
\node at (1.5*\len,1.02*\len) {\scriptsize$\mathsf{ac|bd}$};
\end{scope}
\end{tikzpicture}
\end{center}

\subsubsection*{Natural preposets and naturalization}
Suppose that the underlying set $S$ of a preposet $Q$ is equipped with a linear (total) order $w:S\to[|S|]$.  A relation $x\preceq_Qy$ is called \defterm{$w$-unnatural}  if $w(x)>w(y)$.  We say that $Q$ is \defterm{$w$-natural} if it has no $w$-unnatural strict relations (i.e., if $w(x)>w(y)$ and $x\prec_Qy$, then in fact $x\equiv_Qy$).  Observe that a set composition is $w$-natural if and only if it coarsens the set composition $\comp(w)$; in particular, its blocks are intervals with respect to $w$.  For a preposet $Q$, there is a unique finest $w$-natural set composition $\cN_Q=\cN_{w,Q}$ such that for every $w$-unnatural relation $x\preceq_Qy$, the interval $[y,x]_w=\{z\in S:\ w(y)\leq w(z)\leq w(x)\}$ is contained in a block of $\cN_Q$.  We say that $\cN_Q$ is the \defterm{naturalization} of $Q$ with respect to~$w$.

\begin{proposition}\label{prop:intersection}
Let $Q$ be a preposet and let $w$ be a linear order on its ground set.  Then $\CCC_Q\cap\CCC_{\comp(w)}=\CCC_{\cN_{w,Q}}$.
\end{proposition}

We omit the proof, which is straightforward.

\subsection{Generalized permutohedra} \label{sec:background-for-gp}

Let $\pp\subset\Rr^n$ be a polyhedron.  For each $\xx\in\Rr^n$, let $\lambda_\xx$ be the linear functional on $\Rr^n$ given by $\lambda_\xx(\yy)=\xx\cdot\yy$, and let $\pp_\xx$ be the face of $\pp$ maximized by $\lambda_\xx$, if it exists.  The \defterm{normal cone} of a face $\qq\subset\pp$ is
\[N^\circ_\pp(\qq)=\{\xx\in\Rr^n:\ \pp_\xx=\qq\}.\]
This is a relatively open polyhedral cone of dimension $n-\dim\qq$.  The normal cones of faces comprise the \defterm{normal fan} $\NN_\pp$.
The polytope $\pp$ is a \defterm{generalized permutohedron} (GP) if $\NN_\pp$ is a coarsening of the braid fan $\BB_n$, that is, each normal cone is a union of braid faces.  Equivalently, every edge of $\pp$ is parallel to $\ee_i-\ee_j$, where $\{\ee_1,\dots,\ee_n\}$ is the standard basis of $\Rr^n$.  Every set composition $\cA\compn n$ gives rise to a face $\pp_{\cA}\subseteq\pp$ defined by
\begin{equation} \label{maximizing-face}
\pp_{\cA}=\{\xx\in\pp:\ \lambda(\xx)\geq\lambda(\yy) \ \ \forall \lambda\in\sigma_{\cA},\ \yy\in\pp\}.
\end{equation}
If $A$ is a maximal set composition (i.e., one with $n$ blocks), then the braid cone $\sigma_{\cA}$ has full dimension, hence is contained in a full-dimensional cone of $\NN_\pp$, so $\pp_{\cA}$ is a vertex of $\pp$ (and all vertices arise in this way).  Moreover, for each face $\qq\subseteq\pp$, the album
\begin{equation} \label{normal-preposet}
\{\cA\compn n:\ \sigma_{\cA}\subseteq N^\circ_\pp(\qq)\} = \{\cA\compn n:\ \pp_{\cA}=\qq\}
\end{equation}
consists precisely of the set compositions coarsening some preposet $Q$ on $[n]$, the \defterm{normal preposet of $\qq$}.  Often we will work simultaneously with a face $\qq$ and its normal preposet~$Q$, which contain equivalent information. Notice that $|Q|=n-\dim(\qq)$.  The definition of a GP implies that normal cones of faces carry combinatorial structure.  Accordingly, we define the following fans and their corresponding albums:
\[\begin{array}{r@{}l@{}l l r@{}l@{}l}
\CC_\qq^\circ	&{}= \{\sigma_{\cA}:~ \sigma_{\cA}\subseteq \NN_\pp(\qq)\}
			&{}= \{\sigma_{\cA}:~ \pp_{\cA}=\qq\}, &\quad&
\CCC_Q^\circ	&{}= \{\cA:~ \sigma_{\cA}\in\CC_\qq^\circ\}
			&{}= \{\cA:~ \pp_{\cA}=\qq\},\\
\CC_\qq		&{}= \{\sigma_{\cA}:~ \sigma_{\cA}\subseteq\ov{\NN_\pp(\qq)}\}
			&{}= \{\sigma_{\cA}:~ \pp_{\cA}\supseteq\qq\}, &&
\CCC_Q		&{}= \{\cA:~ \sigma_{\cA}\in\CC_\qq\}
			&{}= \{\cA:~ \pp_{\cA}\supseteq\qq\},\\
\partial\CC_\qq	&{}= \CC_\qq\,\sm\,\CC^\circ_\qq
			&{}= \{\sigma_{\cA}:~ \pp_{\cA}\supsetneq\qq\}, &&
\partial\CCC_Q	&{}= \CCC_Q\,\sm\,\CCC^\circ_Q
			&{}= \{\cA:~ \pp_{\cA}\supsetneq\qq\}.
\end{array}\]

In particular, $\CCC_\pp=\CCC_\pp^\circ$ is the ($n-\dim\pp$)-dimensional vector space of functions that are constant on $\pp$, and $\partial\CCC_\pp=\0$.  In the ``full-dimensional'' case $\dim\pp=n-1$, the space $\CCC_\pp$ is just the line $\sigma_{\oneblock}$.

The definition of $\CCC_Q$ is consistent with~\eqref{eq:album-closure}.  The following lemma gives the geometry-combinatorics dictionary explicitly.  For $A\in\CCC_Q$, we say that $\cA$ \defterm{collapses a relation of $Q$} if some block of $\cA$ contains elements $x,y$ for which $x\prec_Qy$.

\begin{lemma}\label{lem:comb-dictionary}
Let $Q$ be a preposet.  Then
\begin{align*}
\partial\CCC_Q &=\{\cA\in\CCC_Q:\ \text{ $\cA$ collapses some relation of $Q$}\},\\
\CCC_Q^\circ &= \{\cA\in\CCC_Q:\ \text{ $\cA$ collapses no relation of $Q$}\}.
\end{align*}
\end{lemma}

\begin{proof}
The two claims are equivalent by~\eqref{eq:album-closure}, and the description of $\bd\CCC_Q$ is equivalent to~\cite[Prop.~3.5(2)]{PRW}.
\end{proof}

\begin{example}[\bf A cone]\label{ex:cone}
Consider the polyhedron
\begin{equation}\label{def:cone}
\pp = \left\{\xx=(x_1,x_2,x_3,x_4)\in \Rr^4:\ x_3\leq 3,\quad x_4\leq 4,\quad x_2+x_4\leq 6,\quad x_1+x_2+x_3+x_4=10\right\},
\end{equation}
which is a three-dimensional simplicial cone with vertex $(1,2,3,4)$ and rays in directions $\ee_1-\ee_3$, $\ee_1-\ee_2$, and $\ee_2-\ee_4$.
The normal cone of the vertex is the cone given by the inequalities $x_1\leq x_2\leq x_4$ and $x_1\leq x_3$. This cone is subdivided by three braid cones as shown in Figure \ref{fig:cone}. For instance, $1|24|3$ is in the boundary (since $2,4$ are comparable) while $1|2|34$ is in the interior, (since 3,4 are incomparable).

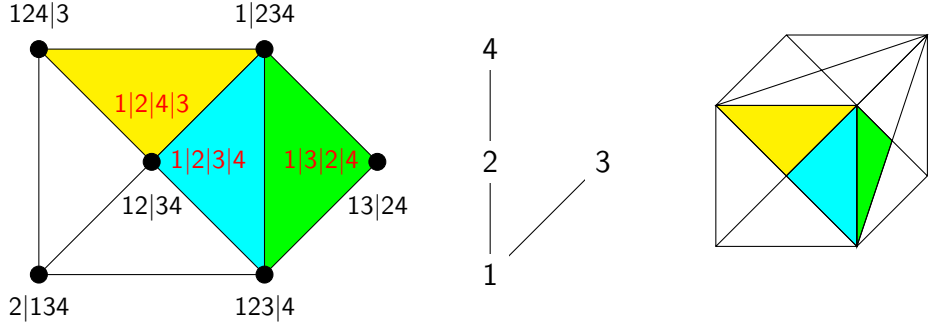
\begin{figure}[ht]
\begin{center}
\begin{tikzpicture}[scale=1.5]
   \draw [fill=yellow] (0,2) -- (2,2) -- (1,1) -- cycle;   \node[red] at (1,1.5) {\footnotesize$\mathsf{1|2|4|3}$};
   \draw [fill=cyan] (2,2) -- (2,0) -- (1,1) -- cycle;	   \node[red] at (1.5,1) {\footnotesize$\mathsf{1|2|3|4}$};
   \draw [fill=green] (2,2) -- (2,0) -- (3,1) -- cycle;   \node[red] at (2.5,1) {\footnotesize$\mathsf{1|3|2|4}$};
   \draw (0,2)--(0,0)--(2,0) (0,0)--(1,1);
   \foreach \x/\y in {0/0, 0/2, 2/0, 2/2, 1/1, 3/1} \draw[fill=black] (\x,\y) circle (.075);
   \node at (2,-.3) {\footnotesize$\mathsf{123|4}$};
   \node at (2,2.3) {\footnotesize$\mathsf{1|234}$};
   \node at (0,-.3) {\footnotesize$\mathsf{2|134}$};
   \node at (0,2.3) {\footnotesize$\mathsf{124|3}$};
   \node at (1,.6) {\footnotesize$\mathsf{12|34}$};
   \node at (3,.6) {\footnotesize$\mathsf{13|24}$};
\begin{scope}[shift={(6,.25)}]
   \newcommand{\ra}{1.25}
   \draw [fill=yellow] (0,\ra) -- (\ra,\ra) -- (\ra/2,\ra/2) -- cycle;
   \draw [fill=cyan] (\ra,0) -- (\ra,\ra) -- (\ra/2,\ra/2) -- cycle;
   \draw [fill=green] (\ra,0) -- (\ra,\ra) -- (1.25*\ra,.75*\ra) -- cycle;
   \draw (0,0) -- (0,\ra) -- (\ra,\ra) -- (\ra,0) -- cycle (0,0) -- (\ra,\ra) (0,\ra) -- (\ra,0);
   \draw (\ra/2, 1.5*\ra) -- (0,\ra) -- (\ra,\ra) -- (1.5*\ra,1.5*\ra) -- cycle (\ra/2,1.5*\ra) -- (\ra,\ra) (0,\ra) -- (1.5*\ra,1.5*\ra);
   \draw (1.5*\ra,\ra/2) -- (\ra,0) -- (\ra,\ra) -- (1.5*\ra,1.5*\ra) -- cycle (1.5*\ra,\ra/2) -- (\ra,\ra) (\ra,0) -- (1.5*\ra,1.5*\ra);
\end{scope}
\begin{scope}[shift={(4,0)}]
  \draw (0,0)--(0,2)  (0,0)--(1,1);
  \node[fill=white] at (0,2) {\sf4};	\node[fill=white] at (1,1) {\sf3};
  \node[fill=white] at (0,1) {\sf2};
  \node[fill=white] at (0,0) {\sf1};
\end{scope}
\end{tikzpicture}
\end{center}
\caption{The normal fan of the cone in Example \ref{ex:cone}.}
\label{fig:cone}
\end{figure}
\end{example}

Note that $\Simplex{\CC}_\qq$ and $\partial\Simplex{\CC}_\qq$ are (closed) simplicial subcomplexes of $\Sigma^{n-2}$, homeomorphic to $\Bb^{n-\dim \qq-2}$ and $\Ss^{n-\dim \qq-3}$ respectively.

\begin{remark}\label{rem:origint}
For every proper face $\qq\subset\pp$, the origin is a vertex of $\CC_\qq$, hence of $\partial\CC_\qq$.  Therefore, $\oneblock\in\CCC^\circ_\qq$ if and only if $\qq=\pp$.  More generally, $\CCC^\circ_\pp$ is the space of linear functionals that are constant on $\pp$; this is a vector space of dimension $n-\dim\pp$. Also, if $\qq_1,\qq_2$ are faces of $\pp$, then $\ov{\qq_1}\subseteq\ov{\qq_2}$ if and only if $\CC_{\qq_1}\supseteq\CC_{\qq_2}$.
\end{remark}

An \defterm{extended generalized permutohedron} (EGP) \cite[Defn.~1.3.8]{AA} is a polyhedron whose normal fan coarsens some \emph{convex subfan} of the braid fan $\BB_n\subset\Rr^I=\Rr^n$.  Equivalently, it is a polyhedron (not necessarily bounded) such that the affine span of every face is a translate of a subspace spanned by vectors of the form $\ee_i-\ee_j$ \cite[Thm.~3.1.6 and Remark 3.1.7]{AA}.
Many of the above statements about generalized permutohedra can be carried over to this more general setting \textit{mutatis mutandis}.  For example, EGPs are preserved by products, and the face of $\pp$ corresponding to a set composition (i.e., maximized by the linear functionals in some face of $\BB$) is an EGP whenever it is well-defined (i.e., whenever any, hence all, such functionals are bounded from above on~$\pp$).

\subsection{Hopf monoids}\label{sec:exHopf}
For definitions concerning Hopf monoids, we refer the reader to~\cite{AA} and~\cite{AgMa}.  All of our Hopf monoids are connected.  We recall here the Hopf monoid structures on linear orders, matroids, and generalized permutohedra.

\medskip\noindent
\textbf{$\bL$: Linear orders with concatenation product.}
For any finite set $I$, let $\bl[I]$ denote the set of linear orders on $I$, i.e., all bijections $w:[n]\to I$, where $n=|I|$. We represent $w$ by a bracketed list $[w(1),\dots,w(n)]$ or (when no confusion can arise) a string $w(1)\cdots w(n)$.  As a set species, $\bl$ has a Hopf monoid structure in which product is concatenation and the coproduct is $\Delta^\bL_{I,J}(w)=w|_I\otimes w|_J$, where $w|_I$ and $w|_J$ are the orders induced by~$w$ on $I,J$ respectively.
In particular, $\bl$ is cocommutative but not commutative.  The antipode on the linearization $\bL$ is given~\cite[p.250]{AgMa} by $\anti^{\bL}(w)=(-1)^{|I|}w^\rev$, where $(w(1),\dots,w(n))^\rev=(w(n),\dots,w(1))$.

\medskip\noindent
\textbf{$\bL^*$: Linear orders with shuffle product.}
More important for our purposes is $\bL^*$, the dual Hopf monoid of $\bL$.  For the general theory of duality on Hopf monoids, see \cite[\S8.6]{AgMa}; here we give a self-contained description of $\bL^*$.  As a vector species, $\bL^*[I]$ is again the $\kk$-vector space spanned by all linear orders of $I$.  To define the product and coproduct on $\bL^*$, we first need to introduce the notion of a \textit{shuffle}.

Let $w^{(1)},\dots,w^{(q)}$ be linear orders on pairwise-disjoint sets $I_1,\dots,I_q$.
A \defterm{shuffle} of the $w^{(i)}$ is an ordering on $I=I_1\cup\cdots\cup I_q$ that restricts to $w_j$ on each $I_j$.  The set of all shuffles is denoted $\shuffle(w^{(1)},\dots,w^{(q)})$.
For example, $\shuffle(12,3) = \{123,132,312\}$ and $\shuffle(12,34) = \{1234,1324,1342,3124,3142,3412\}$.
The shuffle operation is commutative and associative, and $|\shuffle(w^{(1)},\dots,w^{(q)})| = \binom{|I|}{|I_1|,\,\dots,\,|I_q|}$.  The product on $\bL^*$ is given by
\begin{subequations}
\begin{equation} \label{eq:product-dual-L}
\mu_{I,J}(w,u)=\sum_{v\in\shuffle(w,u)}v.
\end{equation}
Second, let $w=(w(1),\dots,w(n))\in\bl[I]$.  An \defterm{initial segment} of $w$ is a set of the form $\ini{k}{w}=\{w(1),\dots,w(k)\}$ for some $k\in[0,n]$; the complement of an initial segment is a \defterm{final segment}.  The set of all initial segments of $w$ is denoted $\Initial(w)$.  With this in hand, the coproduct on $\bL^*$ is defined by
\begin{equation} \label{eq:coproduct-dual-L}
\Delta_{I,J}(w) = \begin{cases}
w|_I\otimes w|_J & \text{ if } I \in \Initial(w),\\
0 & \text { otherwise.}
\end{cases}
\end{equation}
\end{subequations}
Thus $\Delta_{I,J}(w)$ is nonzero if and only if all elements of $I$ precede all elements of $J$ in $w$.
More generally, if $\cA=A_1|\cdots|A_k$, then $\Delta_{\cA}(w)$ is nonzero if and only if $w$ consists of a shuffle of $A_1$, followed by a shuffle of $A_2$, etc., and in this case $\Delta_{\cA}(w) = w|_{A_1}\otimes\cdots\otimes w|_{A_k}$.

Note that $\bL^*$ is commutative but not cocommutative (in general, Hopf duality interchanges the two properties), and $\bL^*$ is not linearized (unlike $\bL$).  The antipode in $\bL^*$ is the same as that in $\bL$ (a general property of duality for any commutative or cocommutative Hopf monoid).

\emph{Henceforth, product, coproduct, and antipode on linear orders will always be taken to mean the operations of $\bL^*$ rather than $\bL$.}

For later use, we calculate the map $\mu_{\cA}\circ\Delta_{\cA}$ for any set composition $\cA\compn I$.  Say that two linear orders $u,w\in\bl[I]$ are \defterm{$\cA$-consistent}, written $u\approx_{\cA} w$, if all pairs of elements in the same block of~$\cA$ appear in the same order in $w$ and $u$; that is, if $i\equiv_{\cA} j$ then $u(i)<u(j)$ if and only if $w(i)<w(j)$.  For example, if $\cA=13|2$, then $\{312,321,231\}$ is an equivalence class under $\approx_{\cA}$.  Then
\begin{equation} \label{muDeltaW}
\mu_{\cA}(\Delta_{\cA}(w)) = \begin{cases}
\sum_{u\in\bl[I]:\ u\approx_{\cA} w} u & \text{ if } \cA\refinedbyeq\comp(w),\\
0 & \text{ otherwise.}
\end{cases}
\end{equation}

The following definition will also be useful.  Recall \cite[\S1.4]{EC1} that $i\in[n-1]$ is a \defterm{(right) descent} of a permutation $v\in\Sym_n$ if $v(i)>v(i+1)$.  The set of descents of $v$ is denoted by $\Des(v)$, and the number of descents is $\des(v)$.

\begin{definition} \label{defn:descent-comp}
Let $w,u$ be linear orders on $I$.  The \defterm{$u$-descent set composition of $w$} is the coarsening $\cD(w,u)$ of $\comp(w)=w(1)|\cdots|w(n)$ whose separators correspond to descents of the permutation $u^{-1}w$; equivalently, $w(i)\equiv w(i+1)$ if and only if $i\notin\Des(u^{-1}w)$.
\end{definition}

For example, suppose $I=\{\mathsf{a,b,c,d,e,f,g,h}\}$ and let $w=\mathsf{aebfcdhg}$, $u=\mathsf{bdahfgce}\in\bl[I]$, so that $u^{-1}w=38\cdot157\cdot246$ (with the descents marked by dots).  Then $\cD(w,u)=\mathsf{ae|bfc|dhg=ae|bcf|dgh}$.

Descent set compositions have the following basic properties:
\begin{align}
&\cD(w,u)=\oneblock \iff u=w;\\
&\forall \cA\refinedbyeq\comp(w): \quad u\approx_{\cA} w \iff \cD(w,u)\refinedbyeq \cA;\label{shuffle-des}\\
&\dim\sigma_{\cD(w,u)} = |\cD(w,u)| = \des(u^{-1}w)+1. \label{dimD}
\end{align}

In light of~\eqref{shuffle-des}, we can usefully rewrite~\eqref{muDeltaW} (when $\cA=A_1|\cdots|A_k$ is $w$-natural) as
\begin{equation} \label{muDeltaW:2}
\mu_{\cA}(\Delta_{\cA}(w)) = \sum_{u\in\shuffle(A_1,\dots,A_k)} u
= \sum_{u\in\bl[I]:\ \cA\in\EEE_{w,u}} u.
\end{equation}
where
\begin{equation} \label{define-EEE}
\EEE_{w,u}=\{\cA\compn[n]: \cD(w,u)\refinedbyeq \cA\refinedbyeq\comp(w)\}.
\end{equation}

\medskip\noindent
\textbf{$\Mat$: Matroids.}
Let $\Mat[I]$ be the $\kk$-vector space spanned by all matroids with ground set $I$. It has a Hopf monoid structure with product $\mu(M_1 \otimes M_2) = M_1*M_2$ (join) and coproduct $\Delta_{I,J}(M) = M|I \otimes M/I$, where $M|I$ is the restriction to $I$ and $M/I$ is the contraction of $I$ (so $\Mat$ is commutative but not cocommutative).  Moreover, $M/I$ is described simplicially as the link of any facet of $M|I$; the choice of facet does not matter.  In fact, as we now prove, this property characterizes matroids.  The proof is not difficult, but to the best of our knowledge this characterization of matroids has not previously appeared in the literature.

\begin{theorem} \label{thm:new-matroid}
Let $\Gamma$ be a simplicial complex on ground set $E$.  Then $\Gamma$ is a matroid complex if and only if it has the property of \defterm{link invariance}: for every $X\subseteq E$ and every facets $\sigma,\tau\in\Gamma|X$ we have $\link_\Gamma(\sigma)=\link_\Gamma(\tau)$.
\end{theorem}

\begin{proof}
($\implies$) First, note that $\link_\Gamma(\sigma)$ and $\link_\Gamma(\tau)$ are both simplicial complexes on $Y=E\sm X$, because $\sigma,\tau$ are facets (not just arbitrary faces) of $\Gamma|X$.  Moreover, they are pure, since links in pure complexes are pure.  By symmetry between $\sigma$ and $\tau$, it is enough to show that every facet of $\link_\Gamma(\sigma)$ is a face of $\link_\Gamma(\tau)$.  Accordingly, let $\alpha$ and $\beta$ be facets of $\link_\Gamma(\sigma)$ and $\link_\Gamma(\tau)$ respectively, so that $\sigma\cup\alpha$ and $\tau\cup\beta$ are facets of $\Gamma$. We will show that in fact $\alpha\in \link_\Gamma(\tau)$.  If $\alpha=\beta$ then there is nothing to prove; otherwise, we induct on the size of the symmetric difference $|\beta\sd\alpha|$.  Let $v\in\beta\sm\alpha$; by basis exchange there exists $w\in(\sigma\cup\alpha)\sm(\tau\cup\beta) = (\sigma\sm\tau)\cup(\alpha\sm\beta)$ such that $(\tau\cup\beta)-v+w = \tau\cup(\beta-v)+w$ is a facet of $\Gamma$.  In particular $\tau+w$ is a face, so it cannot be the case that $w\in\sigma\sm\tau$ (otherwise $\tau$ would not be a facet of $\Gamma|A$).  Therefore $w\in\alpha\sm\beta$ and the new facet is $\tau\cup\beta'$, where $\beta'=\beta-v+w$.  Thus $|\beta'\sd\alpha|=|\beta\sm\alpha|-1$ and the result follows by induction.

($\impliedby$) Suppose that $\Gamma$ is not a matroid complex; then there is some $X\subseteq E$ such that $\Gamma|X$ is not pure.  Let $\sigma,\tau$ be facets of $\Gamma|X$ of different cardinalities, say $|\sigma|<|\tau|$.  We may assume WLOG that $X=\sigma\cup\tau$.  Let $d=\dim\sigma$ and let $\tau=\{v_1,\dots,v_k\}$.  For $0\leq i\leq k$, define $X_i=\sigma\cup\{v_1,\dots,v_i\}$ and $\Gamma_i=\Gamma|X_i$.  Then
\[\langle\sigma\rangle = \Gamma_0 \subsetneq \Gamma_1 \subsetneq \cdots \subsetneq \Gamma_k=\Gamma|X\]
and $\dim\Gamma_0=\dim\Gamma_1=d<\dim\Gamma_k$.  Let $j$ be the smallest index such that $\dim\Gamma_j>d$ (necessarily, $\dim\Gamma_j=d+1$) and let $\varphi$ be a facet of $\Gamma_j$ such that $\dim\varphi=d+1$.  Then $\varphi$ must contain $v_j$, so $\varphi'=\varphi-v_j$ is a face of $\Gamma_{j-1}$, hence a facet (since $\dim\Gamma_{j-1}=d=\dim\varphi'$).  On the other hand, $\sigma$ is also a facet of $\Gamma_{j-1}$ (since it is a facet of $\Gamma_k$), and $x_j$ belongs to $\link_\Delta(\varphi')$ but not to $\link_\Delta(\sigma)$.
\end{proof}

\begin{remark} \label{darij}
Another proof of the ($\impliedby$) direction was pointed out to the authors by Darij Grinberg.  Let $\sigma,\tau$ be facets of $\Gamma|X$.  First observe that $|\sigma|=|\tau|$, since matroids are pure.  Let $\alpha\in\link_\Gamma(\sigma)$, so $\sigma\cup\alpha\in\Gamma$.  Then repeatedly applying condition~(3) of the definition of matroids gives $\tau\cup\alpha\in\Gamma$ (since $\sigma$ itself cannot donate any vertices to $\tau$), so $\alpha\in\link_\Gamma(\tau)$ and therefore $\link_\Gamma(\sigma)\subseteq\link_\Gamma(\sigma)$, and the argument is symmetric in $\sigma,\tau$.
\end{remark}

\begin{corollary} \label{Mat-universal}
$\Mat$ is the largest subspecies of $\SC$ that admits a Hopf monoid structure with the operations of join and restriction/link.
\end{corollary}

\medskip\noindent
\textbf{$\GP$ and $\GP^+$: Generalized permutohedra.}
Let $\GP[I]$ be the $\kk$-vector space spanned by all generalized permutohedra in $\Rr^I$.  Aguiar and Ardila \cite{AA} studied the Hopf monoid structure given by $\mu(\pp_1 \otimes \pp_2) = \pp_1\times \pp_2$ and $\Delta_{I,J}(\pp) = \pp|I \otimes \pp/I$ (for the definitions of the latter, see \cite[Prop.~1.4.2]{AA}. In particular $\mu_{\cA}(\Delta_{\cA}(\pp)) = \pp_{\cA}$ for any $\cA\compn I$ and $\pp\in\GP[I]$.
The antipode in $\GP$ was computed by Aguiar and Ardila~\cite[Thm.~1.6.1]{AA} using topological methods:
\begin{equation}\label{eq:antipodegp}
\anti^{\GP}(\pp) = (-1)^{|I|} \sum_{\qq\leq \pp} (-1)^{\codim\qq} \qq.
\end{equation}

The Hopf monoid $\GP^+$ is defined by setting $\GP^+[I]$ to be the $\kk$-vector space spanned by all extended generalized permutohedra in $\Rr^I$.  Multiplication is still Cartesian product, while comultiplication~\cite[\S1.4.5]{AA} is
\begin{equation}\label{eq:combination}
\Delta_{\cA}(\pp) = \begin{cases}
\pp_{\cA} & \text{ if the linear functional $\mathbf{1}_I$ is bounded from above on $\pp$},\\
0 & \text{ otherwise.} \end{cases}
\end{equation}

The map $\Mat\to\GP$ sending $M$ to its base polytope $\pp_M$ is an injective morphism of Hopf monoids~\cite[Thm.~3.1.3, Thm.~3.1.4, Prop.~3.3.3]{AA}.
The antipode in $\Mat$ is best understood via the embedding $\Mat\to\GP$; see \cite[Thm.~3.3.4]{AA}.

\section{Hopf classes of ordered complexes}\label{sec:Hopf-class}

\subsection{Definitions and basic properties}\label{sec:Hopf-class-basic}
An \defterm{ordered complex} is a triple $(w,\Gamma,I)$ where $\Gamma$ is a  simplicial complex on finite vertex set $I$, and $w$ is a linear order on $I$.  If the ground set is clear from context, we may write simply $(w,\Gamma)$.

For an initial segment $A$ of~$w$, the \defterm{(initial) restriction} and \defterm{(initial) contraction} of $(w,\Gamma)$  with respect to $A$ are defined as $(w,\Gamma)|A=(w|_A,\Gamma|A,A)$ and 
$(w,\Gamma)/A=(w|_{I\sm A}, \Gamma/ A,I\sm A)$, where $\Gamma/A=\link_{\Gamma}(\varphi)$, where $\varphi$ is the facet of $\Gamma|A$ that is lex-minimal with respect to $w$.

Note that restricting to the entire ground set (as an initial segment), or contracting the empty initial segment, leaves $(w,\Gamma)$ unchanged, while restricting to the empty set or contracting the empty initial segment produces the trivial ordered complex $([],\{\0\},\0)$, where $[]$ denotes the (trivial) ordering of the empty set and $\{\0\}$ is the trivial simplicial complex (not the void complex!).

Initial restriction and contraction behave well when iterated (analogously to deletion and contraction for matroids; see~\cite[Prop.~3.1.26]{Oxley}), in the following sense.

\begin{lemma}\label{lem:coprod}
Let $(w,\Gamma)$ be an ordered complex, and suppose that~$I$ and $I\dju J$ are initial segments of~$w$. Then the following restriction/contraction relations hold: 
\begin{enumerate}
    \item $(w,\Gamma)|I = \big((w,\Gamma)|(I\dju J)\big)|I$.
    \item $\big((w,\Gamma)/I\big)|J = \big((w,\Gamma)|(I\dju J)\big)/I$.
    \item $\big((w,\Gamma)/I\big)/J = (w,\Gamma)/(I\dju J)$.
\end{enumerate}
\end{lemma}

\begin{proof}
Assertion~(1) is straightforward.  For assertion~(2), both simplicial complexes equal $\link_{\Gamma|_J}(\varphi)$, where $\varphi$ is the lex-minimal facet of $\Gamma|I$.  Finally, assertion~(3) follows from the observation that the lex-minimal facet $\varphi$ of $\Gamma|(I\dju J)$ can be decomposed as $\varphi_I\dju \varphi_J$, where $\varphi_I$ is the lex-minimal facet of $\Gamma|I$ and $\varphi_J$ is the lex-minimal facet of $\link_{\Gamma} \varphi_I$.
\end{proof}

In the ordered setting, we need to specify in addition an ordering on the ground set of the join; any shuffle of the orderings of the join factors will do.  Accordingly, for any $w\in\shuffle(w_1,w_2)$, we define the \defterm{ordered join} operation $\oj{w}$ by
\begin{equation} \label{define-ordered-join}
(w_1,\Gamma_1,I_1)\oj{w}(w_2,\Gamma_2,I_2)=(w,\Gamma_1*\Gamma_2,I_1\sqcup I_2).
\end{equation}
The trivial ordered complex $([],\{\0\})$ is a two-sided identity for ordered join.  Moreover, ordered join is compatible with restriction and contraction in the following sense.

\begin{lemma}\label{lem:joins}
Let $(w_1,\Gamma_1,I_1)$ and $(w_2,\Gamma_2,I_2)$ be  ordered complexes, let $A_1\in\Initial(w_1)$ and $A_2\in\Initial(w_2)$, and let $w\in \shuffle(w_1, w_2)$ such that $A_1\dju A_2\in\Initial(w)$. Let $\hat{w}$ and $\check{w}$ be the restrictions of $w$ to $A_1\dju A_2$ and $(I_1\sm A_1)\dju(I_2\sm A_2)$ respectively. Then
\begin{enumerate}
	\item $\big((w_1,\Gamma_1)\oj{w} (w_2,\Gamma_2)\big)|(A_1\dju A_2) = (w_1,\Gamma_1)|A_1\oj{\hat{w}} (w_2,\Gamma_2)|A_2$; and
	\item $\big((w_1,\Gamma_1)\oj{w} (w_2,\Gamma_2)\big)/(A_1\dju A_2) = (w_1,\Gamma_1)/ A_1 \oj{\check{w}}(w_2,\Gamma_2)/ A_2$.
\end{enumerate}
\end{lemma}
We omit the proof, which is a routine calculation.

\begin{definition} \label{defn:Hopf-class}
A class $\hclass$ of ordered complexes is called a \defterm{Hopf class} if it satisfies the following three conditions.
\begin{enumerate}
	\item (\textbf{Closure under ordered join}) If $(w_1,\Gamma_1),(w_2,\Gamma_2)\in\hclass$, then $(w_1,\Gamma_1)\oj{w}(w_2,\Gamma_2)\in\hclass$ for every $w\in\shuffle(w_1,w_2)$.
	\item (\textbf{Closure under initial restriction}) If $(w,\Gamma)\in\hclass$ and $A$ is an initial segment of~$w$, then $(w,\Gamma)|A\in \hclass$. 
	\item (\textbf{Closure under initial contraction}) If $(w,\Gamma)\in\hclass$ and $A$ is an initial segment of~$w$, then $(w,\Gamma)/A\in\hclass$.
\end{enumerate}
\end{definition}

We will show (Theorem~\ref{thm:quasiHopf}) that every Hopf class gives rise to a Hopf monoid: ordered join provides a product, and initial restriction and contraction provide a coproduct.

Pure simplicial complexes do \textit{not} form a Hopf class, since they are not closed under restriction (although they are closed under contraction and join). We say that a pure ordered complex is \defterm{prefix-pure} if all its initial restrictions (hence all its initial contractions) are pure.  Evidently, if every complex in a Hopf class $\hclass$ is pure, then in fact every complex in $\hclass$ is prefix-pure.  (In fact this condition appears in Brylawski's fundamental paper~\cite{bry} on broken-circuit complexes, although it does not play a major role there.)

\begin{proposition} \label{universal-Hopf}
The class $\uhopf$ of all prefix-pure complexes is a Hopf class, hence the unique largest Hopf class whose members are all pure complexes.
\end{proposition}

\begin{proof}
The definition of prefix-purity implies that $\uhopf$ is closed under initial restriction, and it is closed under initial contraction because every link in a pure complex is pure.  Moreover, if $(w_1,\Gamma_1)$ and $(w_2,\Gamma_2)$ are prefix-pure, then by Lemma \ref{lem:joins} every initial restriction or contraction of an ordered join $(w_1,\Gamma_1)\oj{w}(w_2,\Gamma_2)$ is a join of initial restrictions or contractions of the two of them, hence is pure.  Thus $\uhopf$ is a Hopf class.
\end{proof}

Henceforth, we will only consider Hopf classes of prefix-pure complexes.

The intersection of Hopf classes is again a Hopf class.  Therefore, every sub-collection $\hclass\subseteq\uhopf$ has a well-defined \defterm{Hopf closure} $\bar\hclass$, namely the intersection of all Hopf classes containing~$\hclass$.  We may also speak of the Hopf class \defterm{generated} by a collection of prefix-pure ordered complexes.

The unique smallest Hopf class, and in fact the only finite Hopf class, is the singleton class $\triv$ containing only the trivial ordered complex.  Additional elementary examples are the classes $\omat$ of ordered matroid independence complexes, its subclass $\oum$ of ordered independence complexes of uniform matroids (equivalently, skeletons of simplices), and the smaller subclass $\osim$ of all ordered simplices.

In the unordered setting, the largest class of pure simplicial complexes that is closed under join, restriction, and deletion is precisely the class of matroid complexes.  Thus any Hopf class between $\omat$ and $\uhopf$ can be regarded as an extension of matroids in the ordered setting.  There are many Hopf classes that include non-matroidal complexes.

\subsection{A zoo of Hopf classes} \label{sec:Hopf-class-zoo}

\begin{example}[\bf Strongly lex-shellable complexes]
A pure simplicial complex $\Gamma$ is \defterm{shellable} if its facets can be ordered $\varphi_1,\dots,\varphi_n$ such that whenever $j<i$, there is an index $k<i$ and a vertex $x\in \varphi_i$ such that
\begin{equation} \label{shell}
\varphi_j\cap \varphi_i \subseteq \varphi_k\cap \varphi_i=\varphi_i\sm\{x\}.
\end{equation}
Such an order is called a \defterm{shelling order}.  There are equivalent definitions of shellability, but this is the most convenient for our purposes; see, e.g., \cite[\S7.2]{bjorner}.

We say that a pure ordered complex $(w,\Gamma)$ is \defterm{lex-shellable} if the lexicographic order $<_w$ on facets of $\Gamma$ induced by $w$ is a shelling order.  We define $(w,\Gamma)$ to be \defterm{strongly lex-shellable} if it is prefix-pure and every restriction to an initial segment is lex-shellable.  Strong lex-shellability is more restrictive than lex-shellability: for example, the graph with edges 12, 14, 34 and the natural ordering on vertices is lex-shellable but not strongly lex-shellable.
On the other hand, Hopf classes are closed under restriction to initial segments, so if every element of a Hopf class $\hclass$ is lex-shellable then in fact every element is strongly lex-shellable.

\begin{remark} \label{lex-shellable-shellable}
Lex-shellability in this sense is a stronger condition than shellability.  We carried out a brute-force computation using SageMath~\cite{Sage} to check that the boundary of Lockeberg's simplicial 4-polytope~\cite{Lockeberg}, with 12 vertices and 48 facets, is not lex-shellable.  We do not have a computer-free proof of this observation, nor do we have any reason to believe that this example is minimal.
\end{remark}

The class $\slsh$ of all strongly lex-shellable complexes is a Hopf class, for the following reasons.  First, it is closed under restriction by definition, and it is closed under contraction because shelling orders on $\Gamma$ restrict to shelling orders on all its links, and restricting a lex-shelling to a final segment produces a lex-shelling.  It remains to check closure under ordered join.

Let $(w_1,\Gamma_1,I_1),(w_2,\Gamma_2,I_2)\in\slsh$ and let $w\in\shuffle(w_1,w_2)$.  Let $\varphi= \varphi_1\dju \varphi_2$ and $\psi=\psi_1\dju \psi_2$ be facets of $\Gamma_1*\Gamma_2$, with $\varphi_i,\psi_i\in\Gamma_i$, such that $\psi<_w\varphi$.  Then either $\psi_1 <_w \varphi_1$ or $\psi_2 <_w \varphi_2$; assume without loss of generality that the first case holds.  Since restricting $<_w$ to~$I_1$ gives a shelling order of $\Gamma_1$, it follows that $\Gamma_1$ has a facet $\rho_1<_w\varphi_1$ and a vertex $x\in \varphi_1$ such that $\psi_1\cap \varphi_1\subseteq \rho_1\cap \varphi_1 = \varphi_1\sm\{x\}$.  Then it is routine to check that $\rho=\rho_1\dju \varphi_2$ satisfies~\eqref{shell}, verifying that $(w_1,\Gamma_1)\oj{w}(w_2,\Gamma_2)\in\slsh$.
\end{example}

Not every prefix-pure complex is shellable.  For example, Ziegler~\cite{UCantShellThis} constructed a non-shellable 3-dimensional ball $Z$ with 10 vertices and 21 facets.  According to computation with SageMath, $Z$ is prefix-pure under 6528 of the $10! =  3628800$ possible vertex orderings.

\begin{example}[\bf Shifted complexes]\label{ex:shifted}
An ordered simplicial complex $(w,\Gamma,I)$ is \defterm{shifted} if, whenever $\gamma\in\Gamma$ and $e\in\gamma$, then $\gamma\cup{f}\backslash {e}$ is a face for every $f<_we$.  When $\Gamma$ is pure of dimension $d-1$, this is equivalent to the statement that the facets of $\Gamma$ form an order ideal in \defterm{Gale order}, which is the following partial order on $\binom{I}{d}$: for $\varphi,\psi\in\binom{I}{d}$ with $\varphi=\{f_1 <_w \cdots <_w f_d\}$ and $\psi=\{g_1 <_w \cdots <_w g_d\}$ we have
$\varphi\leqgale \psi \quad\iff\quad f_i\leq_wg_i \ \ \forall\,i$.
In fact Gale order on $\binom{I}{d}$ is a distributive lattice, isomorphic to the principal order ideal generated by a $d\x(|I|-d)$ rectangle in Young's lattice of integer partitions.  For more detail, see \S\ref{sec:shifted-matroids}.

Initial restrictions and initial contractions of shifted complexes are easily seen to be shifted.  However, pure shifted complexes are not closed under ordered join and therefore do not form a Hopf class.  Nevertheless, the class of \textit{ordered joins} of shifted complexes is a Hopf class $\shift$, because join is compatible with initial restriction and contraction (Lemma~\ref{lem:joins}).
\end{example}

\begin{example}[\bf Quasi-matroidal classes]
The \emph{quasi-matroidal classes} studied in~\cite{Samper} are defined as Hopf classes that satisfy additional conditions: a quasi-matroidal class must contain all ordered matroids and shifted complexes, and be closed under taking links of arbitrary faces (not just initial segments).  The unique smallest quasi-matroidal class
is $\qmin=\overline{\omat\cup\shift}$.
The unique largest quasi-matroidal class $\pure$ \cite[Example 3.3]{Samper} is defined recursively as follows: $(w,\Gamma,I)\in\pure$ if either
$\Gamma$ has exactly one facet, or both the following conditions hold:
\begin{itemize}
\item $(w|_{I'},\Gamma|I')\in\pure$, where $I'$ is obtained by deleting the $w$-maximal non-cone vertex; and
\item $(w|_{I\sm F},\link_\Gamma(F))\in\pure$ for every $F\in\Gamma$.
\end{itemize}
All complexes in $\pure$ are vertex-decomposable \cite[Thm.~3.5]{Samper}, hence shellable; on the other hand, Hopf classes can contain non-shellable complexes, such as Ziegler's non-shellable ball $Z$.

The quasi-matroidal classes $\qi$, $\qe$, and $\qc$ are defined by the \emph{quasi-independence}, \emph{quasi-exchange}, and \emph{quasi-circuit} axioms respectively \cite[Defn.~4.1]{Samper}.  None of these classes is contained in another one \cite[Thm.~4.3]{Samper}, so each strictly contains $\shift$ and is strictly contained in $\pure$.
\end{example}

\begin{example}[\bf Gale truncations]\label{ex:gale}
Let $(w,\Gamma,I)$ be a pure ordered complex of dimension~$d-1$ and let $\JJ\subseteq\binom{I}{d}$ be an order ideal in Gale order (see Example~\ref{ex:shifted}).  The \defterm{Gale truncation} of $(w,\Gamma)$ at $\JJ$ is $(w,\Gamma_\JJ)$, where $\Gamma_\JJ$ is the (pure) subcomplex of $\Gamma$ generated by the facets in $\JJ$.  Gale truncations generalize shifted complexes, because a shifted complex is just a Gale truncation of an ordered uniform matroid.  In fact, $\shift$ is the Hopf class of all Gale truncations of direct sums of ordered uniform matroids.

According to computation with SageMath, there exists at least one vertex order~$w$ on Ziegler's non-shellable ball $Z$ such that $(w,Z)$ is prefix-pure, but not all Gale truncations are prefix-pure.  Thus Gale truncation is not a well-defined operation on Hopf classes in general. On the other hand, by~\cite[Theorem 4.11]{Samper}, the class $\qe$ is closed under Gale truncations, so for any Hopf class $\hclass\subset\qe$, the collection of Gale truncations of elements of $\hclass$ generates a Hopf class $\hclass^\mathrm{Gale}$ such that $\hclass\subseteq\hclass^\mathrm{Gale}\subseteq\qe$.
\end{example}

\begin{example}[\bf Matroid threshold complexes]
Let $\Gamma$ be a matroid independence complex on ground set $I$, let $\ell$ be a generic linear functional on $\mathbb{R}^I$, and let $r\in\Rr$. Such a generic $\ell$ induces a linear order $w_\ell$ on $I$ defined by $i<j$ iff $\ell(\ee_i)<\ell(\ee_j)$. Let $\Gamma(\ell, r)$ be the subcomplex of~$\Gamma$ generated by the facets $\varphi$ such that $\ell(\varphi)\le r$.  We call $\Gamma(\ell,r)$ a \defterm{matroid threshold complex}; by the proof of \cite[Thm.~2(B)]{HS}, it is a Gale truncation of $(w_\ell,\Gamma)$.  Thus, if $\hclass\subseteq\omat$, then the class of matroid threshold complexes of elements of $\hclass$ generates a Hopf class $\hclass^{\text{thr}}\subseteq\hclass^{\text{Gale}}\subseteq\omat^{\text{Gale}}\subseteq\qe$.

When $\Gamma$ is a simplex skeleton, the definition of matroid threshold complex reduces to the usual definition of a threshold complex (see, e.g., \cite{KlivansReiner}).  Therefore, the class $\oum^{\text{thr}}$ is the smallest Hopf class containing all threshold complexes.
\end{example}

\begin{example}[\bf Color-shifted complexes]\label{ex:color-shifted}
Color-shifted complexes generalize pure shifted complexes (which arise as the case $k=1$).  They were introduced by Babson and Novik \cite{BN}; see also \cite[\S4.4]{DKM-survey}.
A simplicial complex $\Gamma$ is \defterm{color-shifted} if its vertices can be partitioned into disjoint sets $I_1,\dots,I_k$ (``color classes'') with the following properties:
\begin{enumerate}
\item There is a ``palette vector'' ${\bf a}=(a_1,\dots,a_k)\in\Nn^k$ such that every facet of $\Gamma$ has exactly $a_i$ vertices of color~$i$.  (That is $\Gamma$ is ``{\bf a}-balanced''; this condition was introduced in \cite{balanced}.)
\item The color classes are equipped with linear orders $w_1,\cdots,w_k$, with the following property: if $x,y$ both have color~$j$ with $x<_{w_j}y$, and $\gamma\in\Gamma$ contains $y$ but not $x$, then $\gamma\sm\{y\}\cup\{x\}\in\Gamma$.
\end{enumerate}
Note that when $k=1$, the definition reduces to that of a pure shifted complex.

We define an ordered complex $(w,\Gamma,I)$ to be color-shifted if there exists a partition $I=I_1\sqcup\cdots\sqcup I_k$ and orderings $w_1,\dots,w_k$ such that $\Gamma$ is color-shifted in the above sense and $w\in\shuffle(w_1,\dots,w_k)$.  This property is easily seen to be preserved by initial restriction, initial contraction, and ordered join. Thus color-shifted complexes constitute a Hopf class $\Color$.

Evidently $\shift\subseteq\Color$.  In fact, this inclusion is strict.  Let $I_1=\{x_1,y_1\}$ and $I_2=\{x_2,y_2\}$ and
$\Gamma=\langle x_1x_2, x_1y_2, y_1x_2\rangle$.  This complex is color-shifted under the orderings $x_i<_{w_i}y_i$, but it is join-irreducible and is not shifted with respect to any ordering on $I_1\cup I_2$.

The inclusions $\Color\subseteq\qe$ and $\Color\subseteq\qi$ can be proven by adapting the relevant parts of the proof of~\cite[Thm.~4.2]{Samper} from shifted to color-shifted complexes \textit{mutatis mutandis}.  Both inclusions are strict because matroid complexes are in general not color-shifted.  We suspect that $\Color\not\subseteq\qc$ (the proof from \cite{Samper} does not carry over to the color-shifted setting).
\end{example}

\begin{example}[\bf Broken-circuit complexes]\label{ex:BC}
Let $(w,\Gamma,I)$ be an ordered matroid complex.  Recall that a \defterm{circuit} of the corresponding matroid is a minimal non-face of $\Gamma$.  A \defterm{broken circuit} is a set of the form $C\sm\min_w(C)$, where $C$ is a circuit.  The \defterm{unreduced broken-circuit complex} $\BC_w(\Gamma)$ consists of all subsets of~$I$ containing no broken circuit.  This complex is always a cone over the first vertex $x_1$ of~$w$; the base of the cone is called the \defterm{reduced broken-circuit complex} $\rBC_w(\Gamma)$.  Both $\BC_w(\Gamma)$ and $\rBC_w(\Gamma)$ are pure \cite[Prop.~7.4.2]{bjorner} and lex-shellable \cite[Thm.~7.4.3]{bjorner}.  Reduced broken-circuit complexes are prefix-pure \cite[Prop.~4.4]{bry}; more strongly, one can see that $\BC_w(\Gamma)|A=\BC_{w|_A}(\Gamma|A)$ for every $A\in\Initial(w)$, so $\BC_w(\Gamma)$ is strongly lex-shellable, and this property is preserved by deconing.  Reduced broken-circuit complexes do not form a Hopf class because they are not closed under contractions of initial segments, but they generate a Hopf subclass $\bc\subseteq\slsh$.  Moreover, the class of broken-circuit complexes strictly includes that of matroid complexes \cite[Thm.~4.2, Rmk.~4.3]{bry}, so $\omat\subsetneq\bc$.
\end{example}

We summarize the hierarchy of Hopf classes of prefix-pure complexes as follows.

\begin{proposition} \label{Hopf-hierarchy}
The Hopf classes described above are ordered by inclusion as in Figure~\ref{fig:Hopf-class-hierarchy}.
\end{proposition}

\begin{figure}[ht]
\begin{center}
\begin{tikzpicture}

\coordinate (pref) at (0,7);
\coordinate (slsh) at (0,6);
\coordinate (pure) at (0,5);
\coordinate (bc) at (4,3);
\coordinate (qe) at (0,4);
\coordinate (qi) at (-2,4);
\coordinate (qc) at (2,4);
\coordinate (omatplus) at (8,4-.05);
\coordinate (color) at (-2,3);
\coordinate (qmin) at (0,3);
\coordinate (omatgale) at (2,3.1);
\coordinate (shift) at (0,2);
\coordinate (omat) at (3,2);
\coordinate (oum) at (0,1);
\coordinate (osim) at (0,0);
\coordinate (triv) at (0,-1);

\draw[dashed] (pref)--(omatplus);
\draw[thick] (pref)--(slsh);
\draw[dashed] (slsh) -- (pure);
\draw[thick] (pure) -- (qe);
\draw[thick] (pure) -- (qi);
\draw[thick] (pure) -- (qc);
\draw[thick] (qi) -- (color);
\draw[thick] (qe) -- (color);
\draw[thick] (qe) -- (qmin);
\draw[thick] (qe) -- (omatgale);
\draw[thick] (qi) -- (omatgale);
\draw[thick] (qc) -- (qmin);
\draw[thick] (qi) -- (qmin);
\draw[dashed] (slsh) -- (bc);
\draw[thick] (qmin) -- (shift);
\draw[thick] (qmin) -- (omat);
\draw[thick] (color) -- (shift);
\draw[thick] (bc) -- (omat);
\draw[thick] (omatgale) -- (omat);
\draw[thick] (omatgale) -- (shift);
\draw[thick] (omatplus) -- (omat);
\draw[thick] (shift) -- (oum);
\draw[thick] (omat) -- (oum);
\draw[thick] (oum) -- (osim);
\draw[thick] (osim) -- (triv);

\node[fill=white] at (pref) {$\uhopf$};
\node[fill=white] at (slsh) {$\slsh$};
\node[fill=white] at (pure) {$\pure$};
\node[fill=white] at (qe) {$\qe$};
\node[fill=white] at (qc) {$\qc$};
\node[fill=white] at (qi) {$\qi$};
\node[fill=white] at (bc) {$\bc$};
\node[fill=white] at (omatplus) {$\omat^+$};
\node[fill=white] at (color) {$\Color$};    
\node[fill=white] at (qmin) {$\qmin$};    
\node[fill=white] at (omatgale) {$\omat^{\Gale}$};    
\node[fill=white] at (shift) {$\shift$};
\node[fill=white] at (omat) {$\omat$};
\node[fill=white] at (oum) {$\oum$};
\node[fill=white] at (osim) {$\osim$};
\node[fill=white] at (triv) {$\triv$};
\end{tikzpicture}
\end{center}
\caption{The hierarchy of Hopf classes of prefix-pure complexes.    Solid lines indicate inclusions known to be strict; dashed lines indicate possible equalities.  All classes are described in \S\S\ref{sec:Hopf-class-basic}--\ref{sec:Hopf-class-zoo}, except $\omat^+$, which will be described in \S\ref{omatplus-from-hclass}.\label{fig:Hopf-class-hierarchy}}
\end{figure}
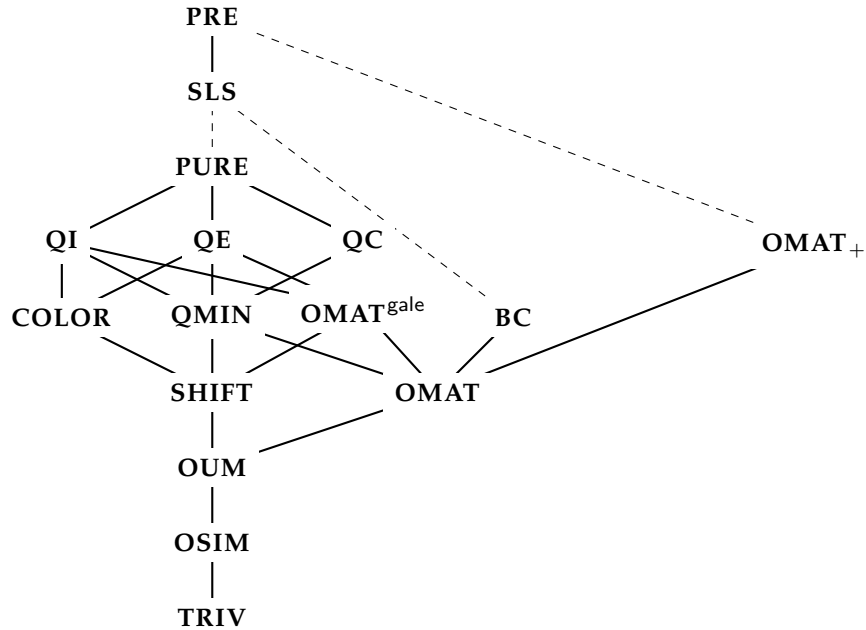

\begin{conjecture} \label{conj:bc-omat-gale-qe}
$\bc\subsetneq\omat^{\Gale}\subsetneq\qe\cap\qi$.
\end{conjecture}

The first inclusion in Conjecture~\ref{conj:bc-omat-gale-qe} is very close to~\cite[Thm.~1.4]{CastilloSamper}, which states that every broken-circuit complex is an order ideal in Las~Vergnas' internal activity poset (see~\cite[Thm.~3.4]{LV}), whose relations are all relations in Gale order.  The conjecture is true when $\Gamma$ is shifted, since then its broken-circuit complex is also shifted~\cite[Thm.~5]{Klivans-SMC}.  Moreover, the $f$-vectors of broken-circuit complexes are more constrained than those of shifted complexes, for the following reasons.  Pure shifted complexes are shellable, hence Cohen-Macaulay; on the other hand, the Cohen-Macaulay property and the $f$-vector are preserved by algebraic shifting~\cite[\S4.1]{Kalai-shifting}, so the $f$-vectors of pure shifted complexes coincide with those of Cohen-Macaulay complexes and satisfy the so called $M$-sequence inequalities. On the other hand, $f$-vectors of broken circuit complexes form a much smaller class. As discussed in \cite[Section 4]{CastilloSamper}, the family of $h$-vectors of broken-circuit complexes with $d$ positive entries and last entry $k$ is finite, while for Cohen-Macaulay complexes, and hence shifted complexes, it is infinite.

The discussion in Example~\ref{ex:gale} implies $\omat^{\Gale}\subseteq\qe$, and moreover $\omat^{\Gale}\subseteq\qi$ by \cite[Thm.~5.8]{Samper}.  Since $\qe$ and $\qi$ are incomparable, both these inclusions are strict.  We expect that the axioms for $\qe$ and $\qi$ are too loose to capture all Gale truncations of matroids.  We note that this conjecture cannot be proved on the level of $f$-vectors, since all simplicial complexes in $\qe\cap\qi$ are shellable, and $\omat^{\Gale}$ contains all shifted complexes, hence all Cohen-Macaulay $f$-vectors by the discussion above.

\begin{conjecture} \label{conj:omat-and-slsh}
$\omat^+\subseteq\slsh$.
\end{conjecture}
It was proven in \cite[Thm.~2(A)]{HS} that every generic real-valued function on the ground set of a matroid, when extended linearly to bases, gives rise to a shelling order on the independence complex.  A proof of Conjecture~\ref{conj:omat-and-slsh} might proceed along similar lines.

\subsection{Hopf monoids from Hopf classes} \label{sec:hmonoid-to-hclass}
We now show that every Hopf class $\hclass$ naturally gives rise to a Hopf monoid $\mathbf{H}=\hclass^\natural$.  Define a set species $\mathbf{h}[I]$ to be the set of all ordered complexes in $\hclass$ on ground set~$I$, and let $\mathbf{H}$ be the corresponding vector species.  Observe that these species are connected, since $\mathbf{h}[\0]$ contains only the trivial ordered complex; this gives rise to a unit $u$ and a counit $\epsilon$.  We define a Hopf product $\mu_{I,J}$ and coproduct $\Delta_{I,J}$ on basis elements by join and restriction/contraction respectively:
\begin{subequations}
\begin{equation} \label{Hopf-class-product}
\mu_{I,J}((w_1,\Gamma_1)\otimes (w_2,\Gamma_2)) = \sum_{w\in\shuffle(w_1, w_2)} (w_1,\Gamma_1)\oj{w}(w_2,\Gamma_2),
\end{equation}
\begin{equation} \label{Hopf-class-coproduct}
\Delta_{I,J}(w,\Gamma) = \begin{cases}
(w|_I,\Gamma|I)\otimes (w|_J,\Gamma/ I) & \text{ if } I \text{ is an initial segment of } w,\\ 
0 & \text{ otherwise.}
\end{cases}
\end{equation}
\end{subequations}

\begin{theorem}\label{thm:quasiHopf}
For every Hopf class $\hclass$, the tuple $(\mathbf{H}, \mu, u, \Delta, \epsilon)$ is a connected Hopf monoid.
\end{theorem}

\begin{proof}[Proof of Theorem \ref{thm:quasiHopf}]
Most of the Hopf monoid axioms are straightforward; the only ones that require substantial proof are coassociativity and compatibility.

First, we check coassociativity.  Let $(w,\Gamma,I)$ be an ordered complex and $I= A\dju B\dju C$.  For the sake of legibility, we drop disjoint union symbols: e.g., $AB=A\dju B$.
If $A$ and $AB$ are not both initial segments of $w$, then
\[(\Delta_{A,B} \otimes \Id)\circ \Delta_{AB, C}(w,\Gamma)= (\Id \otimes \Delta_{B,C})\circ \Delta_{A, BC}(w,\Gamma) = 0.\]
On the other hand, if $A$ and $AB$ are both initial segments, then
\begin{align*}
\Delta_{A,B} \otimes \Id_C\big( \Delta_{AB, C}(w,\Gamma)\big)&~=~
\big(w,\Gamma\big)|AB|A &\otimes\ &
\big(w,\Gamma\big)|AB/ A &\otimes\ &
\big(w,\Gamma\big)/AB,\\
\Id_A \otimes\, \Delta_{B,C}\big( \Delta_{A, BC}(w,\Gamma)\big)&~=~
\big(w,\Gamma\big)|A &\otimes\ &
\big(w,\Gamma\big)/ A|B &\otimes\ &
\big(w,\Gamma\big)/ A/ B,
\end{align*}
and by Lemma~\ref{lem:coprod} the two expressions coincide.

Second, we check compatibility.  Let $(w_1,\Gamma_1,AB)$ and $(w_2,\Gamma_2,CD)$ be ordered complexes, and for short let $\xi$ denote their tensor product in $\mathbf{H}[AB]\otimes\mathbf{H}[CD]$. If either $A\notin\Initial(w_1)$ or $C\notin\Initial(w_2)$, then $\Delta_{A,B}\x\Delta_{C,D}(\xi)=0$, and in addition $AC$ is not an initial segment of any shuffle of $w_1$ and $w_2$, so $\Delta_{AC, BD}(\mu_{AB, CD}(\xi))=0$ as well.

On the other hand, suppose that $A\in\Initial(w_1)$ and $C\in\Initial(w_2)$. Notice that if $w$ is in $\shuffle(w_1,w_2)$, then $\Delta_{AC, BD}((w_1,\Gamma_1)\oj{w}(w_2,\Gamma_2))$ is non-zero only if $AC\in \Initial(w)$. Hence
\begin{align}
\Delta_{AC, BD}(\mu_{AB, CD}(\xi))
&= \sum_{w\,\in\,\shuffle(w_1,w_2)} \Delta_{AC, BD}((w_1,\Gamma_1)\oj{w} (w_2,\Gamma_2))\notag\\
&= \sum_{\substack{w\,\in\,\shuffle(w_1,w_2):\\AC \,\in\, \Initial(w)}}
\big((w_1,\Gamma_1)|AC \oj{w} (w_2,\Gamma_2)|AC\big)\otimes
\big((w_1,\Gamma_1)/AC \oj{w} (w_2,\Gamma_2)/AC\big). \label{compat:1}
\end{align}

Compatibility asserts that the last expression equals
\begin{align}
&\left(\mu_{A,C}\otimes \mu_{B,D}\right) \circ \tau \circ \left(\Delta_{A,B}\otimes \Delta_{C,D}\right)(\xi)\notag\\
&\quad= \mu_{A,C}\otimes \mu_{B,D}\Big((w_1,\Gamma_1)|A\otimes (w_2,\Gamma_2)|_C \otimes (w_1,\Gamma_1)/ A\otimes (w_2,\Gamma_2)/ C\Big)\notag\\
&\quad= \sum_{\substack{u\,\in\,\shuffle(w_1|A,w_2|C)\\ v\,\in\,\shuffle(w_1/A, w_2/C)}} \Big((w_1,\Gamma_1)|A \oj{u} (w_2,\Gamma_2)|_C \Big) \otimes \Big((w_1,\Gamma_1)/ A\oj{v} (w_2,\Gamma_2)/ C \Big).
\label{compat:2}
\end{align}

In fact, the index sets of the two sums~\eqref{compat:1} and~\eqref{compat:2} are in bijection: $w$ in the former may be taken to be the concatenation of $u$ and $v$ in the latter.  Moreover, the summands are equal by Lemma~\ref{lem:joins}, completing the proof of compatibility.
\end{proof}

Of particular interest to us are the Hopf monoids $\UHopf=\uhopf^\natural$ (the largest Hopf monoid of pure ordered simplicial complexes in which we can work) and $\OMat=\omat^\natural$.

\begin{proposition}
There is an isomorphism of Hopf monoids $\OMat\isom\bL^*\x\Mat$.
\end{proposition}
\begin{proof}
In both cases the basis elements are matroid complexes equipped with an ordering of the ground set, giving equality on the level of vector species.
Moreover, the product and coproduct defined on $\OMat$ by~\eqref{Hopf-class-product} and~\eqref{Hopf-class-coproduct} coincide with the Hopf operations on $\bL^*\x\Mat$ described in \S\ref{sec:exHopf}.
\end{proof}

\begin{proposition}
If $\hclass$ is a Hopf class containing $\omat$, then the map $\OMat\to\mathbf{H}$ given by $(w,M)\mapsto (w,\II(M))$ is a Hopf monoid monomorphism.
\end{proposition}
\begin{proof}
For independence complexes of matroids, the product and coproduct in $\mathbf{H}$ coincide with those in $\OMat$. 
\end{proof}

\section{Ordered generalized permutohedra}\label{sec:ogp}
Having shown that every Hopf class $\hclass$ gives rise to a Hopf monoid ${\bf H}$, we now want to find a cancellation-free formula for the antipode in ${\bf H}$ in terms of the combinatorics of~$\hclass$.  We start by focusing on Hopf classes with inherent geometry, such as $\omat$, where we can hope to express the coefficients of the antipode as appropriate Euler characteristics, as in \cite{AA}.  In order to do this, we may as well work in the more general setting of ordered generalized permutohedra. 

As stated in the introduction, \cite[Thm.~4.1]{GGMS} states that a polyhedron $\pp\subset\Rr^I$ is a matroid base polytope for some matroid on $I$ if and only if $\pp$ satisfies the following three conditions:
\begin{enumerate}
    \item[(M1)] $\pp$ is bounded;
    \item[(M2)] $\pp$ is an 0/1-polyhedron;
    \item[(M3)] Every edge of $\pp$ is parallel to some $\ee_i-\ee_j$, where $\{\ee_i\}_{i\in I}$ is the standard basis of $\Rr^I$.
\end{enumerate}
Conditions~(M1) and~(M3) together define the class of generalized permutohedra (see \S\ref{sec:background-for-gp}), while~(M3) alone defines extended generalized permutohedra.  In order to stay within the realm of Hopf classes of ordered simplicial complexes, we must retain condition~(M2).  In \S\ref{sec:OIGP}, we will study the Hopf class of ordered simplicial complexes arising from possibly-unbounded polyhedra satisfying conditions~(M2) and~(M3).  

\begin{definition}
An \defterm{ordered generalized permutohedron} is a pair $(w,\pp)$, where $\pp\subset\Rr^I$ is a generalized permutohedron and $w$ is a linear order on $I$.  
The \defterm{Hopf monoid of ordered generalized permutohedra} is the Hadamard product $\OGP=\bL^*\x\GP$.
\end{definition}
Like both $\bL^*$ and $\GP$, the monoid $\OGP$ is commutative but not cocommutative.  Since $\bL^*$ is not linearized as a Hopf monoid, neither is $\OGP$.  The inclusion $\Mat\inj\GP$ gives rise to an inclusion $\OMat\inj\OGP$, where $\OMat=\bL^*\x\Mat$, the \defterm{Hopf monoid of ordered matroids.}

\begin{remark} \label{why-not-L:1}
It is also possible to study the Hopf monoid $\bL\x\GP$.  This monoid, unlike $\bL^*\x\GP$, is a linearized Hopf monoid, so its antipode can be calculated using the methods of Benedetti and Bergeron \cite{linearized}.  Moreover, the projection map $\pi:\bL\x\GP\to\GP$ is a morphism of Hopf monoids, unlike the map $\pi^*:\bL^*\x\GP\to\GP$, which is only a morphism of vector species---if $w\otimes\pp\in\OGP[I]$ and $v\otimes\qq\in\OGP[J]$, then
$\pi(w\otimes\pp)*\pi(v\otimes\qq)=\pp\x\qq$, but
$\pi((w\otimes\pp)*(v\otimes\qq)) = \binom{|I|+|J|}{|I|}\pp\x\qq$.

On the other hand, $\bL\x\GP$ is neither commutative nor cocommutative (since $\bL$ is cocommutative but not commutative), and does not arise from a Hopf class.  In the ordered setting, deletion and contraction require splitting the ground set into an initial and a final segment, which corresponds to the coproduct in $\bL^*$.  As we will see, another good property that $\bL^*\x\GP$ enjoys but $\bL\x\GP$ lacks is that its antipode is \textit{local} in a precise geometric sense; see Proposition~\ref{locality} and Remark~\ref{why-not-L:2}.
\end{remark}

The idea of ordering coordinates by forming a Hadamard product with $\bL^*$ carries over from generalized permutohedra to EGPs.  However, the appropriate Hopf monoid to consider is \textit{not} the full Hadamard product $\bL^*\x\GP^+$, for the following reason.  Suppose that $\pp$ is an EGP in $\Rr^I$, and let $\cA\compn I$.  If the cone $\sigma_{\cA}$ belongs to the normal fan $\NN_\pp$, then there is a well-defined face $\pp_{\cA}$ that maximizes the linear functionals in $\sigma_{\cA}$ (and if $\cA$ is a linear order, then $\pp_{\cA}$ is a vertex).  However, if $\sigma_Q\not\subseteq |\NN_\pp|$, then linear functionals in $\sigma_{\cA}$ are unbounded on $\pp$ and no such face exists.  Accordingly, we define
\begin{equation}\label{ellpi}
\bl_\pp[I] = \{w\in\bl[I]:\ \sigma_{w^\rev}\subseteq\NN_\pp\} = \{w\in\bl[I]:\ \pp_{w^\rev} \ \text{is a well-defined vertex of $\pp$}\}.
\end{equation}
In particular, $\bl_\pp[I]=\bl[I]$ if and only if $\pp$ is bounded.

\begin{remark}\label{rem:reverse}
The reason to reverse the order is the following. To compute the restriction of $\pp\subset\Rr^{I\sqcup J}$ to the initial segment $I$ we need to find the face of $\pp$ maximized by the function $\mathbf{1}_I\in\sigma_{J|I}$, and the set composition $J|I$ is refined not by $\comp(w)$, but rather by $\comp(w)^\rev$.  Geometrically, this means that $\sigma_{w^\rev}\in\NN_\pp$.  Equivalently, linear functionals in $\sigma_w$ are bounded from \emph{below} on~$\pp$.
(The reversal would not be necessary under the conventions in~\cite{AA}; see the note in \S\ref{sec:set-comps-and-gps}.)
\end{remark}

\begin{lemma}\label{lem:lp} Let $\pp\in \Rr^{I\sqcup J}$ be a (possibly unbounded) generalized permutohedron. Then
\begin{equation}\label{eq:lp}
\{w\in\bl_\pp[I\dju J]: I\in\Initial(w)\}=\{uv:\ u\in\bl_{\pp|I}[I],\ v\in\bl_{\pp/I}[J]\}
\end{equation}
where $uv$ denotes the concatenation of $u$ and $v$.
\end{lemma}

\begin{proof}
We use the following basic fact about faces of polyhedra \cite[eqn.~(2.3), p.10]{Sturmfels}: for all linear functionals $\lambda_{\xx_1},\lambda_{\xx_2}$ on $\Rr^{I\sqcup J}$, either
\begin{equation}\label{eq:faces}
    \left(\pp_{\xx_1}\right)_{\xx_2} = \pp_{\xx_1+\epsilon\xx_2}
\end{equation}
for sufficiently small $\epsilon>0$, or neither of the faces exist. We now prove~\eqref{eq:lp} by double inclusion.

($\supseteq$) Let $u,v$ be as in the right hand side of \eqref{eq:lp} and let $w=uv$. Then the face $(\pp|I\times\pp/I)_{v^\rev u^\rev}=\left(\pp_{J|I}\right)_{w^\rev}$ exists, because the functional $\lambda_{w^\rev}$ is bounded on $\pp_{J|I}$ by virtue of being bounded on each factor.  Now let $\xx_1\in\sigma_{J|I}$ and $\xx_2\in\sigma_{w^\rev}$, then by~\eqref{eq:faces} we have
\[\left(\pp_{J|I}\right)_w = \left(\pp_{\xx_1}\right)_{\xx_2} = \pp_{\xx_1+\epsilon \xx_2}.\]
The condition $I\in\Initial(w)$ implies $J|I\refinedby\comp(w)^\rev$, so $\xx_1+\epsilon \xx_2\in\sigma_{w^\rev}$.  Therefore $\pp_{w^\rev}$ is a vertex of $\pp$, i.e., $w\in\bl_\pp[I\sqcup J]$ as desired.

($\subseteq$) Let $w$ be such that $\sigma_{w^\rev}\in\NN_\pp$ with $I\in\Initial(w)$, and let $u=w|_I$ and $v=w|_J$.  Since $\NN_\pp$ is closed it contains $\CC_{w^\rev}$. Since $J|I\refinedbyeq\comp(w)^\rev$, we have $\sigma_{J|I}\in\CC_{w^\rev}\subset\NN_\pp$. Therefore, $\pp$ has a face $\pp_{J|I}=\pp|I\x\pp/I$, on which $\lambda_{w^\rev}$ is bounded.  Moreover, for $(x,y)\in\pp|I\x\pp/I=\pp_{J|I}$, we have $\lambda_{w^\rev}(x,y)=\lambda_u(x)+\lambda_v(y)$; in particular, $\lambda_{u^\rev}$ and $\lambda_{v^\rev}$ are both bounded on $\pp|I$ and $\pp/I$ respectively.
\end{proof}

Accordingly, we define $\OGP^+$ as a vector subspecies of $\bL^*\x\GP^+$:
\begin{equation} \label{define-ogp-plus}
\OGP^+[I] = \left\langle w\otimes\pp:\ \pp\in\GP^+[I],\ w\in\bl_\pp[I] \right\rangle.
\end{equation}

\begin{theorem} \label{thm:ogp-plus}
$\OGP^+$ is a Hopf submonoid of $\bL^*\x\GP^+$. 
\end{theorem}
\begin{proof}
It suffices to show that $\mu_{I,J}(\OGP^+[I]\otimes\OGP^+[J])\subseteq\OGP^+[I\dju J]$ and $\Delta_{I,J}(\OGP^+[I\dju J])\subseteq\OGP^+[I]\otimes\OGP^+[J]$, for all disjoint finite sets $I,J$.

For the product, let $u\otimes\pp\in\OGP^+[I]$ and $v\otimes\qq\in\OGP^+[J]$, where $I\cap J=\0$.  Then
\begin{equation} \label{OGP-product}
(u\otimes\pp)(v\otimes\qq) = \sum_{w\in\shuffle(u,v)} w\otimes(\pp\x\qq).
\end{equation}
Every linear functional $\lambda$ on $\Rr^I\x\Rr^J$ is defined by $\lambda(P,Q)=\lambda|_I(P)+\lambda|_J(Q)$.
If $w\in\shuffle(u,v)$ and $\lambda\in\sigma_{w^\rev}$, then $\lambda_I$ (resp., $\lambda_J$) restricts to a functional in $\sigma_{u^\rev}$ (resp., $\sigma_{v^\rev}$), hence is maximized on $\pp$ (resp., $\qq$) at $\pp_{u^\rev}$ (resp., $\qq_{v^\rev}$).  Hence $\lambda$ is maximized on $\pp\x\qq$ at $\pp_{u^\rev}\x\qq_{v^\rev}$.  It follows that every summand in~\eqref{OGP-product} belongs to $\OGP^+[I\dju J]$.

For the coproduct, let $w\otimes\pp\in\OGP^+[I\dju J]$. Then
\[\Delta_{I,J}(w\otimes\pp) = (w|_I\otimes\pp|I)\otimes (w|_J\otimes\pp/ I).\]
Recall that $\pp|I \times \pp/ I$ is a face of $\pp$, so any linear functional $\lambda\in\sigma_{w^\rev}$ is bounded above on $\pp$ and hence on $\pp| I \times \pp/ I$. The restriction of $\lambda$ to  each element in the product is a pair of functionals, one in direction $w^\rev|_I$ and another one in $w^\rev|_J$, and each of those functionals is bounded above on $\pp|I$ and $\pp/ I$ respectively.
\end{proof}

\begin{theorem}\label{thm:symmetrization}
The symmetrization map $\Symm: \GP^+\rightarrow \OGP^+$ defined on $\GP^+[I]$ by
\[
\Symm(\pp) = \pp^\# = \sum_{w\in\bl_\pp[I]} w\otimes\pp
\]
is an injective Hopf morphism.  It restricts to a map $\GP\rightarrow \OGP$ given by $\pp^\# = \sum_{w\in\bl[I]} w\otimes\pp$.
\end{theorem}

\begin{proof}
It is straightforward to check that $\Symm$ is a morphism of vector species,
so the proof reduces to checking that it that commutes with products and coproducts.

First, let $\pp\in\GP^+[I]$, $\qq\in\GP^+[J]$, where $I,J$ are disjoint finite sets.  We must show that $\mu_{I,J}^{\OGP^+}(\pp^\#, \qq^\#) = (\mu_{I,J}^{\GP^+}(\pp,\qq))^\#$.  Indeed,
\begin{align*} 
\mu_{I,J}(\pp^\#, \qq^\#) &= \mu_{I,J}\left(\sum_{u\in \bl_\pp[I]} u\otimes \pp ,\ \sum_{v\in \bl_\qq[J]} v\otimes \qq \right) \\ 
&= \sum_{u\in\bl_\pp[I]}\sum_{v\in \bl_\qq[J]} \left(\sum_{w\in \shuffle(u,v) } w\otimes (\pp\times \qq)\right) \\ 
&= \left( \sum_{u\in \bl_\pp[I]}\sum_{v\in \bl_\qq[J]}\sum_{w\in \shuffle(u,v)} w \right)\otimes(\pp\times\qq) \\ 
&= \left(\sum_{w\in \bl_{\pp\x\qq}[I\dju J]} w\right)\otimes (\pp\times \qq) \\
&= \mu_{I,J}(\pp,\qq)^\#.
\end{align*}
The second-to-last equality follows because every linear order on $I\dju J$ decomposes uniquely as a shuffle of a linear order in $I$ and a linear order in $J$, and because the product $\pp\x\qq$ is bounded in direction $w^\rev$ if and only if $\pp$ and $\qq$ are bounded in directions $w^\rev|_I$ and $w^\rev|_J$ respectively.

Second, let $I,J$ be disjoint finite sets and $\pp\in\GP^+[I\dju J]$. We must show that $\Delta_{I,J}^{\OGP^+}(\Symm(\pp))=(\Symm\otimes\Symm)\circ \Delta_{I,J}(\pp)$.  Indeed,

\begin{align*}
\Delta_{I,J}^{\OGP^+}(\Symm(\pp)) &= \Delta_{I,J}\left(\sum_{w\in\bl_\pp[I\dju J] }w\otimes \pp\right) \\
&= \sum_{w\in\bl_\pp[I\dju J]:\ I \in \Initial(w)} (w|_I\otimes \pp|I) \otimes (w|_J\otimes \pp/ I) \\ 
&= \sum_{u\in\bl_{\pp|I}[I]}\sum_{v\in\bl_{\pp/I}[J]} (u\otimes\pp|I)\otimes (v\otimes \pp/ I) \\
&= (\pp|I)^\# \otimes (\pp/ I)^\#  \\
&= (\Symm\otimes\Symm)\circ \Delta_{I,J}(\pp)
\end{align*}
where the third equality follows from Lemma \ref{lem:lp}.
\end{proof}

At this point we can write down a ``symmetrized antipode'' formula:
\begin{equation} \label{eq:antisym}
\sum_{w\in\bl_\pp[I]}\anti_I(w\otimes\pp)
= \anti_I(\pp^\#)
= \anti_I(\pp)^\#
= (-1)^{|I|} \sum_{\substack{\qq\leq\pp\\ u\in\bl_\pp[I]}} (-1)^{\dim \qq}(u\otimes\qq).
\end{equation}
The second equality arises because symmetrization is a Hopf morphism, hence commutes with the antipode, and the last equality follows from the Aguiar--Ardila formula~\eqref{eq:antipodegp} for the antipode in $\GP^+$.  Note that the right-hand side of~\eqref{eq:antisym} is cancellation- and multiplicity-free.  On the other hand, as we will see, the individual summands $\anti(w\otimes\pp)$ can be extremely complicated.

\begin{remark} \label{remark:quotient-monoids}
Aguiar and Ardila~\cite[\S\S1.4.3, 1.4.5]{AA} defined quotient monoids $\overline{\GP}$ and $\overline{\GP^+}$ whose basis elements are equivalence classes of (extended) generalized permutohedra up to normal equivalence (i.e., equality of normal fans).
All of our results, including the antipode formula to be proved in \S\ref{sec:antipode}, are expressed in terms of normal fans, hence carry over \emph{mutatis mutandis} to $\bL^*\x\overline{\GP}$ and $\bL^*\x\overline{\GP^+}$.
\end{remark}

\section{0/1-extended generalized permutohedra}\label{sec:OIGP}

We next study the indicator complexes of possibly-unbounded 0/1-EGPs and show that they form a Hopf subclass of $\OGP^+$ that contains $\OMat$.

Define a vector subspecies $\OIGP^+\subseteq\OGP^+$ by
\[\OIGP^+[I]=\left\langle w\otimes\pp \in \OGP^+[I]:\ \pp\text{ is a 0/1-EGP} \right\rangle.\]
Evidently $\OIGP^+$ is a Hopf submonoid of $\OGP^+$, because the 0/1-condition is clearly closed under taking products and faces.  A bounded 0/1-EGP $\pp$ is precisely a matroid base polytope \cite[Thm.~4.1]{GGMS}, and the indicator complex $\Upsilon(\pp)$ is just the independence complex of the associated matroid, giving an injection $\Mat\to\OIGP^+$.  On the other hand, if $\pp$ is unbounded, then $\Upsilon(\pp)$ need not be a matroid complex, and some care is required in seeing how to derive $\OIGP^+$ from a Hopf class.

\begin{example}\label{ex:OIGP}
The hypersimplex $\Delta_{2,4}\subset\Rr^I=\Rr^4$ is the solution set of the following equation and inequalities: 
\begin{align*}
x_1+x_2+x_3+x_4&=2 & x_i&\leq1,\quad  (\forall i\in[4]) & x_i+x_j+x_k&\leq 2, \quad (1\leq i<j<k\leq 4)\\
\end{align*}
Ignoring the inequalities $x_4 \leq 1$ and $x_2+x_3+x_4 \leq 2$ produces the unbounded generalized permutohedron $\pp$ shown in Figure~\ref{fig:egp}, with rays pointing in the direction $\ee_4-\ee_1$.  Let $w=1234\in\Sym_4$; then $w\in\bl_\pp[I]$ (specifically, $\pp_{w^\rev}=1100$) and so $(w,\pp)\in\OGP^+$.  Moreover, the indicator complex $\Upsilon(\pp)=\langle 12,13,23,14\rangle$ is prefix-pure with respect to~$w$; we will see that this is not an accident.  On the other hand, $\Upsilon(\pp)$ is not a matroid complex, since its restriction to $\{2,3,4\}$ is not pure.

\begin{figure}[ht]
\begin{center}
\begin{tikzpicture}[scale=0.75]
\newcommand{\scl}{0.7}
\newcommand{\xiooi}{0*\scl}	\newcommand{\yiooi}{6*\scl}
\newcommand{\xooii}{3*\scl}	\newcommand{\yooii}{4*\scl}
\newcommand{\xioio}{-1*\scl}	\newcommand{\yioio}{4*\scl}
\newcommand{\xoioi}{1*\scl}	\newcommand{\yoioi}{2*\scl}
\newcommand{\xiioo}{-3*\scl}	\newcommand{\yiioo}{2*\scl}
\newcommand{\xoiio}{0*\scl}	\newcommand{\yoiio}{0*\scl}
\coordinate (oiio) at (\xoiio,\yoiio);
\coordinate (ioio) at (\xioio,\yioio);
\coordinate (ooii) at (\xooii,\yooii);
\coordinate (iioo) at (\xiioo,\yiioo);
\coordinate (oioi) at (\xoioi,\yoioi);
\coordinate (iooi) at (\xiooi,\yiooi);
\foreach \coo in {(oiio),(iooi),(iioo)} \draw[dashed] (ioio)--\coo;
\draw[thick] (oiio)--(iioo)--(iooi);
\draw[dashed] (ioio)--(4+\xioio,\yioio);
\draw[thick] (iooi)--(4+\xiooi,\yiooi);
\draw[thick] (oiio)--(4+\xoiio,\yoiio);
\draw[thick] (iioo)--(4+\xiioo,\yiioo);
\draw[thick,gray,fill=white] (\xooii,\yooii) circle(.1); \node[gray] at (\xooii,\yioio-.35) {\footnotesize0011};
\draw[thick,gray,fill=white] (\xoioi,\yoioi) circle(.1); \node[gray] at (\xoioi,\yoioi-.35) {\footnotesize0101};
\draw[fill=black] (\xoiio,\yoiio) circle(.1); \node at (\xoiio,\yoiio-.5) {\footnotesize0110};
\draw[fill=black] (\xioio,\yioio) circle(.1); \node at (\xioio+.75,\yioio-.25) {\footnotesize1010};
\draw[fill=black] (\xiioo,\yiioo) circle(.1); \node at (\xiioo-.75,\yiioo) {\footnotesize1100};
\draw[fill=black] (\xiooi,\yiooi) circle(.1); \node at (\xiooi,\yiooi+.5) {\footnotesize1001};
\end{tikzpicture}
\end{center}
\caption{An unbounded 0/1-EGP whose indicator complex is not a matroid. \label{fig:egp}}
\end{figure}
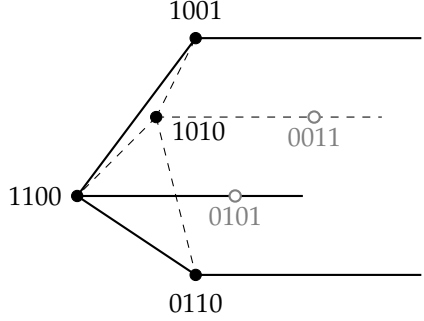
\end{example}

Example~\ref{ex:OIGP} illustrates that indicator complexes of 0/1-EGPs form a nontrivial extension of the class of matroids.  These ``unbounded matroids'' are studied in more detail in~\cite{UM}.

\begin{theorem} \label{omatplus-hclass}
The family of ordered simplicial complexes
\[\omat^+ = \{(w,\Upsilon(\pp),I):\ \pp\subset\Rr^I \text{ is a 0/1-EGP with } w\in\bl_\pp[I]\}\]
is a Hopf class.
\end{theorem}

\begin{proof}
First we show that every element of $\omat^+$ is prefix-pure.  Let $(w,\Upsilon(\pp),I\sqcup J)\in\omat^+$, where $I\in\Initial(w)$, and let $\mathbf{1}_I$ be the linear functional defined by $\mathbf{1}_I(\xx)=\sum_{i\in I}x_i$.  Then $\mathbf{1}_I\in\sigma_{J|I}\subseteq\overline{\sigma_{w^\rev}}$, so $\mathbf{1}_I$ is bounded from above on $\pp$; in particular it is a face of $\pp$ containing $\pp_{w^\rev}$.  The complex $\Upsilon(\pp)|I$ is generated by the faces $\supp(\xx)\cap I$ for all vertices $\xx\in\pp$, while $\Upsilon(\pp|I)$ is the pure subcomplex of~$\Upsilon(\pp)|I$ generated by faces of maximum size (which correspond to vertices of $\pp$ maximized by $\mathbf{1}_I$).  Therefore, it suffices to show that $\Upsilon(\pp)|I=\Upsilon(\pp|I)$. 

Accordingly, let $\gamma=\supp(\xx)\cap I$ be a face of~$\Upsilon(\pp)|I$.  If $\xx$ is not maximized by $\mathbf{1}_I$, then $\pp$ must contain some 1-dimensional face~$e$ incident to~$\xx$ such that walking along~$e$ from~$\xx$ increases the value of $\mathbf{1}_I$.  Since $\pp$ is an EGP, this walk must be in the direction $\ee_i-\ee_j$ for some $i\in I$ and $j\in J$, and $\mathbf{1}_I$ is bounded from above on $\pp$, so the walk must eventually reach another vertex $\yy\in\pp$.  But $\pp$ is an 0/1-polyhedron, so it must be the case that $x_i=y_j=0$ and $x_j=y_i=1$, and $\supp(\yy)=\supp(\xx)\cup\{i\}\sm\{j\}$.  Therefore $\supp(\yy)\cap I=\gamma\cup\{i\}\in\Upsilon(\pp)|I$.  This argument shows that every facet of~$\Upsilon(\pp)|I$ is a face of~$\Upsilon(\pp|I)$, as required.

Closure under initial restriction and contraction follow from Lemma~\ref{lem:lp}. For closure under ordered join, let $(w_1,\Upsilon(\pp_1),I_1),(w_2,\Upsilon(\pp_2),I_2)\in\omat^+$. The faces of $\pp_1\x\pp_2$ are products of faces of $\pp_1$ with faces of $\pp_2$, so
\[
w\in\bl_\pp[I_1\dju I_2]
\quad\iff\quad w|_{I_1}\in\ell_{\pp_1}[I_1] \ \text{ and } \ w|_{I_2}\in\ell_{\pp_2}[I_2]
\quad\iff\quad w\in\shuffle(w_1,w_2)
\]
and if these conditions hold then
\[(w_1,\Upsilon(\pp_1))\oj{w}(w_2,\Upsilon(\pp_2)) = (w,\Upsilon(\pp_1)*\Upsilon(\pp_2)) = 
(w,\Upsilon(\pp_1\x\pp_2))\in\omat^+\]
as desired.
\end{proof}

The Hopf monoid $\OMat^+=(\omat^+)^\natural$ (see \S\ref{sec:hmonoid-to-hclass}) is therefore defined by
\[\OMat^+[I]=\{w\otimes\Upsilon(\pp):\ w\otimes\pp\in\OIGP^+[I]\}.\]
Thus there is a surjective morphism of vector species $\tilde\Upsilon=\Id\otimes\Upsilon:\OIGP^+\to\OMat^+$.

\begin{proposition} \label{omatplus-from-hclass}
The map $\tilde\Upsilon:\OIGP^+\to\OMat^+$ is a surjective Hopf morphism.
\end{proposition}

\begin{proof}
It is necessary to show that $\Upsilon$ is compatible with shuffled join, initial restriction, and initial contraction.  Compatibility with join follows from the observation that the vertices of $\pp\times \qq$ are exactly concatenations of vertices of $\pp$ with vertices of~$\qq$, and compatibility with initial restriction is just the identity $\Upsilon(\pp|I) = \Upsilon(\pp)|I$ obtained in the proof of Theorem~\ref{omatplus-hclass}.

For compatibility with initial restriction, let $w\otimes\pp\in\OIGP^+$ and $I\in\Initial(w)$; we must show that $\Upsilon(\pp/I) = \Upsilon(\pp)/I$.  Let $\varphi$ be the $w$-lex-minimal facet of $\Upsilon(\pp|I)=\Upsilon(\pp)|I$. Then $\ee_\varphi$ is a vertex of $\pp|I$, and by \cite[Prop.~1.4.2]{AA} the vertices of $\pp/I$ are precisely
\[
\{\ee_{\cA}\in\mathbb{R}^{J\backslash I}:\ A\subseteq {J\backslash I},\ \ee_\varphi\x\ee_{\cA} \in \pp\}.
\]
But $\varphi$ is also the $w$-lex-minimal facet of $\Upsilon(\pp)|I$, so 
$\Upsilon(\pp/I)
= \{A\subseteq {J\backslash I}:\ \varphi\cup A \in \Upsilon(\pp)\}
= \link_{\Upsilon(\pp)}(\varphi)
=\Upsilon(\pp)/I
$
as desired.
\end{proof}

The relationships between the various Hopf monoids we have encountered are summarized in the following diagram.

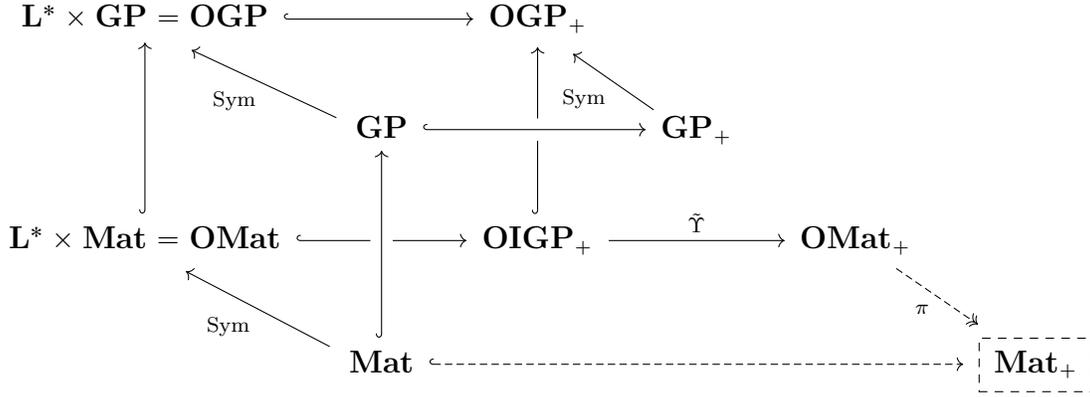
\newcommand\dboxed[1]{\tikz [baseline=(boxed word.base)] \node (boxed word) [draw, rectangle, dashed, line cap=round] {#1};}
\[
\begin{tikzcd}[column sep=small]
\bL^*\x\GP = \OGP \arrow[rr,hook] && \OGP^+ \arrow[from=dd,hook]\\
&\GP\arrow[rr,hook,crossing over] \arrow[ul,shift left,"\Symm"] && \GP^+ \arrow[ul,shift left,"\Symm"] \\
\bL^*\x\Mat = \OMat \arrow[rr,hook] \arrow[uu,hook] && \OIGP^+ \arrow[rr,"\tilde\Upsilon"] && \OMat^+ \arrow[dr,twoheadrightarrow,dashed,"\pi"'] \\
& \Mat \arrow[uu,hook,crossing over]  \arrow[rrrr,hook,crossing over,dashed] \arrow[ul,shift left,"\Symm"] &&&& \dboxed{$\Mat^+$} 
\end{tikzcd}
\]

The vector species $\Mat^+$ is defined as the image of $\OMat^+$ under the projection map $\pi:\bL^*\x\GP\to\GP$; that is, it is the span of all indicator complexes of 0/1-EGPs. By Corollary~\ref{Mat-universal}, the dashed arrows are not Hopf morphisms; indeed, it seems unlikely that $\Mat^+$ can be endowed with a Hopf monoid structure.

\section{Antipodes of facet-initial and shifted complexes} \label{sec:qantipode}

Throughout this section, we fix a natural ordered prefix-pure complex $(w,\Gamma,[n])$ of dimension~$r-1$.

Recall from \S\ref{sec:set-comps-and-gps} that a set composition $\cA=A_1|\cdots|A_k\compn[n]$ is natural iff its blocks are intervals, or equivalently if $\cA\refinedbyeq\cE$, where $\cE=1|2|\cdots|n$.

For any interval $J=[s,t]$, define the corresponding \defterm{interval minor} of $(w,\Gamma,[n])$ as $(w|_J,\Gamma(J),J)$, where $\Gamma(J)=\Gamma(s,t)=(\Gamma|L\sqcup J)/L$ and $L$ is the interval preceding~$J$. Moreover, for any natural set composition $\cA=A_1|\cdots|A_k\compn I$, define the \defterm{reassembly} of $\Gamma$ with respect to $\cA$ as $\Re_\cA(\Gamma)=\Gamma(A_1)*\cdots*\Gamma(A_k)$, so that
\begin{equation} \label{prodcoprod}
\mu_\cA(\Delta_\cA(w\otimes\Gamma)) = \sum_{u\in\shuffle(w|_{A_1},\dots,w|_{A_k})} (u\otimes\Re_\cA(\Gamma)) = \sum_{\substack{u\in\Sym_n:\\ \cD(w,u)\refinedbyeq \cA}} (u\otimes\Re_\cA(\Gamma)).
\end{equation}

Recall from~\eqref{muDeltaW} that $\mu_\cA(\Delta_\cA(w\otimes\Gamma))=0$ if $\cA$ is not natural, so
the Takeuchi formula~\eqref{Takeuchi} gives
\begin{align}
\anti(w\otimes\Gamma)
&= \sum_{\cA\refinedbyeq\comp(w)} (-1)^{|\cA|}\left. \sum_{\substack{u\in\Sym_n:\\ \cD(w,u)\refinedbyeq \cA}} u\otimes\Re_\cA(\Gamma) \right. \notag
\\
&= \sum_{u\in\Sym_n} \sum_{\substack{\cA\compn[n]:\\\cD(w,u)\refinedbyeq \cA\refinedbyeq\cE}} (-1)^{|\cA|}\left. u\otimes\Re_\cA(\Gamma) \right. \notag
\\
&= \sum_{u\in\Sym_n} \sum_{\substack{\text{reassemblies}\\ \Omega\text{ of } \Gamma}} \left(\sum_{\cA\in \CCC^\circ_\Omega\cap\EEE_{w,u}} (-1)^{|\cA|}\right) u\otimes\Omega \label{mason}
\end{align}
where $\EEE_{w,u}$ is as defined in~\eqref{define-EEE} and
\begin{equation} \label{coefficient-fan}
\CCC^\circ_\Omega=\{\cA\compn[n]:\ \cA\refinedbyeq\cE,\ \Re_\cA(\Gamma)=\Omega\}.
\end{equation}
The album $\EEE_{w,u}$, which does not depend on $\Gamma$, will appear later in our calculation of the antipode for $\OGP^+$ (\S\ref{sec:OGP-antipode}).

Further investigation of the antipode along these lines appears to require describing the album $\CCC^\circ_\Omega$, which may be quite complicated in general.  Nevertheless, we make the following conjecture:

\begin{conjecture} \label{conj:mult-free}
In every Hopf monoid arising from a Hopf class (equivalently, in $\uhopf^\natural$), the antipode is multiplicity-free.
\end{conjecture}
Equivalently, the ``Euler characteristic'' of the album $\CCC^\circ_\Omega\cap\EEE_{w,u}$ defined in~\eqref{coefficient-fan} is always 0 or $\pm1$.  We now present evidence in support of this conjecture, for a class of ordered complexes for which the albums $\CCC^\circ_\Omega$ can be described easily.

\subsection{Facet-initial complexes} \label{sec:antipode-FI}

\begin{definition}
Let $(w,\Gamma,[n])$ be an ordered prefix-pure complex, and let $r=\dim\Gamma+1$.  We say that $(w,\Gamma)$ is \defterm{facet-initial} if either $\Gamma=\{\0\}$, or if $[r]$ is a facet of $\Gamma$ (hence the lex-minimal facet).
\end{definition}

Adapting terminology from matroid theory, we say that a vertex of~$\Gamma$ is a \defterm{coloop} if it belongs to every facet, and a \defterm{loop} if it belongs to no facet (hence to no face); in addition, we say that $\Gamma$ is \defterm{primitive} if it has no loops or coloops.  The facet-initial property is (much) more general than shiftedness, and is preserved by restriction and contraction, hence by taking interval minors (though not by join).  The interval minors of $\Gamma$ are
\begin{equation} \label{shifted-interval-minor}
\Gamma(s,t) = \{\gamma\subseteq[s,t]:\ \gamma\cup[1,s-1]\in\Gamma\}.
\end{equation}
In particular, $\Gamma(s,t)$ is a simplex if $t\leq r$ and is empty if $s>r$.
Therefore, for any natural set composition $A=A_1|\cdots|A_k\compn I$ (i.e., one whose blocks are intervals), suppose that $A_j=[s,t]$ is the block of $A$ containing $r$ (so $1\leq s\leq r\leq t\leq n$).  Then the reassembly $\Re_{\cA}(\Gamma)$ that occurs in~\eqref{prodcoprod} is
\newcommand{\OldOmega}{\Gamma^*}
\[\OldOmega(s,t):=\langle[1,s-1]\rangle*\Gamma(s,t);\]

in particular $\Re_{\cA}(\Gamma)$ and the albums $\CCC^\circ_\Omega$ of~\eqref{coefficient-fan} depend only on the interval $[s,t]$.

\begin{theorem} \label{thm:antipode-FI}
Let $(w,\Gamma,[n])$ be a facet-initial ordered simplicial complex.  Then:
\begin{enumerate}
\item Its antipode is given by the formula
\begin{equation} \label{antipode-facet-initial}
\anti(w\otimes\Gamma) = \sum_{s=1}^r \sum_{t=r}^n (-1)^{n-1-t+s} \sum_{u\in\Sh(s,t)}  u\otimes\OldOmega(s,t)
\end{equation}
where $\Sh(s,t)=\shuffle([s-1,s-2,\dots,1],[s,\dots,t],[n,n-1,\dots,t+1])\subseteq\Sym_n$.
\item If in addition $\Gamma$ is shifted and has no loops or coloops, then the formula is cancellation-free.
\end{enumerate}
\end{theorem}

We will establish the formula immediately, but defer the proof of the second statement until we focus on shifted complexes in the next section.

\begin{proof}[Proof of Theorem~\ref{thm:antipode-FI}~(1)]
By the foregoing observations about facet-initial complexes, we can rewrite~\eqref{mason} as
\begin{equation} \label{dixon}
\anti(w\otimes\Gamma)
= \sum_{s=1}^r\sum_{t=r}^n \sum_{\substack{\cA\compn[n]\text{ natural}:\\ [s,t]\in \cA}} (-1)^{|\cA|} \sum_{\substack{u\in\Sym_n:\\ \cD(w,u)\refinedbyeq \cA}} u \otimes  \OldOmega(s,t).
\end{equation}
Observe that if $[s,t]\in \cA$ and $\cD(w,u)\refinedbyeq \cA$ then $u^{-1}(s)<\cdots<u^{-1}(t)$.  Moreover, for fixed~$s$ and~$t$ and a permutation $u$ obeying this last condition, any natural set composition~$\cA$ containing $[s,t]$ as a block satisfies $\cD(w,u)\refinedbyeq \cA$ if and only if $\cB\refinedbyeq \cA\refinedbyeq \cC$, where
\begin{align*}
\cB = \cB_{r,s,t} &= \Big([1,s-1]\,|\,[s,t]\,|\,[t+1,n]\Big) \join \cD(w,u), \\
\cC = \cC_{r,s,t} &= 1\,|\,\cdots\,|\,s-1\,|\,[s,t]\,|\,t+1\,|\,\cdots\,|\,n
\end{align*}
(here $\join$ denotes the join in the lattice of natural set compositions).  Therefore, we may rewrite the right-hand side of~\eqref{dixon} as
\begin{equation}  \label{scylla}
\sum_{s=1}^r\sum_{t=r}^n \sum_{\substack{u\in\Sym_n\\ u^{-1}(s)<\cdots<u^{-1}(t)}} \left(\sum_{\cB\refinedbyeq \cA\refinedbyeq \cC} (-1)^{|\cA|}\right) u\otimes\OldOmega(s,t).
\end{equation}
The parenthesized sum, over a Boolean interval, vanishes unless $\cB=\cC$.  In this case the sum equals $(-1)^{n-1-(t-s)}$, and since $1,\dots,s-1$ and $t+1,\dots,n$ are singleton parts in $\cD(w,u)$ it follows that
\begin{equation} \label{DUD-conditions}
\begin{aligned}
u^{-1}(1)&>\cdots>u^{-1}(s-1),\\
u^{-1}(s)&<\cdots<u^{-1}(t),\\
u^{-1}(t+1)&>\cdots>u^{-1}(n).
\end{aligned}
\end{equation}
(the second condition was noted previously).
These conditions, together with $u^{-1}(s)<\cdots<u^{-1}(t)$, say precisely that $u\in\Sh(s,t)$, yielding the desired formula.
\end{proof}

The formula of Theorem~\ref{thm:antipode-FI} is not cancellation-free in all cases, because there can exist $s,t,s',t'$ such that $\OldOmega(s,t)=\OldOmega(s',t')$ and $\Sh(s,t)\cap\Sh(s',t')$ is nonempty.  However, this possibility is limited, and in fact we can track the cancellation exactly.  Say that a permutation $u\in\Sym_n$ is \defterm{DUD} (for \defterm{down-up-down}) if it satisfies the inequalities
$u^{-1}(1)>\cdots>u^{-1}(S)$, $u^{-1}(S)<\cdots<u^{-1}(T)$,  $u^{-1}(T)>\cdots>u^{-1}(n)$, where $S,\dots,r,\dots,T$ is the maximal increasing subsequence of~$u$ containing~$r$ (in particular, $S\leq r\leq T$).
Observe that $u\in\Sh(s,t)$ if and only if it is DUD, with $s\in\{S,S+1\}$ and $t\in\{T,T-1\}$, or equivalently,
$s-1 \,\leq\, S \,\leq\, s \,\leq\, r \,\leq\, t \,\leq\, T \,\leq\, t+1$.
Thus we can regroup the formula of Theorem~\ref{thm:antipode-FI} by summing over permutations, always remembering that~$S$ and~$T$ depend on $u$:

\begin{equation} \label{antipode:ver2}
\begin{aligned}
\anti(w\otimes\Gamma)
&= \sum_{\substack{u\ \text{DUD}\\ S<r<T}} (-1)^{n-T+S-1} u\otimes\Big( \OldOmega(S,T) - \OldOmega(S,T-1) - \OldOmega(S+1,T) + \OldOmega(S+1,T-1) \Big)\\
&\qquad+ \sum_{\substack{u\ \text{DUD}\\ S<r=T}} (-1)^{n-r+S-1} u\otimes\Big( \OldOmega(S,r)  - \OldOmega(S+1,r)  \Big)\\
&\qquad+ \sum_{\substack{u\ \text{DUD}\\ S=r<T}} (-1)^{n-T+r-1} u\otimes\Big( \OldOmega(r,T) - \OldOmega(r,T-1) \Big)\\
&\qquad+ \sum_{\substack{u\ \text{DUD}\\ S=r=T}} (-1)^{n-T+S-1} u\otimes\OldOmega(S,T)
\end{aligned}
\end{equation}
By definition, $\OldOmega(s,r) = \sx{[1,r]}$ for every $s\leq r$,
so the second sum in~\eqref{antipode:ver2} vanishes.  Cancellation in the third sum is easy to track: $\OldOmega(r,T)$ is generated by the facets of $\Gamma$ of the form $[r-1]\cup\{x\}$ with $r\leq x\leq T$, so the third summand is nonzero if and only if $[r-1]\cup\{T\}\in\Gamma$.

We now consider cancellation in the first sum.  The relevant conditions are:
\medskip

\begin{tabular}{ll}
(a) $T$ is a loop in $\Gamma(S,T)$;    & (c)  $S$ is a coloop in $\Gamma(S,T)$;\\
(b) $T$ is a loop in $\Gamma(S+1,T)$;  & (d) $S$ is a coloop in $\Gamma(S,T-1)$.
\end{tabular}
\smallskip

The conditions for equality between each pair are indicated in the following diagram:
\begin{equation} \label{abcd-table}
\begin{tikzcd}[column sep=large, row sep=large]
\begin{array}{cc}\OldOmega(S,T)\\=\langle[1,S-1]\rangle*\Gamma(S,T)\end{array}
	\arrow[dd, dash, "\text{(a)}"] \arrow[r, dash, "\text{(c)}"] \arrow[ddr, dash, "\text{(a),(c)}"' near start] &
\begin{array}{cc}\OldOmega(S+1,T)\\=\langle[1,S]\rangle*\Gamma(S+1,T)\end{array}
	\arrow[dd, dash, "\text{(b)}"] \arrow[ddl, dash, "\text{(b),(d)}" near start]
\\ \\
\begin{array}{cc}\OldOmega(S,T-1)\\=\langle[1,S-1]\rangle*\Gamma(S,T-1)\end{array}
	 \arrow[r, dash, "\text{(d)}"] &
\begin{array}{cc}\OldOmega(S+1,T-1)\\=\langle[1,S]\rangle*\Gamma(S+1,T-1)\end{array}
\end{tikzcd}
\end{equation}
Note that (a)$\implies$(b) and (c)$\implies$(d).  Moreover, by the diagram above, it is not possible that exactly three of the conditions are true.  The remaining logical possibilities, and the resulting cancellation in the summand, are given by the following table.

\[\begin{array}{ccl}\hline
\textbf{True} & \textbf{False} & \textbf{Simplified form}\\ \hline
\text{a,b,c,d} & \text{---} & 0\\
\text{a,b} & \text{c,d} & 0\\
\text{c,d} & \text{a,b} & 0\\
\text{b,d} & \text{a,c} & \OldOmega(S,T) - \OldOmega(S,T-1) \text{ or } \OldOmega(S,T) - \OldOmega(S+1,T)\\
\text{b} & \text{a,c,d} & \OldOmega(S,T) - \OldOmega(S,T-1)\\
\text{d} & \text{a,b,c} & \OldOmega(S,T) - \OldOmega(S+1,T)\\
\text{---} & \text{a,b,c,d} & \OldOmega(S,T) - \OldOmega(S,T-1) - \OldOmega(S+1,T) + \OldOmega(S+1,T-1)\\ \hline
\end{array}
\]

Putting these observations together leads to the following formula (which is cancellation-free and multiplicity-free) for the antipode of a facet-initial complex:
\begin{equation} \label{antipode:ver4}
\begin{aligned}
\anti(w\otimes\Gamma)
&= \sum_{\substack{u\ \text{DUD}\\ S<r<T\\ \text{(a)\dots(d) false}}} (-1)^{n-T+S-1} u\otimes\Big( \OldOmega(S,T) - \OldOmega(S,T-1) - \OldOmega(S+1,T) + \OldOmega(S+1,T-1) \Big)\\
&\qquad+ \sum_{\substack{u\ \text{DUD}\\ S<r<T\\ \text{(b) true; (a,c) false}}} (-1)^{n-T+S-1} u\otimes\big( \OldOmega(S,T) - \OldOmega(S,T-1) \big)\\
&\qquad+ \sum_{\substack{u\ \text{DUD}\\ S<r<T\\ \text{(d) true; (a,b,c) false}}} (-1)^{n-T+S-1} u\otimes\big( \OldOmega(S,T) - \OldOmega(S+1,T)\big)\\
&\qquad+ \sum_{\substack{u\ \text{DUD}\\ S=r<T\\ [r-1]\cup\{T\}\in\Gamma}} (-1)^{n-T+r-1} u\otimes\big( \OldOmega(r,T) - \OldOmega(r,T-1) \big)\\
&\qquad+ \sum_{\substack{u\ \text{DUD}\\ S=r=T}} (-1)^{n-1} u\otimes\langle[1,r]\rangle.
\end{aligned}
\end{equation}

Here we have combined the fourth and fifth cases in the table.  Alternatively, it is possible to combine the fourth and sixth cases and write a similar formula, which we omit.

\subsection{Shifted complexes}
We now consider the case of a pure ordered complex $(w,\Gamma,[n])$ that is not merely facet-initial but in fact \defterm{shifted}; i.e., its facets form an order ideal in Gale order (see Example~\ref{ex:shifted}).  We write $\Langle\varphi_1,\dots,\varphi_m\Rangle_I$ for the shifted matroid on vertex set~$I$ whose facets are the Gale order ideal generated by the $\varphi_i$: for example, $\Langle 14,23\Rangle_{[4]}=\langle 12,13,14,23\rangle$.  Note that the coloops of~$\Gamma$ form an initial segment $[1,a]$ and its loops form a final segment $[z,n]$, where $a\leq r<z$.  The definitions of loop and coloop imply that
\begin{subequations}
\begin{equation} \label{shifted-coloops}
[1,r+1]\sm\{x\}\in\Gamma \quad\iff\quad x>a
\end{equation}
and
\begin{equation} \label{shifted-loops}
[1,r-1]\sm\{x\}\in\Gamma \quad\iff\quad x<z.
\end{equation}
\end{subequations}
In particular, $\Gamma$ is coloop-free if and only if $[2,r+1]$ is a facet, and is loop-free if and only if $[1,r-1]\cup\{n\}$ is a facet.  (If $\Gamma$ has no coloops we may set $a=0$, and if it has no loops then $z=n+1$.)

\begin{lemma}\label{shifted-interval-loops}
Let $(w,\Gamma)$ be a pure shifted complex of dimension $r-1$ on vertex set $[1,n]$ with $w=e$ the natural ordering.
Suppose that the coloops of $\Gamma$ are $[1,a]$ and the loops are $[z,n]$.  (Note that $a\leq r<z$.)
Let $[s,t]\subseteq[1,n]$ be an interval such that $s\leq r\leq t$.  Then:
\begin{enumerate}
\item If $t=r$, then $\Gamma(s,t)=\langle[s,t]\rangle$ and so $\OldOmega(s,t)=\langle[1,r]\rangle$.
\item If $t>r$, then the coloops of $\Gamma(s,t)$ are $[1,a]\cap[s,t]$ and the loops are $[z,n]\cap[s,t]$.
Thus $\Gamma(s,t)=\langle[1,a]\cap[s,t]\rangle*\Gamma(a+1,\min(t,z-1))$ and so
\begin{align*}
\OldOmega(s,t) &= \langle[1,s-1]\rangle*\langle[1,a]\cap[s,t]\rangle*\Gamma(a+1,\min(t,z-1))\\
&= \langle[1,\max(s-1,a)]\rangle*\Gamma(a+1,\min(t,z-1))
\end{align*}
and moreover $\Gamma(a+1,\min(t,z-1))$ is primitive.
\end{enumerate}
\end{lemma}
\begin{proof}
The first assertion is immediate from~\eqref{shifted-interval-minor}.  Henceforth, suppose that $t>r$.  Let $x\in[s,t]$.  Then by~\eqref{shifted-interval-minor} and~\eqref{shifted-coloops}
\[[s,r+1]\sm\{x\}\in\Gamma(s,t) \quad\iff\quad [1,r+1]\sm\{x\}\in\Gamma \quad\iff\quad x>a\]
and by~\eqref{shifted-interval-minor} and~\eqref{shifted-loops}
\[[s,r-1]\cup\{x\}\in\Gamma(s,t) \quad\iff\quad [1,r-1]\sm\{x\}\in\Gamma \quad\iff\quad x<z\]
from which the statement about loops and coloops follows, and the rest is calculation.
\end{proof}

In the third case ($s\leq r<t$), note that if $\Gamma$ is primitive then so is $\Gamma(s,t)$.  Moreover, $\Gamma(s,t)$ cannot be a simplex (because $a<t$) or empty (because $s<z$).

\begin{corollary}\label{shifted-equal}
$\OldOmega(s,t)=\OldOmega(s',t')$ if and only if (i) $t=t'=r$; or (ii) $t,t'>r$, $\min(t,z-1)=\min(t',z-1)$, and $\max(s-1,a)=\max(s'-1,a)$.
Equivalently: (i) $t=t'=r$; or (ii) either $r<t=t'<z$ or $t,t'\geq z$, and either $s,s'\leq a+1$ or $s=s'>a+1$.
\end{corollary}

We can now revisit the cancellation-free formula~\eqref{antipode:ver4}.  The conditions (a), \dots, (d) of~\eqref{abcd-table} now become
\begin{enumerate}[label=(\alph*)]
\item $T$ is a loop in $\Gamma(S,T)$ $\iff$ $T\geq z$.
\item $T$ is a loop in $\Gamma(S+1,T)$ $\iff$ $T\geq z$.
\item $S$ is a coloop in $\Gamma(S,T)$ $\iff$ $T=r$, or $T>r$ and $S\leq a$.
\item $S$ is a coloop in $\Gamma(S,T-1)$ $\iff$ $T=r$, or $T=r+1$, or $T>r+1$ and $S\leq a$.
\end{enumerate}
In the first sum, to say that (a)\dots(d) all fail is to say that $r+1<T<z$ and $S>a$.
The second sum disappears because conditions (a) and (b) are equivalent for shifted complexes.
In the third sum, to say that (a),~(b),~(c) all fail but (d) holds is to say that $T=r+1$ and $S>a$.
In the fourth sum, the condition $[r-1]\cup\{T\}$ becomes $T<z$ by~\eqref{shifted-loops}.  
Thus we can simplify~\eqref{antipode:ver4} slightly to the cancellation-free formula
\begin{equation} \label{antipode:shift:ver4}
\begin{aligned}
\anti(w\otimes\Gamma)
&=\!\!\! \sum_{\substack{u\ \text{DUD}\\ a<S<r\\ r+1<T<z}} (-1)^{n-T+S-1} u\otimes\big( \OldOmega(S,T) - \OldOmega(S,T-1) - \OldOmega(S+1,T) + \OldOmega(S+1,T-1) \big)\\
&\qquad+ \sum_{\substack{u\ \text{DUD}\\ a<S<r\\ T=r+1<z}} (-1)^{n-r+S} u\otimes\big( \OldOmega(S,r+1) - \OldOmega(S+1,r+1)\big)\\
&\qquad+ \sum_{\substack{u\ \text{DUD}\\ S=r<T<z}} (-1)^{n-T+r-1} u\otimes\big( \OldOmega(r,T) - \OldOmega(r,T-1) \big)\\
&\qquad+ \sum_{\substack{u\ \text{DUD}\\ S=r=T}} (-1)^{n-1} u\otimes\langle[1,r]\rangle.
\end{aligned}
\end{equation}

We can now complete the proof of Theorem~\ref{thm:antipode-FI}.

\begin{proof}[Proof of Theorem~\ref{thm:antipode-FI}~(2)]
Let $(w,\Gamma,I)$ be a shifted complex with no loops or coloops.
By Lemma \ref{shifted-interval-loops}, all complexes $\Gamma(s,t)$ with $s\leq r\leq t$ are also loopless and coloopless. For each complex $\Gamma^*(s,t)$ we can recover $s$ and $t$ from the coloops and the loops respectively.  It follows that formula~\eqref{antipode-facet-initial} is cancellation-free.
\end{proof}

\begin{remark}
The simplicial complexes $\OldOmega(s,t)$ are shifted with respect to $e$ (by virtue of being minors of the shifted complex $\Gamma$) but not necessarily with respect to the permutations $u$ with which they are paired in the formula of Theorem~\ref{thm:antipode-FI}.  (We should not expect them to be so, since ordered shifted complexes do not form a Hopf class or a Hopf monoid; see Example~\ref{ex:shifted}.)
\end{remark}

Observe that the indexing of the terms in the antipode formula given by~\eqref{antipode:shift:ver4} depends only on the parameters $r,n,a,z$; this will be useful shortly. 

\subsection{Shifted (Schubert) matroids} \label{sec:shifted-matroids}

Klivans~\cite[Thm.~5.4.1]{Klivans-thesis} proved that a shifted complex is a matroid independence complex if and only if it is a \textit{principal} ideal in Gale order.  These complexes also appear in the literature under the name \textit{Schubert matroids}.

Let $\varphi=\{a_1<\cdots<a_r\}\subseteq[n]$.
We can represent the shifted matroid $\Langle\varphi\Rangle_{[n]}$ as the Ferrers diagram of the partition $(a_1-1,\dots,a_r-r)$ (here it is convenient to write partitions in weakly increasing order).  This correspondence gives an isomorphism between Gale order on $r$-subsets of $[n]$ and the interval in Young's lattice from the empty partition to the rectangle $r\x(n-r)$ with $r$ rows and $n-r$ columns.  In particular, Gale order is a (distributive) lattice whose join and meet correspond respectively to union and intersection of Ferrers diagrams.  If we label rows $1,\dots,r$ south to  north and columns $r+1,\dots,n$ west to east, then empty rows at the south are coloops and empty columns on the east are loops.  For $s\leq r<t$, the complexes $\Gamma(s,t)$ and $\OldOmega(s,t)$ are obtained by erasing all rows south of~$s$ and all columns east of~$t$ and regarding the result as a subdiagram of $[s,r]\x[r+1,t]$ or of $[1,r]\x[r+1,n]$ respectively.  This observation is the pictorial version of the formula \begin{equation} \label{shifted-matroid-interval-minor}
\Langle\varphi\Rangle_{[n]}(s,t)=\Langle\varphi'\wedge[t-r+s,t]\Rangle_{[s,t]},
\end{equation}
where $\varphi'=\{a_s,\dots,a_r\}$ and $\wedge$ denotes meet in Gale order (we omit the routine verification).  Thus the class of shifted matroids is closed under taking interval minors.

For example, the shifted matroid $\Langle23589\Rangle_{[11]}$ (with $r=5$) corresponds to the Ferrers diagram of the partition $(4,4,2,1,1)=(9-5,8-4,5-3,3-2,2-1)$, considered as a subset of a $r\x(n-r)=5\x6$ rectangle (see Figure~\ref{fig:shifted-matroid-Ferrers}).

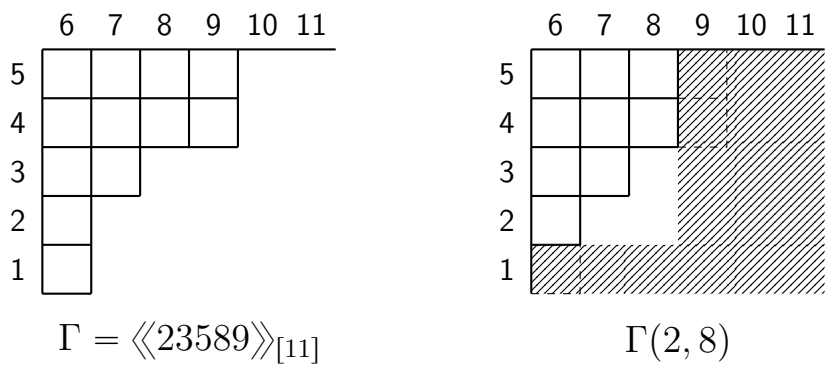
\begin{figure}[ht]
\begin{center}
\begin{tikzpicture}[scale=0.5]
\node at (3,-1) {\large$\Gamma=\Langle23589\Rangle_{[11]}$};
\foreach \y in {1,...,5} \node at (-.5,\y-.5) {\sf\y};
\foreach \x in {6,...,11} \node at (\x-5.5,5.5) {\sf\x};
\draw[thick](0,0)--(1,0);	\draw[thick](0,5)--(0,0);
\draw[thick](0,1)--(1,1);	\draw[thick](1,5)--(1,0);
\draw[thick](0,2)--(2,2);	\draw[thick](2,5)--(2,2);
\draw[thick](0,3)--(4,3);	\draw[thick](3,5)--(3,3);
\draw[thick](0,4)--(4,4);	\draw[thick](4,5)--(4,3);
\draw[thick](0,5)--(6,5);
\begin{scope}[shift={(10,0)}]
\node at (3,-1) {\large$\Gamma(2,8)$};
\fill[pattern = north east lines] (0,1)--(3,1)--(3,5)--(6,5)--(6,0)--(0,0)--cycle;
\draw[thin,dashed] (0,0)--(1,0)--(1,1) (3,3)--(4,3)--(4,5) (3,4)--(4,4);
\foreach \y in {1,...,5} \node at (-.5,\y-.5) {\sf\y};
\foreach \x in {6,...,11} \node at (\x-5.5,5.5) {\sf\x};
					\draw[thick](0,5)--(0,0);
\draw[thick](0,1)--(1,1);	\draw[thick](1,5)--(1,1);
\draw[thick](0,2)--(2,2);	\draw[thick](2,5)--(2,2);
\draw[thick](0,3)--(3,3);	\draw[thick](3,5)--(3,3);
\draw[thick](0,4)--(3,4);
\draw[thick](0,5)--(6,5);
\end{scope}
\end{tikzpicture}
\caption{A shifted matroid and an interval minor, represented as Ferrers diagrams.\label{fig:shifted-matroid-Ferrers}}
\end{center}
\end{figure}

Meanwhile, let us interpret reassemblies of the shifted matroid $\Gamma$ geometrically.  Let $\pp$ be the matroid polytope of~$\Gamma$ (equivalently, its indicator polytope; see \S\ref{sec:scs}).   For every natural set composition $A=A_1|\cdots|A_k\compn I$,
iterating \cite[Prop.~1.4.2]{AA} gives
\[\pp_{\cA}=\pp_{A_1}\x\cdots\x\pp_{A_k},\]
where $\pp_{A_i}$ is the matroid polytope of the interval minor $\Gamma(A_i)$. Thus the indicator complex of $\pp_{\cA}$ is precisely the reassembly $\Re_{\cA}(\Gamma)$.

Let $\Phi\subseteq\Gamma$ be shifted complexes on $[n] $ with the same parameters $r,n,a,z$ of the same dimension, and let $1\leq s\leq t\leq n$.  Thus, as mentioned above,  the terms in the expressions for $\anti(w\otimes\Phi)$ and $\anti(w\otimes\Gamma)$ given by~\eqref{antipode:shift:ver4} are indexed identically.  Moreover, an easy calculation shows that
\begin{equation} \label{gamma-phi}
\Gamma(s,t)\cap\Phi = \Phi(s,t) \qquad\text{and so}\qquad\Gamma^*(s,t)\cap\Phi = \Phi^*(s,t).
\end{equation}
As a result, we can interpret the cancellation-free antipode formula~\eqref{antipode:shift:ver4} of any shifted complex $\Phi$ quasi-geometrically by taking $\Gamma$ to be any shifted matroid containing $\Phi$, so that the terms
$u\otimes\Phi^*(s,t)$ in~\eqref{antipode:shift:ver4} correspond bijectively to the faces of the base polytope of~$\Gamma$.  For example, $\Gamma$ could be taken to be the \defterm{matroid hull} of $\Phi$ (the unique smallest shifted matroid containing $\Phi$, generated by the Gale join of all its facets).

\section{The antipode in \texorpdfstring{$\OGP^+$}{OGP+}} \label{sec:antipode}

The goal of this section is to establish an antipode formula for the Hopf monoid $\OGP^+$, and thus for its submonoids $\OGP$ and $\OMat$.  The argument is modeled after Aguiar and Ardila's topological calculation of the antipode in $\GP^+$
\cite[Thm.~1.6.1]{AA}.  It works in general for $\OGP^+$. The computations rely heavily on the normal fan of the polyhedron and the fact that it is (topologically) closed even for unbounded polyhedra. This implies that the closure of the normal cone of a face still makes sense in constructing objects such as $\EE$ and $\FF$ below.

\subsection{Scrope complexes} \label{sec:scrope}

We begin by describing a class of simplicial complexes that will play a key role in the computation of the antipode on $\OGP^+$.

\begin{definition} \label{def:scrope}
A \defterm{Scrope complex}\footnote{Named after the winner of an important 1389 case in English heraldry law, concerning a coat of arms that looks a lot like the dots-and-stars diagram.} is a simplicial complex on vertices~$[k-1]$ that is either a simplex, or is generated by faces of the form $[k-1]\sm[x,y-1]=[1,x-1]\cup[y,k-1]$, where $1\leq x<y\leq k$.
If $\zz=((x_1,y_1),\dots,(x_r,y_r))$ is a list of ordered pairs of integers in~$[k]$ with $x_i<y_i$ for each $i$, we set $\varphi_i=[k-1]\sm[x_i,y_i-1]$ for $1\leq i\leq r$ and define
\[\Scr(k,\zz) = \left\langle\varphi_1,\dots,\varphi_r\right\rangle.\]
\end{definition}

The facets of a Scrope complex correspond to the intervals $[x,y-1]$ that are minimal with respect to inclusion. By removing redundant generators, we may assume that it is either the full simplex on $[k-1]$, or can be written as $\Scr(k,\zz)$ where $1\leq x_1<\dots<x_r<k$; $1<y_1<\dots<y_r\leq k$; and $x_i<y_i$ for each $i$.

To justify the above notational choices, we will be considering Scrope complexes whose vertices correspond to the $k-1$ separators in a natural set composition of $\{1,2,\dots,k\}$.  Thus $[k-1]\sm[x,y-1]$ is the set of separators in the set composition whose only non-singleton block is $\{x,x+1,\dots,y-1,y\}$, namely
\[1\,\big|\,2\,\big|\,\cdots\,\big|\,x\ x+1\ \cdots\ y-1\ y\,\big|\,y+1\,\big|\,\cdots\,\big|\,k.\]

A Scrope complex can be recognized by its facet-vertex incidence matrix, which we will represent as a $r\x(k-1)$ table whose $(i,j)$ entry is $\star$ or $\cdot$ according as $j\in\varphi_i$ or $j\not\in\varphi_i$.  Thus each row consists of a (possibly empty) sequence of dots sandwiched between two (possibly empty) sequences of stars.
For example, if $k=7$ and $\zz=((1,3),(2,4),(3,6),(4,7))$ then $\Scr(n,\zz)=\langle 3456,1456,126,123\rangle$ is represented by the following diagram:
\[\begin{array}{cccccc}
\cdot & \cdot & \star & \star & \star & \star\\
\star & \cdot & \cdot & \star &\star &\star\\
\star & \star & \cdot & \cdot & \cdot & \star\\
\star & \star & \star & \cdot & \cdot & \cdot
\end{array}\]
It is easy to see from this description that the class of Scrope complexes is stable under taking induced subcomplexes.

\begin{proposition} \label{scrope-homotopy}
Every nontrivial Scrope complex is either contractible or a homotopy sphere.
\end{proposition}
\begin{proof}
Let $\Gamma=\left\langle\varphi_1,\dots,\varphi_r\right\rangle$ be a Scrope complex, labeled as above.  One trivial case is that $r=1$, $x_1=1$, and $y_1=k$, which we allow for inductive purposes.  In this case $\Gamma$ is the trivial complex, which we regard as the $(-1)$-sphere.  Otherwise, if $r=1$, then $\Gamma$ is a simplex, hence contractible.  Also, if $y_r<k$, then each facet contains vertex $k-1$, so $\Gamma$ is again contractible.  Therefore, suppose that $r>1$ and $y_r=k$, so that $\varphi_r=[1,x_r-1]$.

Let $\hat{\Gamma}=\left\langle \varphi_1,\dots,\varphi_{r-1}\right\rangle$.  This complex is certainly a cone with apex $k-1$, hence contractible.
Now $\Gamma$ is the union of the contractible complexes $\hat{\Gamma}$ and $\left\langle \varphi_r\right\rangle$, attached along their intersection $\Gamma'$, namely
\begin{align*}
\left\langle \varphi_r\right\rangle\cap\hat{\Gamma} &= \left\langle \varphi_r\cap \varphi_i:\ 1\leq i\leq r-1\right\rangle\\
&= \left\langle [1,x_r-1]\cap\big([1,x_i-1]\cup[y_i,n-1]\big):\ 1\leq i\leq r-1\right\rangle\\
&= \left\langle [1,x_i-1]\cup[y_i,\min(x_r-1,n-1)]:\ 1\leq i\leq r-1\right\rangle
\end{align*}
which is evidently itself a (possibly trivial) Scrope complex.  By induction, either $\Gamma'$ is contractible, which implies that $\Gamma$ is contractible as well, or else $\Gamma'\htop\Ss^q$ for some $q$, which implies $\Gamma\htop\Ss^{q+1}$.
\end{proof}

\begin{corollary} \label{Euler-SG}
The reduced Euler characteristic of every Scrope complex is 0, 1, or $-1$.
\end{corollary}

As a corollary of the proof,  the reduced Euler characteristic of a Scrope complex can be computed recursively, with computation time linear in the number of generators $r=|\zz|$.

\subsection{The antipode calculation} \label{sec:OGP-antipode}

Throughout this section, $I$ will denote a set of size~$n$ and $w,u$ linear orders on~$I$.  Let~$\cW=\comp(w)$ be the maximal set composition corresponding to~$w$, and let $\cD=\cD(w,u)$ be the $u$-descent set composition of~$w$ (Definition~\ref{defn:descent-comp}).  Define fans and albums as follows: 
\begin{align*}
\EE=\EE_{w,u} &= \{\sigma_{\cA}\in\ov{\CC_w}:\ \cD\refinedbyeq \cA\}, &
\EEE=\EEE_{w,u} &= \{\cA:\ \sigma_{\cA}\in\EE\}=\{\cA:\ \cD\refinedbyeq \cA\refinedbyeq \cW\}, \\
\FF=\FF_{w,u} &=  \{\sigma_{\cA}\in\ov{\CC_w}:\ \cD\notrefinedbyeq \cA\}, &
\FFF=\FFF_{w,u} &= \{\cA:\ \sigma_{\cA}\in\FF\}=\{\cA:\ \cD\notrefinedbyeq \cA,\ \cA\refinedbyeq \cW\}
\end{align*}
(we repeat the definition of~$\EEE$ from~\eqref{define-EEE}).

The topological realization of the Boolean algebra $\CCC_{\cW}$ is the $(n-2)$-simplex $\Simplex{\CC}_w=\Sigma^{n-2}\cap\CC_w=\langle\Simplex{\sigma}_w\rangle$.  The descent set composition $\cD=\cD(w,u)$ corresponds to the face $\delta=\Simplex{\CC}_{\cD}$ of this simplex, and $\FFF$ corresponds to the closed subcomplex $\Simplex{\FF}_{w,u}\subseteq\langle\Simplex{\sigma}_w\rangle$ consisting of faces not containing $\delta$.  In particular,
\[\Simplex{\FF}_{w,u}\isom\begin{cases}
\0 & \text{ if } |\cD|=1 \qquad\iff\quad u=w,\\
\Bb^{n-3} & \text{ if } 1<|\cD|<n,\\
\Ss^{n-3} & \text{ if } |\cD|=n \qquad\iff\quad u=w^\rev.
\end{cases}\]

In fact we can say more about the structure of $\Simplex{\FF}$: it is obtained by coning the boundary of the simplex $\langle\delta\rangle$ successively with each vertex in $\sigma_w\sm\delta$.  It is worth mentioning that if $\Simplex{\FF}$ is nonempty then it is a very special kind of shellable complex: a pure full-dimensional subcomplex of a simplex boundary.  In particular, when $\Simplex{\FF}\isom\Bb^{n-3}$, then the set of star points of~$\FF$ is the union of the cones $\CC^\circ_{\cA}$ for all set compositions~$\cA$ such that $\cA\leq \cW$ and $\sigma_{\cA}\cap\delta=\0$.

With this setup in hand, we now move to the main calculation of this section.  Let $w\otimes\pp$ be a basis element of $\OGP^+[I]$.  We start with the Takeuchi formula:
\begin{align}
\anti(w\otimes \pp) &= \sum_{\cA=(A_1,\dots,A_k)\compn I} (-1)^k \mu_{\cA}(\Delta_{\cA}(w\otimes \pp))\notag 
&= \sum_{\cA\compn I} (-1)^{|\cA|} \mu_{\cA}(\Delta_{\cA}(w)) \otimes  \mu_{\cA}(\Delta_{\cA}(\pp))\notag\\
&= \sum_{\substack{\cA\compn I:\\ \cA\leq \cW}}\ \ \sum_{\substack{u\in\bl_\pp[I]:\\ \cD(w,u)\refinedbyeq \cA\refinedbyeq\cW}} (-1)^{|\cA|} u\otimes \pp_{\cA} & \text{(by~\eqref{muDeltaW:2} and \eqref{eq:combination})}\notag\\
&= \sum_{u\in \bl_\pp[I]} \ \ \sum_{\substack{\cA\compn I:\\ \cD(w,u)\leq \cA\leq \cW}} (-1)^{|\cA|} u\otimes \pp_{\cA}\notag
&= \sum_{u\in \bl_\pp[I]} \sum_{\qq\in\Faces(\pp)} \left( \sum_{\cA\in\EEE:\ \pp_{\cA}=\qq}(-1)^{|\cA|} \right) u\otimes \qq\notag
\end{align}
where $\Faces(\pp)$ denotes the set of faces of~$\pp$.
Thus we need to calculate the coefficient
\begin{equation}\label{antipode-coeff}
a^{w,\pp}_{u,\qq}
=\sum_{\substack{\cA\in\EEE:\ \pp_{\cA}=\qq}}\!\!\! (-1)^{|\cA|}
~=~ \sum_{\cA\in\EEE\cap\CCC^\circ_Q} \!\!\! (-1)^{|\cA|}
~=~ \sum_{\sigma\in\EE\cap\CC^\circ_\qq} \!\!\! (-1)^{\dim \sigma}
\end{equation}
where $Q$ is the normal preposet of~$\qq$.

As a consequence of this partial calculation, we can identify two useful necessary conditions on nonzero terms in the antipode (expressed in equivalent geometric and combinatorial forms).

\begin{proposition}[Locality of antipode in $\OGP^+$]\label{locality}
Suppose that $a^{w,\pp}_{u,\qq}\neq0$.  Then
\begin{subequations}
\begin{align}
\CC_w\cap\CC^\circ_\qq\neq\0 &\qquad\text{or equivalently}\qquad
\CCC_{\cW}\cap\CCC^\circ_Q\neq\0; \qquad\text{and}\label{locality:1}\tag{$\star$}\\
\sigma_{\cD}\in\CC_\qq &\qquad\text{or equivalently}\qquad
\cD\in\CCC_Q.\label{locality:2}\tag{$\star\star$}
\end{align}
\end{subequations}
\end{proposition}

\begin{proof}
First, $\EEE\subseteq\CCC_{\cW}$, so if~\eqref{locality:1} fails then the sum in~\eqref{antipode-coeff} is empty.  (In fact, ($\star$) follows from the third line of the calculation of $\anti(w\otimes\pp)$ above.)
If~\eqref{locality:2} fails, then $\EEE\cap\CCC_\qq=\0$ (since $\CCC_\qq$ is closed under coarsening), whence $\EEE\cap\CCC^\circ_\qq=\0$ and the sum in~\eqref{antipode-coeff} is again empty.
\end{proof}

Recall from \S\ref{sec:set-comps-and-gps} that a preposet $Q$ with ground set $S$ is \defterm{$w$-natural} with respect to an ordering $w$ on $S$ if, for all $x,y\in S$ with $w(x)>w(y)$ and $x\prec_Qy$, we have $x\equiv_Qy$.

\begin{corollary}\label{cor:naturality}
If $a^{w,\pp}_{u,\qq}\neq0$, then $Q$ is $w$-natural.
\end{corollary}
\begin{proof}
Let $x\prec_Qy$ be a relation and $\cA$ a set composition such that $\cA\in\CCC_{\cW}\cap\CCC^\circ_Q$. By Lemma~\ref{lem:comb-dictionary} $x,y$ are in different blocks of $\cA$, and by $w$-naturality of $\cA$ we have $w^{-1}(x)<w^{-1}(y)$.  Therefore $Q$ is $w$-natural.
\end{proof}

Proposition~\ref{locality} says that the support of $\anti^{\OGP}(w\otimes\pp)$ is local with respect to the braid cone $\sigma_w$. Condition ~\eqref{locality:1} is stronger than the property that the face $\qq$ associated to the preposet $Q$ must contain the vertex $\pp_w$: A face can contain the vertex $\pp_w$ without Condition ~\eqref{locality:1}, as in Example \ref{ex:locality}.

\begin{remark} \label{why-not-L:2}
By contrast, the antipode in $\bL\x\GP$ (as opposed to $\bL^*\x\GP$) is not local.  For example, let $\pp\subset\Rr^n$ be the standard permutohedron, whose normal fan is the braid fan $\BB_n$. In particular, if $\cA\in\Comp(n)$, then $\mu_{\cA}\circ \Delta_{\cA}(\pp) =\pp_{\cA}$ is the unique face of $\pp$ with normal cone $\sigma_{\cA}\in\BB_n$.  Thus
\[\anti^{\bL\x\GP}(w\otimes\pp) = \sum_{\cA\compn[n]} \mu^{\bL}_{\cA}(\Delta^{\bL}_{\cA}(w)) \otimes \pp_{\cA}.\]
Here each term in the sum is an element of the canonical basis of $\bL\times \GP[I]$ and the second part of the tensor is different for all terms. Thus \emph{all} faces of $\pp$ appear in the antipode, and we see that the algebraic structure of $\bL\x\GP$ is not local.
\end{remark}

\begin{example}\label{ex:locality}
Let $w=1234\in\Sym_4$, and let $\pp$ be the hypersimplex $\Delta_{2,4}$ (see \S\ref{sec:hypersimplices}), which is an octahedron with vertices $0011, 0101, 0110, 1001, 1010, 1100\in\Rr^4$ (abbreviating, e.g., $0011=(0,0,1,1)$).  By the characterization of faces of hypersimplices (Prop.~\ref{hypersimplex-faces}), the only faces of~$\pp$ that occur in the support of $\anti(w\otimes\pp)$ are as indicated in Figure~\ref{fig:localex}.  (Their names $\qq(A,B)$ will be explained in \S\ref{sec:hypersimplices}.)  The symbol $\conv$ means convex hull.

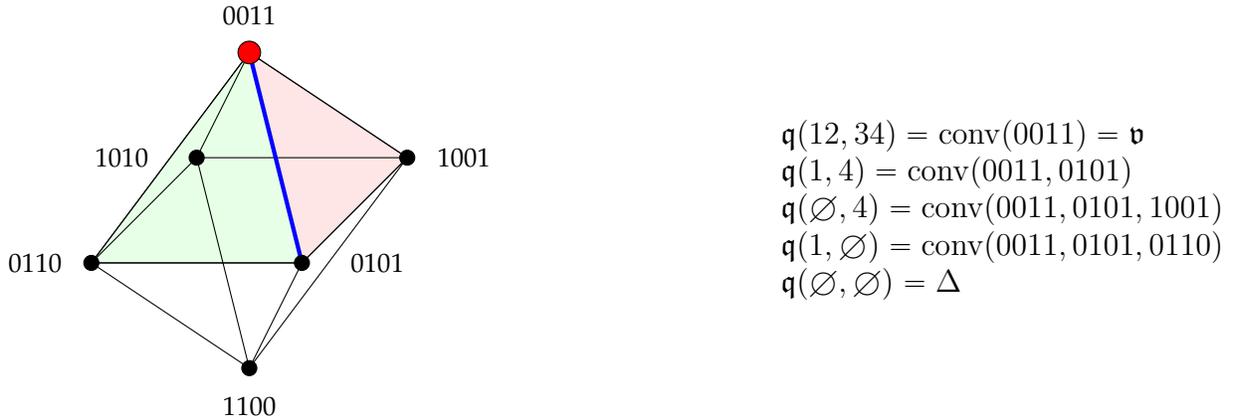
\begin{figure}[ht]
\begin{center}
\begin{tikzpicture}
\newcommand{\scl}{0.5}
\newcommand{\xooii}{0*\scl}	\newcommand{\yooii}{6*\scl}
\newcommand{\xiooi}{3*\scl}	\newcommand{\yiooi}{4*\scl}
\newcommand{\xioio}{-1*\scl}	\newcommand{\yioio}{4*\scl}
\newcommand{\xoioi}{1*\scl}	\newcommand{\yoioi}{2*\scl}
\newcommand{\xoiio}{-3*\scl}	\newcommand{\yoiio}{2*\scl}
\newcommand{\xiioo}{0*\scl}	\newcommand{\yiioo}{0*\scl}
\coordinate (iioo) at (\xiioo,\yiioo);
\coordinate (ioio) at (\xioio,\yioio);
\coordinate (iooi) at (\xiooi,\yiooi);
\coordinate (oiio) at (\xoiio,\yoiio);
\coordinate (oioi) at (\xoioi,\yoioi);
\coordinate (ooii) at (\xooii,\yooii);

\draw[fill=red!10!white] (ooii)--(oioi)--(iooi)--cycle;
\draw[fill=green!10!white] (ooii)--(oioi)--(oiio)--cycle;

\draw (iioo)--(ioio)--(ooii)--(oioi)--cycle;
\draw (iioo)--(oiio)--(ooii)--(iooi)--cycle;
\draw (iooi)--(ioio)--(oiio)--(oioi)--cycle;
\draw[ultra thick,blue] (ooii)--(oioi);

\draw[fill=black] (\xiioo,\yiioo) circle(.1); \node at (\xiioo,\yiioo-.5) {\footnotesize1100};
\draw[fill=black] (\xioio,\yioio) circle(.1); \node at (\xioio-1,\yioio) {\footnotesize1010};
\draw[fill=black] (\xiooi,\yiooi) circle(.1); \node at (\xiooi+.75,\yiooi) {\footnotesize1001}; 
\draw[fill=black] (\xoiio,\yoiio) circle(.1); \node at (\xoiio-.75,\yoiio) {\footnotesize0110};
\draw[fill=black] (\xoioi,\yoioi) circle(.1); \node at (\xoioi+1,\yoioi) {\footnotesize0101};
\draw[fill=red] (\xooii,\yooii) circle(.15); \node at (\xooii,\yooii+.5) {\footnotesize0011};

\node at (10,3*\scl) {$\displaystyle\begin{array}{l}
\qq(12,34) = \conv(0011)\\
\qq(1,4) = \conv(0011,0101)\\
\qq(\0,4)= \conv(0011,0101,1001)\\
\qq(1,\0)= \conv(0011,0101,0110)\\
\qq(\0,\0) = \Delta
\end{array}$};
\end{tikzpicture}
\end{center}
    \caption{Locality for the hypersimplex $\Delta_{2,4}$.}
    \label{fig:localex}
\end{figure}

\end{example}

For the purpose of calculating $a^{w,\pp}_{u,\qq}$, we assume from now on that conditions~($\star$) and~($\star\star$) hold for the pair $u,\qq$.  We start by rewriting $a^{w,\pp}_{u,\qq}$ as a sum of ``reduced Euler characteristics'' over fans:
\[a^{w,\pp}_{u,\qq}=\chi_1-\chi_2-\chi_3,\]
where
\begin{alignat*}{3}
\chi_1 &= 
\sum_{\sigma\in\CC_w\cap\CC_\qq} (-1)^{\dim \sigma} ~&&=~
\sum_{\cA\in\CCC_{\cW}\cap\CCC_Q} (-1)^{|\cA|},\\
\chi_2 &= 
\sum_{\sigma\in\FF\cap\CC_\qq} (-1)^{\dim \sigma} ~&&=~
\sum_{\cA\in\FFF\cap\CCC_Q} (-1)^{|\cA|},\\
\chi_3 &= 
\sum_{\sigma\in\EE\cap\partial\CC_\qq} (-1)^{\dim \sigma} ~&&=~
\sum_{\cA\in\EEE\cap\partial\CCC_Q} (-1)^{|\cA|}.
\end{alignat*}

In light of the argument of~\cite[Thm.~1.6.1]{AA}, one might expect our calculation of $a^{w,\pp}_{u,\qq}$ to proceed by expressing $\EE\cap\CC_\qq^\circ=(\CC_w\backslash\FF)\cap(\CC_\qq\backslash\partial\CC_\qq)$ as a signed sum of the four closed fans $\CC_w\cap\CC_\qq$, $\FF\cap\CC_\qq$, $\CC_w\cap\partial\CC_\qq$, $\FF\cap \partial\CC\qq$.  As it happens, it is difficult to determine the reduced Euler characteristics of the last two of these fans, but it is feasible (via the theory of Scrope complexes, as we will see) to analyze the single (non-closed) fan $\EE\cap\partial\CC_\qq$.

We will now calculate each of $\chi_1$, $\chi_2$, and $\chi_3$ separately, working either geometrically or combinatorially as convenient and assuming in each case that the conditions of Proposition~\ref{locality} hold.
\medskip

\begin{proposition} \label{chi-one}
Under the locality assumptions of Proposition~\ref{locality}, we have
\[
\chi_1=\begin{cases}
-1 & \text{ if $u=w$, $\pp=\qq$, and $\cD=\cN_{w,P}=\oneblock$,}\\
0 & \text{ otherwise.}
\end{cases}
\]
\end{proposition}

\begin{proof}
The fan $\CC_w\cap\CC_\qq=\CC_{\cN_{w,Q}}$ (Proposition \ref{prop:intersection}) is convex and closed and is nonempty (because it contains $\sigma_{\oneblock}$).  Therefore, $\Simplex{\CC}_{\cN_{w,Q}}$ is either the trivial simplicial complex $\{\0\}$ (when $\cN_{w,P}=\oneblock$) or a topological ball (otherwise).  Hence $\chi_1$ is equal to $-1$ if $\cN_{w,Q}=\oneblock$ and $0$ otherwise.

In fact we can sharpen this statement by incorporating the locality assumptions.  First, by~\eqref{locality:1}, we can replace the condition $\CCC_{\cW}\cap\CCC_Q=\{\oneblock\}$ with $\CCC_{\cW}\cap\CCC_Q^\circ=\{\oneblock\}$.  For this to happen, it is necessary that $\qq=\pp$.  Moreover, since $\cD\in\CCC_Q$ by assumption~\eqref{locality:2}, the case $\chi_1=-1$ occurs only when $\cD=\oneblock$, or equivalently $u=w$.
\end{proof}

\begin{remark}
The conditions $\qq=\pp$ and $u=w$ are not \emph{sufficient} to imply $\chi_1=-1$ (see Example~\ref{ex:segment}), although the three conditions $\qq=\pp$, $u=w$, and $\dim\pp=n-1$ together are sufficient.
\end{remark}

\begin{proposition} \label{chi-two}
Under the locality assumptions of Proposition~\ref{locality}, we have
\[
\chi_2=\left\{
\begin{array}{llll}
(-1)^{\des(u^{-1}w)}  & \text{ if } u\neq w \text{ and } \cD=\cN_{w,Q},\text{ and }\CCC_Q^\circ\cap\CCC_w\neq \emptyset,\\
0 &\text{ otherwise.}
\end{array}\right.
\]
\end{proposition}

\begin{proof}
There are three cases to consider.
\begin{enumerate}
\item If $|\cD|=1$ (i.e., $\cD=\oneblock$), then $\FF=\0$ and clearly $\chi_2=0$.  Therefore, in the remaining cases, we assume henceforth that $|\cD|>1$, or equivalently $u\neq w$.
	\item If $\CC_\qq$ contains a star point of $\FF$ other than the origin --- that is, if it contains some cone $\sigma_{\cA}$ such that $\oneblock\neq \cA\leq\cW$ and $\cA\meet \cD=\oneblock$ --- then that point is also a star point of $\FF\cap\CC_\qq$.  Therefore, intersecting with $\Sigma^{n-2}$ produces a topological ball, whose reduced Euler characteristic $\chi_2$ vanishes.  (This case occurs only when $1<|\cD|<n$.)
	\item Suppose that $\CC_\qq$ contains no star point of $\FF$.  That is, no cone $\sigma_{\cA}\in\FF\cap\CC_\qq$, other than $\sigma_{\oneblock}$, satisfies $\cA\meet \cD=\oneblock$.  But in particular this is true when $\sigma_{\cA}$ is a ray in $\CC_\qq$, so in fact $\CC_w\cap\CC_\qq\subseteq[\oneblock,\sigma_\cD]=\CC_\cD$ and then equality holds by assumption~\eqref{locality:2}.  Since $\FF$ contains every cone in $\CC_\cD$ other than $\sigma_\cD$ itself, it follows that $\FF\cap\CC_\qq=\CC_\cD\sm\{\sigma_\cD\}=\partial\CC_\cD$ and therefore $\chi_2=-(-1)^{\dim \sigma_\cD}=(-1)^{\des(u^{-1}w)}$ by \eqref{dimD}.  
\end{enumerate}

Now rewriting the results of the case analysis combinatorially gives the desired formula for $\chi_2$.  (The first locality assumption~\ref{locality:1} was not used, so it needs to be included in the first case of the formula.)
\end{proof}

\begin{proposition} \label{chi-three}
Under the locality assumptions of Proposition~\ref{locality}, we have
\[
\chi_3=\left\{
\begin{array}{llll}
(-1)^{\des(u^{-1}w)}\tilde\chi(\Simplex{\GG})& \text{ if } \cD\in\partial\CCC_Q \text{ and } \CCC_Q^\circ\cap\CCC_\cW\neq\0\\
0 &\text{ otherwise,}
\end{array}\right.
\]
where $\Simplex{\GG}=\Simplex{\GG}(Q,w,u)$ is a certain Scrope complex (to be constructed in the proof).
\end{proposition}

\begin{proof}
Recall that $\chi_3 = \sum_{\cA\in\GGG} (-1)^{|\cA|}$, where $\GGG=\GGG(Q,w,u)=\EEE\cap\partial\CCC_Q$.  By 
Lemma~\ref{lem:comb-dictionary} and Prop.~\ref{prop:intersection}, we have
\begin{align}
\GGG
&=\{\cA\in\CCC_Q:\ \cD\refinedbyeq \cA\refinedbyeq\cW \text{ and $\cA$ collapses some relation of $Q$}\}\notag\\
&=\{\cA\in\Comp(I):\ \cD\refinedbyeq \cA\refinedbyeq\cN \text{ and $\cA$ collapses some relation of $Q$}\}\label{sfG}
\end{align}
where $\cN=\cN_{w,Q}$ is the $w$-naturalization of~$Q$ (see \S\ref{sec:set-comps-and-gps}). Now, given a pair $c_i,d_i\in I$ such that $c_i\prec_Q d_i$ and $c_i\equiv_\cD d_i$, let $N_{x_i}$ and $N_{y_i}$ be the blocks of $\cN$ containing $c_i$ and $d_i$ respectively (so that $x_i<y_i$), and let
\[S_i = N_1 \:|\: \cdots \:|\: N_{x_i-1} \:|\: N_{x_i}\cup\cdots\cup N_{y_i} \:|\: N_{y_i+1} \:|\: \cdots \:|\: N_k.\]
Then $S_i\in\GGG$ by~\eqref{sfG}.  Moreover, every set composition in $\GGG$ is refined by or equal to some $S_i$.  In other words,
\[\GGG = \bigcup_{i=1}^r [\cD,S_i].\]
The maximal Boolean intervals $[\cD,S_i]$ of this form are those for which the integer interval $[c_i,d_i]$ is minimal; in this case we call $c_i\prec_Qd_i$ a \defterm{short} relation.  Thus in computing $\GGG$ it suffices to consider only short relations.

As usual, let us identify every set composition in $[\cD,\cW]$ with the simplex on its separators (by passing to braid cones and intersecting with $\Ss^{n-2}$, as in \S\ref{sec:set-comps-and-gps}).  Then the simplicial complex $\Simplex{\GG}=\Simplex{\GG}(Q,w,u)$ corresponding to $\GGG$ is a Scrope complex whose vertices correspond to the separators between blocks of $\cN$; specifically, $\Simplex{\GG}\isom\Scr(k,\zz)$, where $\zz=((x_1,y_1),\dots,(x_r,y_r))$.  It follows that $\chi_3=(-1)^{|\cD|-1}\tilde\chi(\Simplex{\GG})=(-1)^{\des(u^{-1}w)}\tilde\chi(\Simplex{\GG})\in\{0,-1,1\}$, as desired.

For the locality assumptions, note that $\cD\in\partial\CCC_Q$ directly implies $(\star\star)$.  Condition $(\star)$ did not arise in the calculation, so it is incorporated directly into the formula for $\chi_3$.
\end{proof}

\begin{example}
Let $w=u=\mathsf{12345678}$ (as linear orders) so that $\cD=\oneblock=\mathsf{12345678}$.  Let $Q$ be the $w$-natural preposet shown below, so that that $\cN=\mathsf{1|234|567|8}$.
\begin{center}
\begin{tikzpicture}
\draw (0,0)--(0,2)  (0,1)--(1,2);
\node[fill=white] at (0,2) {\sf567};	\node[fill=white] at (1,2) {\sf8};
\node[fill=white] at (0,1) {\sf234};
\node[fill=white] at (0,0) {\sf1};
\end{tikzpicture}
\end{center}

The set compositions $A$ such that $\cD\refinedbyeq A\refinedbyeq N$ are
\newcommand{\ms}[1]{\mathsf{#1}}
\newcommand{\ds}[1]{\dot{\mathsf{#1}}}
\newcommand{\us}[1]{\underaccent{\dot}{\mathsf{#1}}}
\[\begin{array}{llll}
\ds{1}   \ds{2} \us{3} \ms{4}   \us{5} \ms{6} \ms{7}   \ms{8} \quad&
\ds{1}   \ds{2} \us{3} \ms{4}   \us{5} \ms{6} \ms{7} | \ms{8} \quad&
\ds{1}   \ds{2} \ms{3} \ms{4} | \ms{5} \ms{6} \ms{7}   \ms{8} \quad&
\ds{1}   \ds{2} \ms{3} \ms{4} | \ms{5} \ms{6} \ms{7} | \ms{8} \\
\ms{1} | \ms{2} \us{3} \ms{4}   \us{5} \ms{6} \ms{7}   \ms{8} \quad&
\ms{1} | \ms{2} \us{3} \ms{4}   \us{5} \ms{6} \ms{7} | \ms{8} \quad&
\ms{1} | \ms{2} \ms{3} \ms{4} | \ms{5} \ms{6} \ms{7}   \ms{8} \quad&
\ms{1} | \ms{2} \ms{3} \ms{4} | \ms{5} \ms{6} \ms{7} | \ms{8} 
\end{array}\]
where each of the short relations $1\prec_Q2$ and $3\prec_Q5$ is marked whenever it occurs in a block.
Thus the last two set compositions listed do not belong to $\GGG$.  Moreover,
\begin{align*}
\GGG &= [\cD,S_{12}] \cup [\cD,S_{35}]\\
&= [\mathsf{12345678},\mathsf{1234|567|8}]\cup[\mathsf{12345678},\mathsf{1|234567|8}]
\end{align*}
whose geometric realization is shown below.  Note that $\tilde\chi(\GGG)=0$.
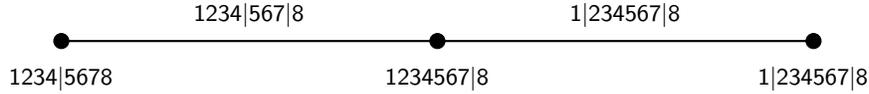
\begin{center}
\begin{tikzpicture}
\newcommand{\len}{5}
\coordinate (P1) at (0,0);
\coordinate (P2) at (\len,0);
\coordinate (P3) at (2*\len,0);
\draw[thick] (P1)--(P3);
\foreach \co in {P1,P2,P3} \draw[black,fill=black] (\co) circle (.1);
\node at (0.0*\len, -.1*\len) {\scriptsize$\mathsf{1234|5678}$};
\node at (0.5*\len, .07*\len) {\scriptsize$\mathsf{1234|567|8}$};
\node at (1.0*\len, -.1*\len) {\scriptsize$\mathsf{1234567|8}$};
\node at (1.5*\len, .07*\len) {\scriptsize$\mathsf{1|234567|8}$};
\node at (2.0*\len, -.1*\len) {\scriptsize$\mathsf{1|234567|8}$};
\end{tikzpicture}
\end{center}
\end{example}

\begin{proposition} \label{mult-cancel}
The formula $a^{w,\pp}_{u,\qq}=\chi_1-\chi_2-\chi_3$ is multiplicity- and cancellation-free.  That is, each of the terms $\chi_1$, $\chi_2$, $\chi_3$ is 0, 1, or $-1$, and at most one of them is nonzero.
\end{proposition}
\begin{proof}
First, suppose that $\chi_1\neq0$.  Then by Proposition~\ref{chi-one} $u=w$ (so $\cD=\oneblock$) and $\pp=\qq$.  It follows from Proposition~\ref{chi-two} that $\chi_2=0$.  Moreover, the preposet $Q$ has no proper relations, so $\GGG=\0$ and $\chi_3=0$.

Second, suppose that $\chi_1=0$ and $\chi_2\neq0$.  Then $\cD\in\CCC^\circ_Q$, but then $\cD\not\in\partial\CCC_Q$ and so $\chi_3=0$.
\end{proof}

By Proposition~\ref{mult-cancel}, we can combine Propositions~\ref{chi-one}, \ref{chi-two}, and \ref{chi-three} into a single formula for the coefficients of the antipode:
\begin{equation} \label{antipode-coeffs}
a^{w,\pp}_{u,\qq} = \begin{cases}
-1 & \text{ if $u=w$, $\pp=\qq$, and $\cD=\cN_{w,P}=\oneblock$,}\\
-(-1)^{\des(u^{-1}w)}  & \text{ if } \cD\neq\oneblock\ \text{ and } \cD=\cN_{w,Q},\text{ and } \CCC_{\cW}\cap\CCC_Q^\circ\neq \emptyset,\\
-(-1)^{\des(u^{-1}w)}\tilde\chi(\Simplex{\GG}(Q,w,u)) & \text{ if $\cD\in\partial\CCC_Q$ and $\CCC_{\cW}\cap\CCC^\circ_Q\neq\0$},\\
0 & \text{ otherwise.}
\end{cases}
\end{equation}

We can finally write down a combinatorial formula for the antipode in $\OGP^+$:
\begin{equation} \label{first-antipode-formula}
\begin{aligned}
\anti(w\otimes \pp) &=
- \xi w\otimes\pp
- \sum_{\substack{(u,\qq)\in(\bl_\pp[I]\sm\{w\})\x\Faces(\pp):\\ \cD=\cN_{w,Q},\ \CCC_{\cW}\cap\CCC_Q^\circ\neq \emptyset}} (-1)^{\des(u^{-1}w)} u\otimes\qq\\
&\quad- \sum_{\substack{(u,\qq)\in\bl_\pp[I]\x\Faces(\pp):\\ \cD\in\partial\CCC_Q,\ \CCC_{\cW}\cap\CCC^\circ_Q\neq\0}} (-1)^{\des(u^{-1}w)}\tilde\chi(\Simplex{\GG}(Q,w,u)) u\otimes\qq
\end{aligned}
\end{equation}
where $\xi$ is 1 if $\CCC_{\cW}\cap\CCC_P=\{\oneblock\}$ and 0 otherwise.  (In particular, $\xi=1$ if $\pp$ has full dimension $|I|-1$.)

In fact this expression can be simplified.  Suppose we extend the range of summation for $u$ in the first sum from $\bl_\pp[I]\sm\{w\}$ to $\bl_\pp[I]$.  If $u=w$, then $\cD=\oneblock$ and $\CCC_\cD=\{\oneblock\}$, so $\CCC_{\cW}\cap\CCC_Q=\CCC_\cD$ if and only if $\oneblock\in\CCC_Q$, i.e., $\qq=\pp$ and $\xi=1$.  Therefore, we can absorb the term $-\xi w\otimes\pp$ into the first sum to obtain the final result, as follows.

\begin{theorem} \label{OGP-antipode}
The antipode in $\OGP^+$ is given by the formula
\begin{equation} \label{giant-antipode-formula}
\anti(w\otimes\pp) =
\underbracket{
\sum_{\substack{(u,\qq)\in\bl_\pp[I]\x\Faces(\pp):\\
\cD=\cN_{w,Q},\ \CCC_{\cW}\cap\CCC_Q^\circ\neq \emptyset}}
(-1)^{1+\des(u^{-1}w)} u \otimes \qq
}_{\anti_1}+\underbracket{
\sum_{\substack{(u,\qq)\in\bl_\pp[I]\x\Faces(\pp):\\
\cD\in\partial\CCC_Q,\ \CCC_{\cW}\cap\CCC^\circ_Q\neq\0}}
(-1)^{1+\des(u^{-1}w)} \tilde\chi(\Simplex{\GG}) u \otimes \qq
}_{\anti_2} 
\end{equation}
where $\Simplex{\GG}=\Simplex{\GG}(Q,w,u)$ is the Scrope complex defined above.
\end{theorem}

\begin{remark}
Recall that summing over all $w\in\Sym_n$ produces the much simpler expression~\eqref{eq:antisym}.  It is not immediately clear why so much cancellation occurs in that symmetrization.
\end{remark}

\begin{remark}
It is tempting to try to combine $\anti_1$ and $\anti_2$ into a single sum.  To accomplish this, it is necessary to understand the Scrope complex $\Simplex{\GG}(Q,w,u)$ in the case $\cD=\cN$.  By~\eqref{sfG}, if $\cN\in\bd\CCC_Q$ then $\GGG=\{N\}$, while if $\cN\in\CCC_Q^\circ$ then  $\GGG=\0$.  But in the first of those cases the condition $\CCC_\cW\cap\CCC_Q^\circ\neq\0$ fails, while in the second case $\chi(\GG)=\chi(\0)$ is 0, not 1, so the sums cannot be combined.
\end{remark}

\begin{remark}\label{s2}
The expression $\anti_2$ may contain some zero summands, since Scrope complexes can be contractible.  However, it is not clear how to predict when this will happen.  It would be helpful to have a non-recursive way of determining the topology of a Scrope complex. On the other hand, when $\cD\in\partial\CCC_Q$ is a maximal element, then the corresponding Scrope is the trivial complex, whose Euler characteristic is nonzero. So $\anti_2$ is nonzero.
\end{remark}

\begin{example} \label{sign-can-depend-on-u}
The sign of a term $u\otimes\qq$ in $\anti(w\otimes\pp)$ may depend on $u$ as well as on $\qq$.
For example, let $\pp$ be the three-dimensional cone of Example~\ref{ex:cone}.  This cone has a unique vertex $\qq=(1,2,3,4)$, whose normal cone $\CC_\qq$ is the union of three braid cones (see Figure~\ref{fig:cone}). Let $w=1234$. The intersection $\partial\CCC_\qq\cap \CCC_w$ has two maximal elements, the vertex $1|234$ and the edge $12|3|4$. Hence if $u=2314$, then $\cD(w,u)=1|234\in \CCC_\qq\cap \CCC_w$ and $(-1)^{1+\des(u^{-1}w)} \tilde\chi(\Simplex{\GG})=1$. Else if $u=4312$, then $\cD(w,u)=12|3|4\in \CCC_\qq\cap \CCC_w$ and $(-1)^{1+\des(u^{-1}w)} \tilde\chi(\Simplex{\GG})=-1$.
\end{example}

\begin{example} \label{ex:segment}
We give an example of an antipode calculation for an unbounded polytope.
Let $\pp$ be the segment $\conv\{(1,0,0),(0,1,0)\}$, and let $\pp'$ be the ray with vertex $(0,1,0)$ and direction $(1,-1,0)$.
Thus $\pp'\supset\pp$, and in fact the normal fan $\NN_{\pp'}$ is a non-complete subfan of $\NN_\pp$, as shown below.

\begin{figure}[ht]
\begin{center}
\begin{tikzpicture}[scale = 0.9]
   \newdimen\R
   \R=2.5cm
\begin{scope}[shift={(0,0)}]
\draw[ultra thick] (1,0) -- node[below] {$\mathfrak{p}$} (-1,0);
\filldraw(1,0) circle (.1) node[above] {$\mathfrak{b}$};
\filldraw (-1,0) circle (.1) node[above] {$\mathfrak{a}$};
\end{scope}
\begin{scope}[shift={(0,-3.5)}]
   \fill[green!10!white] (270:\R*5/6*.866) -- (300:\R*5/6) -- (0:\R*5/6) -- (60:\R*5/6) -- (90:\R*5/6*.866) -- cycle;
   \fill[cyan!10!white] (90:\R*5/6*.866) -- (120:\R*5/6) -- (180:\R*5/6) -- (240:\R*5/6) -- (270:\R*5/6*.866) -- cycle;
   \draw[dashed, red] (150:\R)--(330:\R) (210:\R) -- (30:\R);
   \draw[thick, blue] (270:\R*5/6) -- (90:\R*5/6);   
   \foreach \x/\l in { 180/{2|3|1}, 240/{2|1|3}, 300/{1|2|3}, 0/{1|3|2}, 60/{3|1|2}, 120/{3|2|1} } 
      { \node at (\x:\R*1/2) {\footnotesize$\mathsf{\l}$}; }
   \node[fill=cyan!10!white] at (-1.25*\R,0) {\footnotesize$\NN^\circ_\pp(\aa)$};
   \node[fill=green!10!white] at (1.25*\R,0) {\footnotesize$\NN^\circ_\pp(\bb)$};
   \node[blue] at (0,.9*\R) {\footnotesize$\mathsf{3|12}$};
   \node[blue] at (0,-.9*\R) {\footnotesize$\mathsf{12|3}$};
   \node[left, red] at (150:\R*3/4) {{\footnotesize$\mathsf{23|1}$}\ \ };
   \node[left, red] at (210:\R*3/4) {{\footnotesize$\mathsf{2|13}$}\ \ };
   \node[right, red] at (330:\R*3/4) {\ \ \footnotesize$\mathsf{1|23}$};
   \node[right, red] at (30:\R*3/4) {\ \ \footnotesize$\mathsf{13|2}$};   
   \filldraw (0,0) circle (.1) node[left] {$\oneblock$};
   \node at (270:1.25*\R) {$\NN_\pp$};
\end{scope}
\begin{scope}[shift={(8,0)}]
\draw[ultra thick] (1,0) -- node[below] {$\mathfrak{p}'$} (-1,0);
\draw[ultra thick,->] (1,0) -- (-1,0);
\filldraw (1,0) circle (.1) node[above] {$\mathfrak{b}$};
\end{scope}
\begin{scope}[shift={(8,-3.5)}]
   \fill[green!10!white] (270:\R*5/6*.866) -- (300:\R*5/6) -- (0:\R*5/6) -- (60:\R*5/6) -- (90:\R*5/6*.866) -- cycle;
   \draw[dashed, red] (330:\R) -- (0,0) -- (30:\R);
   \draw[thick, blue] (270:\R*5/6) -- (90:\R*5/6);   
   \foreach \x/\l in { 300/{1|2|3}, 0/{1|3|2}, 60/{3|1|2} }
      \node at (\x:\R*1/2) {\footnotesize$\mathsf{\l}$};
   \node[fill=green!10!white] at (1.25*\R,0) {\footnotesize$\NN^\circ_\pp(\bb)$};
   \node[blue] at (0,.9*\R) {\footnotesize$\mathsf{3|12}$};
   \node[blue] at (0,-.9*\R) {\footnotesize$\mathsf{12|3}$};
   \node[right, red] at (330:\R*3/4) {\ \ \footnotesize$\mathsf{1|23}$};
   \node[right, red] at (30:\R*3/4) {\ \ \footnotesize$\mathsf{13|2}$};
   \filldraw (0,0) circle (.1) node[left] {$\oneblock$};
   \node at (270:1.25*\R) {$\NN_{\pp'}$};
\end{scope}
\end{tikzpicture}
\end{center}
\label{fig:segment}
\end{figure}
The computation of $\anti(w\otimes\pp)$ is shown in the table below.
If $\pp'$ is bounded in the direction given by $w$ (that is, if the braid cone $\sigma_w$ appears in $\NN_{\pp'}$), then the computation of
$\anti(w\otimes\pp')$ is identical to that of $\anti(w\otimes\pp)$, replacing $\pp$ with $\pp'$.  Note that the point $\aa$, which is a face of $\pp$ but not of $\pp'$, does not appear in these cases.
\[\begin{array}{|c|c|c|c|c|} \hline
\pad w & \chi_1 & \chi_2 & \chi_3 & \text{$\pp'$\ \ bounded?}\\ \hline\hline
\pad 123 & 0 & \left(132+312\right)\otimes \pp-321\otimes \bb & \left(-132-312\right)\otimes \bb & \text{Yes}\\ \hline
\pad 132 & -132\otimes\pp & -231\otimes \bb & +132\otimes\bb & \text{Yes}\\ \hline
\pad 312 & 0 & \left(123+132\right)\otimes \pp-213\otimes \bb & \left(-123-132\right)\otimes \bb & \text{Yes}\\ \hline
\pad 321 & 0 & \left(213+231\right)\otimes \pp-123\otimes \aa & \left(-213-231\right)\otimes \aa & \text{No}\\ \hline
\pad 231 & -231\otimes\pp & -132\otimes \aa & +231\otimes\aa & \text{No}\\ \hline
\pad 213 & 0 & \left(231+321\right)\otimes \pp-312\otimes \aa & \left(-231-321\right)\otimes \aa & \text{No}\\ \hline
\end{array}\]
Observe that each row is itself cancellation-free, and that adding rows together recovers the symmetrized formulas
\begin{equation*}
\sum_{w\in\Sym_3} \anti(w\otimes\pp) = \sum_{w\in\Sym_3} w\otimes(\pp-\aa-\bb)
\text{  and }
\sum_{w\in\NN_{\pp'}} \anti(w\otimes\pp') = \sum_{w\in\NN_{\pp'}} w\otimes\pp' - \sum_{w\in\Sym_3} w\otimes\bb
\end{equation*}
confirming~\eqref{eq:antisym} in this case.
\end{example}

\section{Special cases of the antipode} \label{sec:special}

In this section, we specialize the antipode formula of Theorem~\ref{OGP-antipode} to several natural families of OGPs: standard permutohedra, hypersimplices, and zonotopes of star graphs.  We discuss the difficulties involved in calculating the antipode for other families, including general graphical zonotopes and matroid complexes.  Throughout, we fix the ground set $I=[n]$, so that we can identify orderings in $\bl[I]$ with permutations in $\Sym_n$.

\subsection{The standard permutohedron}

Let $\pp=\Pi_{n-1}\subset\Rr^{[n]}$ be the standard permutohedron (see \S\ref{sec:background-for-gp}), let $w\in\Sym_n$, and let $\cW=\comp(w)$.  The normal fan of $\Pi_{n-1}$ is exactly the braid fan, so its faces $\qq=\pp_Q$ are labeled by set compositions $Q$, and every $\CCC_Q$ is a principal order ideal (in fact a Boolean interval) in $\Comp(n)$.
We compute the antipode of $w\otimes \Pi_{n-1}$ using Theorem~\ref{OGP-antipode}.
\begin{enumerate}

\item Since $Q$ must be $w$-natural in order to show up in the antipode (Corollary~\ref{cor:naturality}), the condition $\CCC_{\cW}\cap\CCC_Q=\CCC_\cD$ becomes simply $\CCC_Q=\CCC_\cD$, or $Q=\cD$.  Therefore, the first sum in Theorem~\ref{OGP-antipode} becomes
\[ 
\anti_1=\sum_{u,Q:\; \cD(w,u)=Q} (-1)^{\des(u^{-1}w)}u\otimes \pp_Q = (-1)^{|Q|}\sum_{u:\cD(w,u)=Q}u\otimes \pp_Q.
\]
\item The condition $\cD(w,u)\in\partial\CCC_Q$ is equivalent to $\cD(w,u)\refinedby Q$.
Since $Q$ is a set composition, we have $Q=N$, so $\GGG=\{A:\ \cD\refinedbyeq \cA\refinedby Q\}$, which is a Boolean interval with its top element missing.  Therefore $\Simplex{\GG}=\Simplex{\GG}(Q,w,u)$ is the boundary of a simplex of dimension $|Q|-|\cD|-1=|Q|-\des(u^{-1})-2$ and so $\tilde\chi(\Simplex{\GG})=(-1)^{|Q|-\des(u^{-1})}$.
So the second sum in Theorem~\ref{OGP-antipode} becomes
\[ 
\anti_2=(-1)^{\des(u^{-1}w)}\sum_{u,Q:\; \cD(w,u)\refinedby Q} (-1)^{|Q|-\des(u^{-1})} u\otimes \pp_Q = (-1)^{|Q|}\sum_{u,Q:\; \cD(w,u)\refinedby Q} u\otimes \pp_Q.
\]
\end{enumerate}
Putting the two sums together gives
\begin{equation}\label{antipode-pi-d}
\anti(w\otimes \Pi_{n-1})=\anti_1+\anti_2=\sum_{Q\textrm{ $w$-natural}} (-1)^{|Q|} \sum_{u:\; \cD(w,u)\refinedbyeq Q}u\otimes \pp_Q.
\end{equation}
The condition $\cD(w,u)\refinedbyeq Q$ is equivalent to $\Des(u^{-1}w)\subseteq\{|Q_1|,|Q_1|+|Q_2|,\dots,|Q_1|+\cdots+|Q_{k-1}|\}$, where $Q=Q_1|\cdots|Q_k$. See \cite[\S1.4]{EC1} for more information about enumerating permutations according to their descent set.

\subsection{Spider preposets}

For our next two examples of antipodes in $\OGP^+$, we need the following family of preposets.
\begin{definition}\label{def:spider}
Let $A,B$ be disjoint subsets of $I$.  The corresponding \defterm{spider} $Q=Q(A,B)$ is the preposet on $I$ such that $C=C(Q)=I\sm(\cA\cup B)$ is a block of $Q$ (called the \defterm{center}), and every other element is a singleton block; $C$ lies above the singletons in $A$ and below the singletons in $B$.  
\end{definition}

Let $w$ be a linear order on $[n]$ such that $Q$ is $w$-natural, and let $\cW=\comp(w)$.  Equivalently, $A$ and $B$ are respectively initial and final segments of $w$; that is, $A=\ini{k}{w}=\{w(1),\dots,w(k)\}$ and $B=\fin{\ell}{w}=\{w(n-\ell-1),\dots,w(n)\}$.  For the purpose of the antipode formula, we need to understand the poset $\partial\CCC_Q\cap\CCC_\cW$.

Write $A=\{a_1,\dots,a_k\}$ and $B=\{b_1,\dots,b_\ell\}$, with elements indexed according to their order in $w$.  Define set compositions
\begin{align*}
\cN_Q &= a_1\;|\;\cdots\;|\;a_{k-1}\;|\;a_k\;|\;C\;|\;b_1\;|\;b_2\;|\;\cdots\;|\;b_\ell,\\
\cN_Q^a &= a_1\;|\;\cdots\;|\;a_{k-1}\;|\;a_k\cup C\;|\;b_1\;|\;b_2\;|\;\cdots\;|\;b_\ell,\\
\cN_Q^b &= a_1\;|\;\cdots\;|\;a_{k-1}\;|\;a_k\;|\;C\cup b_1\;|\;b_2\;|\;\cdots\;|\;b_\ell,\\
\cN_Q^{ab} &= a_1\;|\;\cdots\;|\;a_{k-1}\;|\;a_k\cup C\cup b_1\;|\;b_2\;|\;\cdots\;|\;b_\ell.
\end{align*}
Here $\cN_Q$ is the same set composition defined in the proof of Proposition~\ref{chi-three}; moreover, $\partial\CCC_Q\cap\CCC_\cW$ is the order ideal of $\Comp(n)$ generated by $\cN_Q^a$ and $\cN_Q^b$.  In all cases, the interval $[\cN_Q^{ab},\cN_Q]=\{\cN_Q^{ab},\cN_Q^a,\cN_Q^b,\cN_Q\}$ is Boolean; however, if one or more of $A,B,C$ is empty, then some of these set compositions coincide.  If two or more of $A,B,C$ vanish, then $\cN_Q=\cN_Q^a=\cN_Q^b=\cN_Q^{ab}$.

\begin{lemma}\label{lem:improved-spider}
Let $Q=Q(A,B)$ be a spider on $[n]$ and $w$ a linear order on $[n]$ such that $Q$ is $w$-natural.  Let $u\in\Sym_n$ and $\cD=\cD(w,u)$,
and $\Simplex{\GG}=\Simplex{\GG}(q,w,u)$, as constructed in the proof of Proposition~\ref{chi-three}.  Then
\[\tilde{\chi}(\Simplex{\GG}) =
\begin{cases}
(-1)^{|Q|-|\cD|} & \text{ if } \cN_Q^{ab}\refinedbyeq \cD\refinedby \cN_Q,\\
0 & \text{ otherwise.}
\end{cases}
\]
\end{lemma}

\begin{proof}
If at least two of $A,B,C$ are empty, then $\Simplex{\GG}$ is the void complex, with reduced Euler characteristic~0.
Meanwhile, the interval $[\cN_Q^{ab},\cN_Q]$ is trivial, so the condition $\cN_Q^{ab}\refinedbyeq \cD\refinedby \cN_Q$ is impossible.

If $A=\0$ and $B,C\neq\0$, then $Q$ has one short relation, namely $C\prec_Q b_1$, so $\Simplex{\GG}\isom[\cD,\cN_Q^b]$,
which is either a simplex or void (hence has reduced Euler characteristic zero) unless $\cD=\cN_Q^b=\cN_Q^{ab}$.  In this case $\Simplex{\GG}$ is the trivial complex, with reduced Euler characteristic~$-1$.  The case that $B=\0$ and $A,C\neq\0$ is analogous.

Similarly, if $C=\0$ and $A,B\neq\0$, then again $Q$ has one short relation, namely $a_k\prec_Qb_1$, so $\Simplex{\GG}\isom[\cD,\cN_Q^{ab}]$, whose Euler characteristic is $-1$ if $\cD=\cN_Q^{ab}$ and 0 otherwise.

Finally, suppose that $A,B,C$ are all nonempty.  Then $Q$ has two short relations, namely $a_k\prec_Q C$ and $C\prec_Q b_1$.  Thus the face poset of $\Simplex{\GG}$ is isomorphic to $[\cD,\cN_Q^a]\cup[\cD,\cN_Q^b]\subseteq\Comp(n)$.  The cases are as follows:
\begin{itemize}
\item If $\cD=\cN_Q^{ab}$ then $\Simplex{\GG}\isom\mathbb{S}^0$, whose reduced Euler characteristic is $+1$.
\item If $\cD=\cN_Q^a$ or $S=\cN_Q^b$, then $\Simplex{\GG}$ is the trivial complex, with reduced Euler characteristic~$-1$.
\item If $\cD$ coarsens exactly one of $\cN_Q^a$ or $\cN_Q^b$, then the face poset is Boolean of positive rank, so $\Simplex{\GG}$ is a simplex, hence contractible.
\item If $\cD\refinedby \cN_Q^{ab}$, then $\Simplex{\GG}$ is the union of two simplices that intersect in a common subface, so again it is contractible.
\item Otherwise, $\Simplex{\GG}$ is the void complex.
\end{itemize}
Thus the claimed formula holds in all cases.
\end{proof}

\subsection{Hypersimplices} \label{sec:hypersimplices}

For positive integers $n>r$, the \defterm{$(n,r)$ hypersimplex} is defined as the polytope
\[\Delta(n,r)=\left\{(x_1,\dots,x_n)\in[0,1]^n: \sum_{i=1}^n x_i=r\right\}.\]
The hypersimplex is the matroid base polytope of the uniform matroid of rank~$r$ on $n$ ground elements; in particular, it is a generalized permutohedron of dimension $n-1$.  Its vertices are the vectors $(x_1,\dots,x_n)$ with exactly $r$ entries equal to 1 and the rest equal to zero.    In particular, $\Delta(n,0)$ is a point and $\Delta(n,1)$ is the standard simplex, and the polytopes $\Delta(n,r)$ and $\Delta(n,n-r)$ are congruent.

It is not hard to describe and count the faces of a hypersimplex, but we have been unable to find an explicit statement in the literature, so we give a short proof here.

\begin{proposition} \label{hypersimplex-faces}
The faces of $\Delta(n,r)$ are precisely the polytopes
\begin{equation} \label{spidertope}
\qq(A,B)=\{(x_1,\dots,x_n)\in\Delta(n,r): \quad x_a=0\ \ \forall a\in A,\quad x_b=1\ \ \forall b\in B\}
\end{equation}
where $A,B$ are disjoint subsets of $[n]$ such that either $|A|=n-r$ and $|B|=r$ (when $\qq(A,B)$ is a vertex) or 
$|A|<n-r$ and $|B|<r$ (when $\dim\qq(A,B)=n-1-|A|-|B|$).  Consequently, the number of $d$-dimensional faces is
\[f_0(\Delta(n,r))=\binom{n}{r}, \qquad f_d(\Delta(n,r)) = \binom{n}{d+1} \sum_{k=r-d}^{r-1} \binom{n-d-1}{k} \quad (1\leq d\leq n-1).\]
\end{proposition}
\begin{proof}
Let $\lambda(x_1,\dots,x_n)=\lambda_1x_1+\cdots+\lambda_nx_n$ be a linear functional on $\Rr^n$.  Without loss of generality, suppose $\lambda_1\leq\cdots\leq\lambda_n$.  If $\lambda_r<\lambda_{r+1}$ then $\lambda$ is maximized uniquely at the vertex $\qq([1,n-r],[n-r+1,n])$.  Otherwise, let $i,j$ be such that $\lambda_i<\lambda_{i+1}=\cdots=\lambda_{n-r}=\lambda_{n-r+1}=\cdots=\lambda_{j-1}<\lambda_j$.  Then $\lambda$ is maximized at the vertices in $\qq([1,i],[j,n])$, and $|[1,i]|<n-r$ and $|[j,n]|<r$.
\end{proof}

The normal preposet of $\qq(A,B)$ is the spider $Q(A,B)$ of Definition~\ref{def:spider}.

For $w\in\Sym_n$, the spider $Q(A,B)$ is $w$-natural if and only if $A$ and $B$ are initial and final segments of $w$, say $A=\ini{k}{w}$ and $B=\fin{\ell}{w}$.  In this case we write $Q(A,B)=Q^w_{k,\ell}$, and we also write $\qq^w_{k,\ell}$ for the corresponding face of $\Delta(n,r)$.  By Corollary~\ref{cor:naturality}, the $\qq^w_{k,\ell}$ are the only faces occurring in the antipode of $w\otimes\Delta(n,r)$.  This is a strong restriction; for instance, the only vertex that can occur is $\qq_{n-r,r}$.  For $\Delta(4,2)$ and $w=[1234]$, the $w$-natural faces are those indicated in Figure~\ref{fig:localex}.

\begin{theorem} \label{hypersimplex-antipode}
Let $w\in\bl[I]=\Sym_n$.  Then
\begin{align*}
\anti(w\otimes\Delta(n,r))
&= \sum_{k,\ell} \sum_{\substack{u\in\bl[I]:\\ \cN_Q^{ab}\refinedbyeq \cD(w,u)\refinedbyeq \cN_Q}} u\otimes(-1)^{n-\dim\qq} \qq
\end{align*}
where either $(k,\ell)=(n-r,r)$, or else $0\leq k<n-r$ and $0\leq \ell<r$; $Q=Q^w_{k,\ell}$; and $\qq=\qq^w_{k,\ell}$.
\end{theorem}

The condition $\cN_Q^{ab}\refinedbyeq \cD(w,u)\refinedbyeq \cN_Q$ in the above formula is equivalent to $[1,k-1]\cup[n-\ell,n-1]\subseteq\Des(u^{-1}w)\subseteq[1,k]\cup[n-\ell-1,n-1]$.
\begin{proof}
In $\anti_1$ we have $\cD=\cN_Q$, so that $\des(u^{-1}w)+1=|\cD|=|\cN_Q|=|Q|=n-\dim\qq$.  Meanwhile, by Lemma~\ref{lem:improved-spider}, the summand in~$\anti_2$ vanishes unless $\cD\in\{\cN_Q^a,\cN_Q^b,\cN_Q^{ab}\}$, all of which belong to $\bd\CCC_Q$.  Therefore, Theorem~\ref{OGP-antipode} boils down to
\begin{align*}
\anti(w\otimes\Delta(n,r)) &=
\sum_{k,\ell} \sum_{\substack{u\in\bl[I]:\\ \cD(w,u)=\cN_Q}} (-1)^{|Q|} u \otimes \qq +
\sum_{k,\ell} \sum_{\substack{u\in\bl[I]:\\ \cN_Q^{ab}\refinedbyeq \cD(w,u)\refinedby \cN_Q}} (-1)^{|\cD|} (-1)^{|Q|-|\cD|} u \otimes \qq,
\end{align*}
and combining the two sums and rewriting in terms of $\qq$ gives the theorem.
\end{proof}
Another example that can be computed using similar ideas to Theorem \ref{hypersimplex-antipode} is the case of zonotopes of star graphs.

\subsection{Zonotopes of stars}
Let $G$ be a simple graph on vertex set $[n]$.  The corresponding \defterm{graphical zonotope} is the Minkowski sum $\sum_{ij}\ss_{ij}$, where $ij$ ranges over all edges of $G$ and $\mathfrak{s}_{ij}$ is the line segment from $\ee_i$ to $\ee_j$.  There does not appear to be a combinatorial formula for the antipode of an ordered graphic zonotope, since we lack a general description of the preposets corresponding to normal cones of faces.  However, in the following special case, we can give an explicit formula.  Let $G=\St(n,c)$ be the star graph with center vertex $c$ and leaves $[n]\sm\{c\}$, for some $c\in[n]$, so that the corresponding zonotope is
\begin{equation}\label{eq:stineqs}
\zzz_{n,c} = \sum_{i\neq c} \ss_{ic} = \left\{(x_1,\dots,x_n)\in\Rr^n:\ \ \sum_{i=1}^n x_i=n-1,\quad 0\leq x_j\leq 1 \ \ \forall j\neq c\right\},
\end{equation}
which is a parallelepiped.  The faces of $\zzz_{n,c}$ are precisely the polytopes
\[\qq(A,B)=\{(x_1,\dots,x_n)\in\zzz(n,c): \ \ x_a=0\ \ \forall a\in A,\quad x_b=1\ \ \forall b\in B\}\] 
where $A,B$ are disjoint subsets of $[n]\sm\{c\}$.  We reuse the notation of the previous example because, as before, $\dim\qq(A,B)=n-1-|A|-|B|$, and the normal preposet of $\qq(A,B)$ is the spider $Q(A,B)$ of Definition~\ref{def:spider}; the spiders that arise are precisely those whose center contains~$c$.

\begin{theorem} \label{star-antipode}
Let $w\in\bl[I]=\Sym_n$.  Then
\[
\anti(w\otimes\zzz(n,c))
= \sum_{k,\ell} \sum_{\substack{u\in\bl[I]:\\ \cN_\qq^{ab}\refinedbyeq \cD(w,u)\refinedbyeq \cN_\qq}} u\otimes(-1)^{n-\dim\qq} \qq
\]
where $k,\ell$ are nonnegative integers with $k\leq w^{-1}(c)-1$ and $\ell\leq n-w^{-1}(c)$, and $\qq=\qq_{k,\ell}$.
\end{theorem}
The proof is the same as in Theorem \ref{hypersimplex-antipode} \emph{mutatis mutandis}.

In general, an explicit antipode formula for a class of generalized permutohedra requires a description of their normal fans in terms of preposets, as well as an understanding of the Scrope complexes that arise in the $\anti_2$ piece of Theorem~\ref{OGP-antipode}.  For this reason, it is probably intractable to ask for a formula for a class as large as all matroid polytopes, for instance.

\section{Open questions}\label{sec:open}
\begin{enumerate}
\item Is there a closed formula for the Euler characteristic of a Scrope complex?  For individual examples it can be calculated efficiently using Proposition~\ref{scrope-homotopy}, but for the purpose of simplifying antipode calculations it would be better to have a non-recursive combinatorial condition in terms of the intervals that define facets.  Are they instances of the cell complexes described in \cite[\S7.7]{AM17} and \cite[\S12.9]{AM20} whose Euler characteristics compute antipode coefficients for Hopf monoids relative to a hyperplane arrangement?  (The last question was suggested to the authors by an anonymous referee.)

\item A \textit{character} on a Hopf monoid $\bH$ is essentially a multiplicative function: a collection of linear maps $\zeta:\bH[I]\to\Cc$ such that $\zeta(x\cdot y)=\zeta(x)\zeta(y)$; see~\cite[\S2.1]{AA} for details.  Ardila and Aguiar used character theory on $\GP$ and its submonoids for numerous applications, including formulas for multiplicative and compositional inverses of power series, polynomial invariants of Hopf monoids.  What are the parallel results for $\OGP^+$ and its submonoids?

\item For what other families of ordered generalized permutohedra, other than those studied in \S\ref{sec:special}, does the general antipode formula (Theorem~\ref{OGP-antipode}) simplify nicely?  Families of interest include associahedra and nestohedra; see \cite{PRW}.  As mentioned at the end of \S\ref{sec:special}, it is necessary to know their normal fans explicitly and to understand which Scrope complexes arise in the antipode.

 \item The Hopf class $\omat$ of ordered matroids is geometric, in the sense that mapping every matroid to its base polytope gives an embedding
$\OMat=\omat^\natural\to\OGP^+$.  For which other Hopf classes is there a comparable geometric embedding?  Of particular interest is the Hopf class $\bc$ of broken-circuit complexes (see Example~\ref{ex:BC}), in light of their crucial role in the recent work of Ardila, Denham, and Huh \cite{ADH}.  Is there a maximal Hopf class that embeds in $\OGP^+$, and if so, can it be characterized in purely combinatorial terms?  Aguiar and Ardila's maximality theorem for $\GP$~\cite[Thm.~1.5.1]{AA} implies that any such embedding necessarily involves polyhedra that are either unbounded or contain vertices that are not 0/1 vectors. 

\item The basis elements of $\bL^*$ are permutations, which correspond to maximal cones in the braid fan.  We know that the antipode of an element $w\otimes\pp$ of $\OGP^+=\bL^*\x\GP^+$ ``sees'' the local geometry of $\pp$ near the vertex maximized by $w$.  Can we replace $\bL^*$ with some Hopf monoid $\bH$ of set compositions (which correspond to \emph{all} braid cones), so that the antipode of $\cA\otimes\pp$ in the resulting Hadamard product $\bH\x\GP^+$ is local with respect to the \textit{face} of~$\pp$ maximized by $\cA$?  Aguiar and Mahajan \cite[\S\S12.4--12.5]{AgMa} describe a dual pair of Hopf monoids $\mathbf{\Sigma}$, $\mathbf{\Sigma}^*$ of set compositions and give explicit antipode formulas; there are maps $\bL\to\mathbf{\Sigma}$ and $\mathbf{\Sigma}^*\to\bL^*$ \cite[Thm.~12.57]{AgMa} arising from the inclusion of permutations into set compositions.

\item The Hopf morphism $\tilde\Upsilon:\OIGP^+\to\OMat^+$ described in Proposition~\ref{omatplus-from-hclass} is surjective (by definition), but not injective. What is its kernel?  What can be said about 0/1 extended generalized permutohedra with the same indicator complex (whether or not it is a matroid complex)?

\item Ardila and Sanchez \cite{ardila_sanchez} recently studied valuations of generalized permutohedra by passing to a quotient of $\GP^+$ by inclusion/exclusion relations. They showed that this quotient is isomorphic to a Hopf monoid of weighted ordered partitions.  One could look for an ordered analogue of their results, perhaps with a view toward a notion of valuations compatible with linear orders. 

\end{enumerate}

\section*{Acknowledgements} Part of this project was carried out while the third author was a postdoc at ICERM, the University of Miami and the Max-Planck Institute for Mathematics in the Sciences. He wishes to thank all the institutions for providing excellent research environments. We thank Marcelo Aguiar for helping us understand the language of Hopf monoids; Isabella Novik for supporting a weeklong meeting between the second and third authors in 2015, where this project started (in a very different guise); and Mark Denker, Darij Grinberg, and two anonymous referees for numerous helpful comments on the manuscript.

\raggedright
\bibliographystyle{abbrv}
\bibliography{biblio}
\end{document}